\documentclass[reqno]{amsart}

\usepackage[utf8]{inputenc}
\usepackage[english]{babel}
\usepackage{csquotes}		

\usepackage{geometry}
\usepackage{enumerate}
\usepackage{bbm}
\usepackage{xcolor,graphicx}
\usepackage{comment}
\usepackage{tikz}
    \usetikzlibrary{hobby}
    \usetikzlibrary{arrows}
    \usetikzlibrary{external}
        \pgfkeys{/pgf/images/include external/.code=\includegraphics{#1}}
\tikzexternaldisable

\usepackage[hypertexnames=false]{hyperref}		
\usepackage{cleveref}		

\newcommand{\RR}{{\mathbb R}}
\newcommand{\NN}{{\mathbb N}}
\newcommand{\ZZ}{{\mathbb Z}}

\newcommand{\Hc}{\mathcal{H}}
\newcommand{\Lc}{\mathcal{L}}

\newcommand{\Pc}{\mathcal{P}}
\newcommand{\Cc}{\mathcal{C}}
\newcommand{\Qc}{{\mathcal Q}}

\newcommand{\Fc}{{\mathcal F}}
\newcommand{\Bc}{\mathcal{B}}
\newcommand{\Gc}{\mathcal{G}}

\newcommand{\Yy}{y}

\newcommand{\weakly}{\rightharpoonup}
\newcommand{\weakstar}{\stackrel{\ast}{\rightharpoonup}}

\DeclareMathOperator{\supp}{supp}
\DeclareMathOperator{\dive}{div}
\DeclareMathOperator{\dist}{dist}
\DeclareMathOperator{\Reg}{Reg}
\DeclareMathOperator{\Sing}{Sing}

\renewcommand{\ge }{\geq}
\renewcommand{\le }{\leq}

\newcommand{\de}{\partial}
\DeclareMathOperator{\interior}{Int}

\newcommand{\eps}{\varepsilon}		

\renewcommand{\phi}{\varphi}

\renewcommand{\tilde }{\widetilde}

\DeclareMathOperator{\osc}{osc} 	

\newcommand{\dif}{\mathop{}\!\mathrm{d}}	

\newcommand{\Per}{\mathrm{Per}}

\newcommand{\inset}[1]{\left\{ #1 \right\}}

\newcommand*{\sphere}[1][n-1]{\mathbb{S}^{{#1}}}

\newcommand{\dista}[1]{{w_{{#1}}}}

\usepackage{mathtools,
	amsthm,
	amssymb		
}
    	\mathtoolsset{centercolon} 
\usepackage{autonum}

\numberwithin{equation}{section}						
\renewenvironment{proof}[1][\proofname]{{\par\medskip\noindent\bfseries #1. }}{\qed\par}		

\crefname{subsection}{Subsection}{Subsections}		

\theoremstyle{plain}
\newtheorem{theorem}{Theorem}[section]
\newtheorem{proposition}[theorem]{Proposition}
\newtheorem{lemma}[theorem]{Lemma}
\newtheorem{corollary}[theorem]{Corollary}
\newtheorem{definition}[theorem]{Definition}
\newtheorem{assumption}[theorem]{Assumption}
\theoremstyle{definition}
\newtheorem{remark}[theorem]{Remark}

\begin{document}

\title[Regularity of surfaces with degenerate weights]{Regularity for free boundary surfaces minimizing degenerate area functionals}

\author[C.~Gasparetto]{Carlo Gasparetto}

\author[F.~Paiano]{Filippo Paiano}

\author[B.~Velichkov]{Bozhidar Velichkov}

\address {Carlo Gasparetto \newline \indent
Dipartimento di Matematica, Universit\`a di Pisa \newline \indent
Largo Bruno Pontecorvo, 5, 56127 Pisa, Italy}
\email{carlo.gasparetto@dm.unipi.it}

\address {Filippo Paiano \newline \indent
Dipartimento di Matematica, Universit\`a di Pisa \newline \indent
Largo Bruno Pontecorvo, 5, 56127 Pisa, Italy}
\email{filippo.paiano@phd.unipi.it}

\address {Bozhidar Velichkov \newline \indent
Dipartimento di Matematica, Universit\`a di Pisa \newline \indent
Largo Bruno Pontecorvo, 5, 56127 Pisa, Italy}
\email{bozhidar.velichkov@unipi.it}


\subjclass[2020] {49Q05, 49Q10, 35R35}
\keywords{Minimal surfaces, free boundaries, degenerate weights, $\varepsilon$-regularity, almost-minimality, Harnack inequality.}

\begin{abstract}
    We establish an epsilon-regularity theorem at points in the free boundary of almost-minimizers of the energy $\mathrm{Per}_{w}(E)=\int_{\partial^*E}w\,\mathrm{d}\mathcal{H}^{n-1}$, where $w$ is a weight asymptotic to $d(\cdot,\mathbb{R}\setminus\Omega)^a$ near $\partial\Omega$ and $a>0$.
    This implies that the boundaries of almost-minimizers are $C^{1,\gamma_0}$-surfaces that touch $\de \Omega$ orthogonally, up to a Singular Set $\Sing(\de E)$ whose Hausdorff dimension satisfies the bound
        $d_\Hc(\Sing(\de E)) \le n +a -(5+\sqrt{8})$.
\end{abstract}

\maketitle

\section{Introduction}\label{sec:intro}

Let $\Omega\subset\RR^n$ be an open set, and let $\dista{\Omega}(x)\coloneqq \dist(x;\RR^n\setminus\Omega)^a$, where $a>0$ is fixed hereafter.
For $E\subset\Omega$ with finite perimeter and $A\subset\RR^n$ a Borel set, we define the weighted perimeter (or $\dista{\Omega}$-perimeter for short) of $E$ in $A$ as
\begin{equation}
\Per_{\dista{\Omega}}(E;A) := \int_{\de^* E\cap A} \dista{\Omega}(x) \dif\Hc^{n-1}(x),
\end{equation}
where $\de^*E$ denotes the reduced boundary of $E$.
We say that $E$ is a local minimizer if it minimizes $\Per_{\dista{\Omega}}$ among all its compact perturbations. Similarly, we call $E$ an almost-minimizer if it minimizes $\Per_{\dista{\Omega}}$ up to a deficit that depends on the size of the perturbation.

The purpose of this work is to study the boundary regularity of almost-minimizers of $\Per_\dista{\Omega}$.
In particular, the main result is an $\eps$-regularity theorem, which we roughly state as follows:

\vspace{0.2cm}
{\it Let $E\subset\Omega$ be an almost-minimizer, $x\in\de E \cap \de\Omega$, and assume that $\de\Omega$ is sufficiently regular.
If the minimizing deficit of $E$ is small and $\de E\cap B_1(x)$ is contained in a small neighborhood of a plane $\Pi(x)$ orthogonal to $\de\Omega$, then $\de E\cap B_{1/2}(x)$ coincides with the graph of a $C^{1,\gamma}$ function that satisfies suitable a-priori estimates.
}
\vspace{0.2cm}

In the literature, perimeter with similar weights arise naturally in the case $\Omega = \{x_n > 0\}$, see \cite{Dierkes2003} for further historical notes and motivations on this problem.
For $a=1$, these weights model heavy surfaces (used in architecture).
For $a=m\in\NN$, minimizers of $\Per_{\dista{}}$ correspond to rotationally invariant perimeter minimizers in $\RR^{n+a}$.
Indeed, the set $\tilde{E}= \{(x',x'') \in \RR^{n-1}\times\RR^{a+1} : (x',|x''|)\in E\}$ is an (almost)-minimizer for the classical perimeter if and only if $E$ is an (almost)-minimizer of $\Per_{\dista{\Omega}}$.
This result holds because $|x''|^a$ is the Jacobian of $a$ rotations around $\{x''=0\}$.
While this correspondence fails for non-integer $a\in\RR$, various results concerning weighted perimeters, as the weighted isoperimetric inequality \cite{CabreRosOtonSerra2016} or the boundary monotonicity formula for minimizers (see \Cref{prop:monotonicity}), suggest that we are dealing with objects where the \enquote{relevant dimension} is $(n+a)$. 

More recently, in a paper by the third author \cite{Bianco_Manna_Velichkov_2022}, weights for which $\dista{}(x) \sim d_\Omega(x)^2$ as $x\to\de\Omega$ were introduced to study a free boundary problem in dimension $n=2$.

A similar problem to ours arises in the case where $a<0$ and $\Omega = \{x_n>0\}$.
In particular, in the case $a=-1$, the weight corresponds to the one induced by the hyperbolic metric in the half-space, thus giving a connection with the asymptotic behavior at infinity of area-minimizing surfaces in hyperbolic spaces.
This problem have been investigated in a series of works \cite{Hardt_Lin_1987,Lin_1989CPAM,Lin_1989,Tonegawa_1996}.
Unlike our case, in these works $ \bar E\cap \de \Omega$ is fixed, and the authors investigate the asymptotic behavior as $E$ approaches $\de \Omega$.

In the case $a>0$, to the best of our knowledge, previous studies on minimizers of $\Per_{\dista{}}$ have mostly addressed the Bernstein problem for this setting; see, for instance, \cite{Dierkes1988,Dierkes_1990,Dierkes1995}.
However, no $\eps$-regularity result currently exists in the literature. 
In our approach, we rely on some structural properties of almost-minimizers, such as density estimates, to establish regularity. These properties are ensured by recent progress in the study of the relative isoperimetric problem; see \cite{CabreRosOtonSerra2016,Cinti_Glaudo_Pratelli_Ros-Oton_Serra_2022}.

Finally, we mention that related results concerning regularity of solutions of partial differential equations with degenerate weights have been studied in \cite{Audrito_Fioravanti_Vita_2024, Restrepo_Ros-Oton_2024, Sire_Terracini_Vita_2021, Terracini_Tortone_Vita_2024}. In particular, the results in \cite{Sire_Terracini_Vita_2021} will be used in the proof of our main result.
We also refer the readers to \cite{Bevilacqua_Stuvard_Velichkov_2024,  De_Masi_Edelen_Gasparetto_Li_2024, Ferreri_Tortone_Velichkov_2023} for $\eps$-regularity theorems which employ techniques similar to ours.

\vspace{0.2cm}

We now introduce the setting of the problem.

\subsection{Setting of the problem and assumptions on \texorpdfstring{$\de\Omega$}{the boundary}}
To state the $\eps$-regularity theorem, we need to provide a quantitative definition of almost-minimality and specify the regularity assumptions on $\de\Omega$.

\begin{definition}[Almost-Minimizer of $\Per_{\dista{\Omega}}$]
Let $\vartheta,\beta>0$ and $\Omega\subset\RR^n$ be an open set. 
We say that a set $E\subset\overline{\Omega}$ is a $(\vartheta,\beta)$-minimizer of $\Per_{\dista{\Omega}}$ in an open set $D\subset\RR^n$ if, for all $B_r(x_0)\Subset D$ and every $F\subset\overline{\Omega}$ such that $E\Delta F \Subset B_r(x_0)$, it holds
\begin{equation}
\Per_{\dista{\Omega}}(E;B_r(x_0)) \le (1+\vartheta r^\beta) \Per_{\dista{\Omega}}(F;B_r(x_0)).
\end{equation}
\end{definition}

We refer to \Cref{sec:propalmostminimizers} for a rigorous definition of $\Per_\dista{}$.
We now turn our attention to the assumptions on $\de\Omega$.
To prove that almost-minimizers are $C^{1,\gamma}$, the natural assumption is that $\de\Omega$ is a $C^{1,\alpha}$ surface for some $\alpha\in(0,1)$.

More precisely, in this paper, we make the following assumption.

\begin{definition}\label{ass:boundaryof_Omega}
	Given $\alpha\in(0,1)$ and $R>0$, we say that \textit{$\Omega$ is $\varkappa$-flat} in $B_R$ if $\Omega\subset\RR^n$ is an open set, $0\in\de\Omega$ and there exists a function $g\in C^{1,\alpha}(\RR^{n-1})$ such that
	\begin{equation}
		\Omega \cap B_R = \left\{ x = (x',x_n) \in B_R : x_n > g(x') \right\},
	\end{equation}
	where $x'$ is the projection of $x$ onto $\RR^{n-1}$ and $g$ satisfies the following conditions
	\begin{enumerate}
		\item $g(0) = |\nabla g(0)| =0$;
		\item The H\"older modulus of $\nabla g$ is bounded by $\varkappa R^{-\alpha}$, i.e.,
			\begin{equation}
				[g]_{C^{1,\alpha}(B_R)} = \sup\left\{ \frac{|\nabla g(x') - \nabla g(y')|}{|x'-y'|^\alpha} :\; x',y' \in B_R' ,\; x\ne y \right\} \le \frac{\varkappa}{R^\alpha}.
			\end{equation} 
	\end{enumerate}
\end{definition}

We will simply refer to $\Omega$ as $\varkappa$-flat whenever $B_R = B_1$.

\subsection{Main results}
The main result of this paper is an $\eps$-regularity theorem at the boundary. We prove that if an almost-minimizer is sufficiently flat along some direction, then its boundary is the graph of a $C^{1,\gamma}$-function.

\begin{theorem}[$\eps$-Regularity]\label{thm:main_theorem}
	There exist constants $\eps_0, \lambda_0 > 0$ (small), $C_0>0$ (large) and $\gamma_0\in(0,1)$ depending only on $n, a, \alpha$, and $\beta$ such that the following holds.
	Let $\Omega$ be $\varkappa$-flat in the sense of \Cref{ass:boundaryof_Omega}, and let $E$ be a $(\vartheta,\beta)$-minimizer of $\Per_\dista{\Omega}$.
	Furthermore, assume that
	\begin{equation}
		\de E \cap \Omega\cap B_1 \subset \left\{ x \in B_1 : |x\cdot \nu| \le \eps \right\},
	\end{equation}
	for some $\nu\in\sphere$ with $\nu\perp e_n$ and $\eps > 0$, and that
	\begin{equation}
		(\varkappa+\vartheta)^{\lambda_0} \le \eps \le \eps_0.
	\end{equation}
	Then, there exists a function $u \in C^{1,\gamma_0}(\RR^{n-1})$ such that
	\begin{equation}
		\de E \cap\Omega\cap B_{1/2}= \left\{x\in\Omega\cap B_{1/2} : x = x'' + u(x'')\nu \mbox{ and }x''\in\nu^\perp\right\},
	\end{equation}
	and
	\begin{equation}
		\| u\|_{C^{1,\gamma_0}(B''_{1/2})} \le C_0 \eps.
	\end{equation}
\end{theorem}

We refer the reader to \Cref{subsec:sketch} for an outline of the main ideas of the proof.
In the remainder of this subsection, we briefly discuss two consequences of \Cref{thm:main_theorem}: a generalization to a broader class of weights and its connection to the Bernstein problem.

\subsection*{Hölder continuous weights.}
For any nonnegative function $\dista{} : \RR^n \to [0,+\infty)$, we define the weighted perimeter of a set of finite perimeter $E$ (within a Borel set $A$) as
\begin{equation}
\Per_{\dista{}}(E;A) = \int_{\de^*E\cap A} \dista{}(x) \dif \Hc^{n-1}(x),
\end{equation}
and all previous definitions of minimizing properties remain valid.
We make the following assumption on $\dista{}$.

\begin{assumption}\label{ass:weights}
	We assume that there exist $C>0$ and $b\in(0,1)$ such that
	\begin{equation}
		\dista{}(x) \mbox{ and } \frac{\dista{}(x)}{d_\Omega(x)^a}  \, \in C^{0,b}(\overline{\Omega})
	\end{equation}
		and
	\begin{equation}
		\frac{1}{C} \le \frac{\dista{}(x)}{d_\Omega(x)^a} \le C,\quad \mbox{in }\overline{\Omega}. 
	\end{equation}
\end{assumption}  

\begin{proposition}\label{prop:changing_weight}
	Let $\dista{}\in C^{0,b}(\RR^n)$ satisfy \Cref{ass:weights}. 
	If $E \subset \overline{\Omega}$ is a $(\vartheta,\beta)$-minimizer for $\Per_{\dista{}}$, then there exist constants $\vartheta' > 0$ and $\beta' > 0$ such that $E$ is also a $(\vartheta',\beta')$-minimizer for $\Per_{\dista{\Omega}}$.
\end{proposition}

\begin{proof}
Let $F\subset\RR^n$, $x_0\in\RR^n$ and $r>0$ such that  
 $E\Delta F \Subset B_r(x_0)\subset B_1$, then 
\begin{equation}
\begin{split}
    \Per_{\dista{\Omega}}(E;B_r(x_0)) 
        &\le \sup_{y\in B_r(x_0)}\left( \frac{d_\Omega(y)^a}{\dista{}(y)} \right) \Per_{\dista{}}(E;B_r(x_0))\\
        &\le \sup_{y\in B_r(x_0)}\left( \frac{d_\Omega(y)^a}{\dista{}(y)} \right)(1+\vartheta r^\beta) \Per_{\dista{}}(F;B_r(x_0))\\
        &\le \sup_{y\in B_r(x_0)}\left( \frac{d_\Omega(y)^a}{\dista{}(y)} \right)\sup_{y\in B_r(x_0)}\left( \frac{\dista{}(y)}{d_\Omega(y)^a} \right)(1+\vartheta r^\beta) \Per_{\dista{\Omega}}(F;B_r(x_0)).
\end{split}
\end{equation}
Using the Hölder continuity of $\dista{}$ and \Cref{ass:weights}, we estimate:  \begin{equation}
    \sup_{y\in B_r(x_0)}\left( \frac{d_\Omega(y)^a}{\dista{}(y)} \right)\sup_{y\in B_r(x_0)}\left( \frac{\dista{}(y)}{d_\Omega(y)^a} \right) 
    \le \left( \frac{d_\Omega(x_0)^a}{\dista{}(x_0)} + Cr^b \right)\left( \frac{\dista{}(x_0)}{d_\Omega(x_0)^a} + Cr^b \right) 
    \le 1 + C r^b,
\end{equation}
and therefore, letting $\beta'=\min\{b,\beta\}$ and $\vartheta' = \vartheta + C$, we get
\begin{equation}
    \Per_{\dista{\Omega}}(E;B_r(x_0)) \le (1+ \vartheta'r^{\beta'}) \Per_{\dista{\Omega}}(F;B_r(x_0)).
\end{equation}

\end{proof}

By \Cref{prop:changing_weight} above, the following result is a straightforward consequence of \Cref{thm:main_theorem}.

\begin{corollary}\label{cor:main_theorem}
	There exist constants $\eps_0',\lambda_0'>0$ (small), $C_0'>0$ (large) and $\gamma_0'\in(0,1)$ depending only on $n,a,\alpha,\beta,C$ and $b$ with the following property.
    Let $\Omega$ be $\varkappa$-flat in the sense of \Cref{ass:boundaryof_Omega}, let $w$ satisfy \Cref{ass:weights} and let $E$ be a $(\vartheta,\beta)$-minimizer of $\Per_w$.
    Furthermore, assume that
    \begin{equation}
				\de E \cap \Omega \cap B_1 \subset \left\{ x \in B_1 : |x\cdot \nu |\le \eps \right\};
	\end{equation} 
    for some $\nu\in\sphere$ with $\nu\perp e_n$ and $\eps>0$ and that
    \begin{equation}
        (\varkappa+\vartheta)^{\lambda_0'}\le\eps\le\eps_0'.
    \end{equation}
    Then there exists a function $u\in C^{1,\gamma_0'}(\RR^{n-1})$ such that
	\begin{equation}
		\de E \cap\Omega \cap B_{1/2} = \left\{x\in\Omega\cap B_{1/2} : x = x'' + u(x'')\nu \mbox{ and }x''\in\nu^\perp\right\},
	\end{equation}
	and 
	\begin{equation}
		\| u\|_{C^{1,\gamma_0'}(B_{1/2}'')} \le C_0' \eps.
	\end{equation}
\end{corollary}

\subsection*{Singular Set}
In order to state the following results regarding almost-minimizers with singular boundaries, we introduce the following terminology: we say that $E\subset\{x_n\ge0\}\subset\RR^{n}$ is a \textit{regular cone at $0$} if there exists $\nu\in\sphere$, $\nu\cdot e_n=0$, such that
\begin{equation}
    \overline{E} = \{x\in\RR^n\colon x_n\ge0\mbox{ and }x\cdot\nu\le0\}.
\end{equation}
In \cite{Dierkes_1990}, it was proved that if $E\subset\RR^n$ is a cone at $0$ (meaning $\frac{E}{r}=E$ for all $r$), it is a minimizer of $\Per_{w}$ for $w=(x_n)_+^a$ and $n<5+\sqrt{8}-a$, then $E$ is regular at $0$. 

Given $a>0$, let $n^*_a$ denote the smallest $n\in\NN$ such that there exists a cone $E$ that is a minimizer of $\Per_{(x_n)_+^a}$ in $\{x_n\ge0\}\subset \RR^n$ that is not regular.
Notice that, by the aforementioned result, it holds $n^*_a\ge 5+\sqrt{8}-a$.

Expanding on the above discussion, it is natural to state the following



\begin{corollary}
    Let $\gamma_0\in(0,1)$ be as in \Cref{thm:main_theorem}, let $\Omega$ be $\varkappa$-flat in the sense of \Cref{ass:boundaryof_Omega}, and let $E$ be a $(\vartheta,\beta)$-minimizer of $\Per_{\dista{\Omega}}$.
    Then there exist two disjoint sets $\Sing(\de E)$ and $\Reg(\de E)$ such that
    \begin{equation}
        \overline{\de E \cap {\Omega}} = \Sing(\de E) \cup \Reg(\de E),
    \end{equation}
    with the following properties:
    \begin{enumerate}
        \item  $\Reg (\de E)$ is relatively open in $\overline{\de E\cap \Omega}$ and $\Sing(\de E)$ is relatively closed in $\overline{\de E\cap \Omega}$.
        
        \item The set $\Reg(\de E)$ is a $(n-1)$-dimensional $C^{1,\gamma_0}$-manifold with boundary given by 
        \begin{equation}
            \de \Reg(\de E) = \Reg(\de E) \cap \de \Omega.
        \end{equation}
        
        \item $\Hc^{s}(\Sing(\de E))=0$ for all $s > n-n_a^*$.
    \end{enumerate}
    In particular, if $n< n_a^*$, then $\de E$ is a $(n-1)$-manifold with boundary $\de(\de E) = \de E \cap \de \Omega$.
\end{corollary}
The proof of the corollary relies on Federer’s dimensional reduction principle, which is nowadays considered classical, thus we omit it.

\subsection{Strategy of the proof}\label{subsec:sketch}

\Cref{thm:main_theorem} follows from an \textit{improvement of flatness}-type result, which we state next.
The below result, combined with an analogous one for points away from $\de\Omega$ (\Cref{prop:interior_IOF}) gives a $C^{1,\gamma}$-decay of oscillations up to $\de\Omega$ (\Cref{cor:iof_iter}); we prove all these results in \Cref{sec:IOF}.
Then, with \Cref{cor:iof_iter} at hand, the proof of \Cref{thm:main_theorem} is classical (see, for instance, \cite{Velichkov_2023}) and we omit it.

\begin{theorem}[Improvement of flatness]\label{thm:improvement_flatness}
    There exist universal constants $\eps_{1},\lambda_{1}, \eta_{1}$ (small) and $C_1$ (large) with the following property.
    Let $\Omega$ be $\varkappa$-flat and let $E$ be a $(\vartheta,\beta)$-minimizer of $\Per_{\dista{\Omega}}$ in $B_1$. Furthermore, assume that
    \begin{equation}
		\de E \cap \Omega\cap B_1 \subset \left\{ x \in B_1 : |x\cdot \nu| \le \eps \right\},
	\end{equation}
	for some $\nu\in\sphere$ with $\nu\perp e_n$ and $\eps > 0$, and that
	\begin{equation}
		(\varkappa+\vartheta)^{\lambda_{1}} \le \eps \le \eps_{1}.
	\end{equation}
    Then there exists $\tilde\nu\in\sphere$ such that $\tilde\nu\cdot e_n=0$, $|\tilde\nu-\nu|\le C_1\eps$ and
    \begin{equation}\label{eq:IOF_thesis_0}
        \de E\cap \Omega\cap B_{\eta_{1}}\subset\left\{x\colon |x\cdot\tilde\nu|\le\frac{1}{2}\eps \eta_{1}\right\}.
    \end{equation}
\end{theorem}

The improvement of flatness asserts that if the boundary of an almost-minimizer $E$ is sufficiently close to a plane orthogonal to $\de\Omega$ at $0$, then there exists another plane (still orthogonal to $\de\Omega$ at $0$) to which the boundary of $E$ is closer, even after a rescaling.

We provide an outline of the proof, emphasizing key ideas and challenges. The approach follows the framework developed in \cite{Savin_2007, DeSilva_2011, DeSilva_Savin_2021}.

We begin by showing that if $E\subset\overline{\Omega}$ is a $(\vartheta,\beta)$-minimizer, then $E$ is close (in a sense we will specify later) to being a viscosity solution of the differential problem:
\begin{equation}\label{eq:viscosity_system01}
	\begin{cases}
		H_{E}(x) + a\frac{\nu_E(x)\cdot \nabla d_\Omega(x)}{d_{\Omega}(x)}=0, & \mbox{in } \Omega\\
		\nu_E(x) \cdot \nabla d_{\Omega}(x) =0 & \mbox{on } \de \Omega,
	\end{cases}
\end{equation}
where $H_E$ denotes the generalized mean curvature of $\de E$ and $\nu_E$ is its outer unit normal vector.

Exploiting the above property, we then prove that the oscillations of $\de E$ satisfy a quasi-Harnack inequality, which entails that the oscillations of $\de E$ decay up to a scale comparable to the flatness parameter. Specifically, we will prove a Harnack inequality valid away from $\de\Omega$ using techniques based on \cite{DeSilva_Savin_2021} and one near $\de\Omega$ using \cite{DeSilva_2011}.
The most technical issue in the procedure above will be the \textit{weak maximum principle} describing in which sense an almost-minimizer is a viscosity solution to \eqref{eq:viscosity_system01}: we will spend a few more words about it in the next subsection. 

Once we have obtained the aforementioned control of oscillations, we argue by contradiction.
Assume that the statement of \Cref{thm:improvement_flatness} fails; then there exist sequences $\{\Omega_j\}_{j\in\NN}$ of $\varkappa_j$-flat sets and $\{ E_j\}_{j\in\NN}$ of $(\vartheta_j,\beta_j)$-minimizers of $\Per_{\dista{\Omega_j}}$ such that $0\in\de E_j$,
\begin{equation}
	\de E_j \cap B_1 \subset \left\{x\in B_1 : |x\cdot e_1|\le \eps_j\right\},
\end{equation}
with $\eps_j\to0$ and $\vartheta_j+\varkappa_j \ll\eps_j$, but for which there is no direction $\tilde{\nu}_j$ for which \eqref{eq:IOF_thesis_0} holds.
Then we consider the sets
\begin{equation}
	\tilde E_j = \left\{ \left(\frac{x_1}{\eps_j},x''\right) \in B_1 : (x_1,x'')\in \de E_j  \right\} \subset \{|x_1| \le 1\}.
\end{equation}

The oscillation control provided by the aforementioned quasi-Harnack inequality guarantees a uniform $C^{0,\sigma}$-type control on $\de\tilde{E}_j$.
Consequently, $\de \tilde{E}_j$ converges in the Hausdorff distance to a closed set which can be shown to be the graph of a H\"{o}lder function $u$ over $\{x_1=0\}$.
By the stability of viscosity properties under uniform convergence, $u$ is the solution of the linearized problem of \eqref{eq:viscosity_system01}, i.e., it is the solution of the following Neumann problem:
\begin{equation}
	\begin{cases}
		\dive(x_n^a \nabla_{x''} u) = 0, & \mbox{in }\{x_n>0\},\\
		\de_n u = 0 & \mbox{on } \{x_n = 0\},
	\end{cases}
\end{equation}
where $x=(x_1,x'')$.

Finally, the desired conclusion \eqref{eq:IOF_thesis_0} follows by the regularity results proved in \cite{Sire_Terracini_Vita_2021}.

\subsection{Maximum Principles}
We now elaborate on how an almost-minimizer $E$ satisfies \eqref{eq:viscosity_system01} in an appropriate sense.

The boundary condition $\nu_E\cdot\nabla d_\Omega =0$ should be understood in a classical viscosity sense.
Specifically, if $E$ is a $(\vartheta,\beta)$-minimizer, and a closed and smooth set $F$ touches $E$ at $0\in\de \Omega$ from outside (that is $0\in\de E\cap \de F$ and $F\supset E\cap B_r\cap \overline{\Omega}$ for some $r>0$), it can be shown that $\nu_F(0)\cdot\nabla d_\Omega(0)\le 0$. 

For the condition in the interior of $\Omega$, the situation is less straightforward. Indeed, as we will see later, any  $C^{1,\gamma}$-regular set is an almost minimizer of $\Per_\dista{\Omega}$ with respect to compact perturbations that are sufficiently small and localized far enough from the boundary.
Therefore, we cannot expect an almost-minimizer to satisfy a partial differential equation in any classical viscosity sense, at least not pointwise.
In \cite{DeSilva_Savin_2021}, the authors introduced a weaker viscosity-type condition that holds true for almost minimizers and they showed how to use it to obtain regularity.
In this paper, we adapt their techniques to our case. In the following lines we will give an overview of what is the appropriate viscosity condition.

To introduce the viscosity condition for almost-minimizers, we first discuss the viscosity properties of a \emph{minimizer} $E$.
Let $E$ be a minimizer and $F$ be a smooth set that touches $E$ from the outside at $x_0$ in $B_r(x_0)$, for some $r>0$.
Following the ideas introduced in \cite{Caffarelli_Cordoba_1993}, if
\begin{equation}\label{eq:wrong_curvature}
	H_{\de F}(x_0) + a\frac{\nu_{\de F}(x_0) \cdot \nabla d_\Omega(x_0)}{d_\Omega(x)} > 0, 
\end{equation}
then we may \enquote{push} $F$ a bit further through $\de E$, defining a competitor set $G\subset \Omega$ with $G\Delta E \Subset B_r(x_0)$ and
\begin{equation}\label{eq:example_non_minimality}
	\Per_{\dista{\Omega}}(G;B_r(x_0)) < \Per_{\dista{\Omega}}(E;B_r(x_0))
\end{equation}
contradicting the minimality of $E$.

In the case of an \textit{almost}-minimizer, this last condition is not a contradiction, and to overcome this issue we need to define a competitor set $G$ that satisfies \eqref{eq:example_non_minimality} with a quantitative estimate on the gap.
The key idea is to connect the size of the oscillations of $\de F$ (up to the second order) with the radius $r$ of the ball $B_r(x_0)$ where $F$ acts as a barrier and touches $E$ from outside.
More precisely, let us assume that $\Omega$ is $\varkappa$-flat, $E$ is a $(\vartheta,\beta)$-minimizer and $F = \{ x_1 \le \phi(x'')\}$, with $\|\phi\|_{C^{2}} \lessapprox c$ and the {\it wrong curvature}.
The idea is to show that for sufficiently small $\vartheta$ and $\varkappa$, if $F$ touches $E$ from outside at $x_0$ in a ball $B_c(x_0)$ (the same $c>0$ as above), then we have enough space to define a competitor $G$ that contradict the almost-minimality.

These arguments are initially developed in \Cref{lemma:technical_geom}, which serves as a technical basis for the proofs of \Cref{prop:weak_viscosity}, \Cref{prop:harnack_boundary} and \Cref{thm:improvement_flatness}. 
There we introduce appropriate maximum principles, to deal both with points that are far enough from $\de\Omega$ (where $H_{\de F}$ dominates \eqref{eq:wrong_curvature}), and with points close to $\de\Omega$ (where $\frac{\nu_{\de F}\cdot\nabla d_\Omega}{d_\Omega}$ is the leading term).

We now spend a few more words to discuss why we get different viscosity notions at the boundary and at the interior.
The boundary equation is of first-order, and more importantly it is scale invariant.
Indeed, if a smooth $F$ touches an almost-minimizer $E$ from outside at some $x_0\in\de\Omega$ then the tangent half-space $F_0=\{(x_0- x)\cdot \nu_F(x_0)\le 0\}$ touches the blow-up limit of $E$ from outside at $x_0$.
Moreover, blow-up limits of almost-minimizers are minimizers (see \Cref{prop:compactnessqmin}).
Therefore, the viscosity notion for almost-minimizers and for minimizers are equivalent and it coincides with the \enquote{classical} one: a pointwise information on the tangent space to $\de F$ at the contact point $x_0$.

\begin{figure}[h]
\centering
	\begin{tikzpicture}[use Hobby shortcut, thick,even odd rule]
		\begin{scope}
		 \clip (-2.5,-0.5) rectangle (2.5,2);
		 \clip	(0,0) circle (2);
		 	
		 \draw[fill = black!10, yscale=0.25] (-4,12) -- (-4,-0.5) .. (-3,-1) .. (-2,0.5) .. (-1,-0.5) .. (0,0) .. (1,0.5) .. (2,-0.2) .. (3,1) .. (4,0.5) -- (4,12) -- cycle;
		 
		 \begin{scope}
		 	\clip[yscale=0.25,shift={(0,-0.9)}] (-4,12) -- (-4,-0.5) .. (-3,-1) .. (-2,0.5) .. (-1,-0.5) .. (0,0) .. (1,0.5) .. (2,-0.2) .. (3,1) .. (4,0.5) -- (4,12) -- cycle;
		 	\draw[fill = red!20] (0.3,-0.6).. (0,0) .. (-0.25,1) .. (0,3)  -- (0,6) -- (6,6) -- (6,-6) -- (0,-6) -- cycle;
		 \end{scope}

		 \begin{scope}
		 	\clip[yscale=0.25] (-4,12) -- (-4,-0.5) .. (-3,-1) .. (-2,0.5) .. (-1,-0.5) .. (0,0) .. (1,0.5) .. (2,-0.2) .. (3,1) .. (4,0.5) -- (4,12) -- cycle;
		 	
		 		\draw[fill = red!20] (0,0) .. (-0.25,1) .. (0,3) -- (0,6) -- (6,6) -- (6,-6) -- cycle;

		 		\draw[fill = blue!20] (0,6) -- (3,3) .. (0.25,2) .. (0,1.5) .. (0.25,1) .. (0,0.5) .. (0,0) .. (1,-1) -- (1,-6) -- (6,-6) -- (6,6) -- cycle;
		 		
				\draw[red!70,->] (0.35,0) arc (0:107:0.35cm);
		 		\draw[yscale=0.25] (-4,12) -- (-4,-0.5) .. (-3,-1) .. (-2,0.5) .. (-1,-0.5) .. (0,0) .. (1,0.5) .. (2,-0.2) .. (3,1) .. (4,0.5) -- (4,12) -- cycle;
		 \end{scope}	
		 		
		\end{scope}
		\node at (-1,1) {$\Omega$};
		\node at (0,-0.3) {$x_0$};
		\node at (0,1) {$F$};
		\node at (1,1) {$E$};
		\node at (0.3,0.5) {$\theta$};
		
	\end{tikzpicture}
    \hspace{1cm}
    \begin{tikzpicture}[use Hobby shortcut, thick,even odd rule]
		\begin{scope}
		 \clip (-2.5,-0.5) rectangle (2.5,2);
		 \clip	(0,0) circle (2);
		 			 
		 \draw[fill = black!10, yscale=0.25] (-4,-4) -- (4,4) -- (4,12) -- (-2,12) -- cycle;
		 
		 \begin{scope}
		 	\clip[yscale=0.25,shift={(0.5,-0.5)}] (-4,-4) -- (4,4) -- (4,12) -- (-2,12) -- cycle;

		 	\draw[fill = red!20, shift={(0.375,-1)}] (0,0)  -- (-1.5,4) .. (4,4) -- (4,0);
		 \end{scope}

		 \begin{scope}
		 	\clip[yscale=0.25] (-4,-4) -- (4,4) -- (4,12) -- (-2,12) -- cycle;

		 		\draw[fill = blue!20,scale=4] (0,6) -- (3,3) .. (0.25,2) .. (0,1.5) .. (0.25,1) .. (0,0.5) .. (0,0) .. (1,-1) -- (1,-6) -- (6,-6) -- (6,6) -- cycle;
		 		
				\draw[red!70,->] (0.35,0) arc (0:107:0.35cm);
				\draw[yscale=0.25] (-4,-4) -- (4,4) -- (4,12) -- (-2,12) -- cycle;

		 \end{scope}	
		 		
		\end{scope}
		
		\node at (-1,1) {$\Omega$};
		\node at (0,-0.3) {$x_0$};
		\node at (-0.3,1.5) {$F$};
		\node at (1,1) {$E$};
		\node at (0.3,0.5) {$\theta$};

	\end{tikzpicture}
    \caption{Two sets touching at $\de\Omega$: before blow-up (left) and after (right)}
\end{figure}

On the other hand, the equation in the interior is of second order, and it vanishes under rescaling.
More precisely, let us assume $x_0\in\de E \cap \Omega$, $E$ is an almost-minimizer and $F$ is a smooth set that touches $E$ from outside at $x_0$.
At any scale $r>0$, $F_r=\frac{1}{r}(F-x_0)$ still touches $E_r=\frac{1}{r}(E-x_0)$ from outside (at the origin), but 
\begin{equation}
    H_{F_r}(0) + a \frac{\nu_{F_r}(0)\cdot \nabla d_{\Omega_r}(0)}{d_{\Omega_r}(0)} = r \left(H_F(x_0) + a \frac{\nu_{F}(x_0)\cdot \nabla d_\Omega(x_0)}{d_\Omega(x_0)}\right)\xrightarrow{r\to0} 0,
\end{equation}
thus taking the blow-up limit will lead to a complete loss of information.
Therefore, we cannot reduce ourselves to the case of minimizers as we did for the boundary case, but we need to introduce a truly weaker viscosity notion.
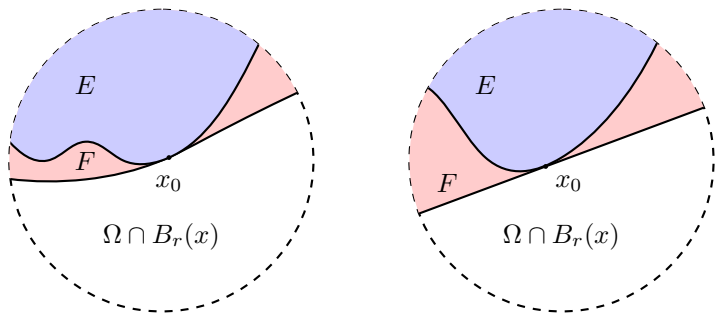
\begin{figure}[ht]
    \centering
    \begin{tikzpicture}[use Hobby shortcut, thick,even odd rule,rotate =90]
	\draw[dashed] (0,0) circle (2);
		\begin{scope}
		 \clip	(0,0) circle (2);
		 	
		 	\draw[fill = red!20] (1,-2) .. (0.3,-0.6).. (0,0) .. (-0.25,1) .. (0,3)  -- (0,6) -- (6,6) -- (6,-6) -- (0,-6) -- cycle;

		 	\draw[fill = blue!20] (0,6) -- (3,3) .. (0.25,2) .. (0,1.5) .. (0.25,1) .. (0,0.5) .. (0,0) .. (1,-1) .. (4,-2) -- cycle;
		 		
		 	\draw (0.04,-0.1) circle (0.5pt);

		\end{scope}
		\node at (-1,0) {$\Omega\cap B_r(x)$};
		\node at (-0.3,-0.1) {$x_0$};
		\node at (0,1) {$F$};
		\node at (1,1) {$E$};
		
	\end{tikzpicture}
    \hspace{1cm}
	\begin{tikzpicture}[use Hobby shortcut, thick,even odd rule,rotate=90]
		\draw[dashed] (0,0) circle (2);
		\begin{scope}
	        \clip	(0,0) circle (2);
		 	
		 	\draw[fill = red!20, shift={(0.375,-1)},scale=10] (3,-8) -- (0,0)  -- (-3,8) .. (8,8) -- (4,-8);

		 	\draw[fill = blue!20,xscale=4,yscale=2,shift={(0.01,0)}] (0,6) -- (3,3) .. (0.25,2) .. (0,1.5) .. (0.25,1) .. (0,0.5) .. (0,0) .. (1,-1) -- (1,-6) -- (6,-6) -- (6,6) -- cycle;
		 	
                \draw[shift={(-0.078,0.208)}] (0,0) circle (0.5pt);
		  
            \node at (-1,0) {$\Omega\cap B_r(x)$};
		\node at (-0.3,-0.1) {$x_0$};
		\node at (-0.3,1.5) {$F$};
		\node at (1,1) {$E$};

		\end{scope}
		
	\end{tikzpicture}
    \caption{Two sets touching at points in $\Omega$: before blow-up (left) and after (right) }

\end{figure}

\subsection{Regularity away from \texorpdfstring{$\de\Omega$}{the boundary} and optimality}\label{subsection:Tamanini}
Regularity for almost-minimizer of $\Per_{\dista{}}$ (where $w$ satisfies \Cref{ass:weights}) may be inferred from classical regularity results for almost-minimizers of the euclidean perimeter, see for instance \cite{Tamanini_1982}.

To see this, we briefly show that if $E$ is a $(\vartheta,\beta)$-minimizer of $\Per_{\dista{\Omega}}$ and $D\Subset\Omega$, then there exist $\vartheta_D,\beta_D$ such that $E$ is a $(\vartheta_D,\beta_D)$-minimizer of $\Per$ in $D$ with respect to small enough perturbations.
Namely, if $F\subset\RR^n$ is such that $E\Delta F\Subset B_r(x)\subset D$, then
\begin{multline}
	\Per(E;B_r(x_0)) \le \frac{1}{(d_{\Omega}(x_0)-r)^a}\Per_{\dista{\Omega}}(E;B_r(x_0))\\
	\le \frac{1+\vartheta r^\beta}{(d_\Omega(x_0)-r)^a}\Per_{\dista{\Omega}}(F;B_r(x_0))\\
	\le (1+\vartheta r^\beta)\left(\frac{d_\Omega(x_0)+r}{d_\Omega(x_0)-r}\right)^a\Per(F;B_r(x_0)).
\end{multline}
Assuming $r\le1$ and taking into account that $d_\Omega(x_0)-r$ can be bounded from below by some constant depending only on $D$, we may bound
\begin{equation}
    \left(\frac{d_\Omega(x_0)+r}{d_\Omega(x_0)-r}\right)^a\le 1+\bar Cr
\end{equation}
where $\bar C$ depends only on $a$ and $D$.
Summing up, we obtain
\begin{equation}
    \Per(E;B_r(x_0))\le (1+\vartheta_D r^{\beta_D})\Per(F;B_r(x_0))
\end{equation}
with $\vartheta_D\coloneqq (\vartheta+\bar C)^2$ and $\beta_D\coloneqq\min\{\beta,1\}$.

With the above computation, we can reduce ourselves to the setting of \cite{Tamanini_1982} and obtain $C^{1,\gamma}$ interior regularity for an almost-minimizer $E$.
We stress that, as one should expect, any a-priori estimate given by \cite{Tamanini_1982} must degenerate near $\de\Omega$.
Furthermore, \cite{Tamanini_1982} would only apply if the flatness $\frac{oscillations}{radius}$ overcame the non-minimality $\vartheta_D$, which actually degenerates as $D$ gets closer to $\de\Omega$. In particular, differently than in the case of (almost)-minimizers of $\Per$, even under the a-priori assumption that the flatness is the same at all scales, \cite{Tamanini_1982} will fail for points near $\de\Omega$.

The latter consideration highlights that the core difficulty in obtaining $\eps$-regularity for almost-minimizers of $\Per_w$ lies in the fact that $w$ degenerates near $\de\Omega$, rather than in the assumption of \textit{almost}-minimality.
In particular, we point the attention of the reader to the fact that the interior Harnack inequality (\Cref{prop:interior_harnack}) requires the use of the weak viscosity condition introduced in \cite{DeSilva_Savin_2021} even for minimizers of $\Per_{\dista{\Omega}}$.

Finally, we remark that $C^{1,\gamma}$ regularity is optimal, as any set with $C^{1,\gamma}$-regular boundary is an almost minimizer for $\Per_{\dista{\Omega}}$ (at least for localized enough perturbations). This is true both in the interior of $\Omega$ and at boundary points.

\subsection{Outline of the paper}

In \Cref{sec:propalmostminimizers}, we recall some notions from geometric measure theory to properly state the relevant theorems and definitions. 
We also discuss properties such as density estimates and the compactness of almost-minimizers, which are crucial for the following sections.

In \Cref{sec:var_visc_properties}, we analyze the variational properties of stationary points, introducing a monotonicity formula at boundary points that enables us to establish the boundary viscosity property. 
Additionally, we present a technical geometric lemma that plays a key role in proving the interior maximum principles.

In \Cref{sec:eps_regularity}, we discuss the main results of the paper. The first two subsections focus on developing the interior and boundary Harnack inequalities, while the third subsection is dedicated to the proof of \Cref{thm:improvement_flatness}.

\subsection*{Acknowledgments}
The authors are supported by the European Research Council (\textsc{ERC}), under the European Union's Horizon 2020 research and innovation program, through the project \textsc{ERC VAREG} - Variational approach to the regularity of the free boundaries (grant agreement No. 853404). The authors also acknowledge the MIUR Excellence Department Project
awarded to the Department of Mathematics, University of Pisa, CUP I57G22000700001.
C.G. acknowledges the support of Gruppo Nazionale per l’Analisi Matematica, la Probabilità e le loro Applicazioni (GNAMPA) of Istituto
Nazionale di Alta Matematica (INdAM) through the INdAM-GNAMPA project 2024 CUP E53C23001670001 and the INdAM-GNAMPA project 2025 CUP E5324001950001.
F.P. is partially supported by Gruppo Nazionale per l’Analisi Matematica, la Probabilità e le loro Applicazioni (GNAMPA) of Istituto
Nazionale di Alta Matematica (INdAM).
B.V. acknowledges also support from the project MUR-PRIN ``NO3'' (No. 2022R537CS).

\section{Preliminaries and basic properties of almost-minimizers of \texorpdfstring{$\Per_w$}{Per w}}\label{sec:propalmostminimizers}

We fix hereafter $n\in\NN$ with $n\ge2$, $a>0$ and $\alpha,\beta\in(0,1)$. For convenience, the dependence of any constants on $n,a,\alpha$ and $\beta$ will not be stated, and constants depending only on those three parameters will be called \textit{universal}.

\subsection{Regular domains and distance function}\label{subsec:distance}
As previously stated, given an open domain $\Omega\subset\RR^n$, we let $d_\Omega(x)\coloneqq \inf\{|x-y|\colon y\in\RR^n\setminus\Omega\}$, so that $d_\Omega>0$ in $\Omega$ and $d_\Omega\equiv0$ in $\RR^n\setminus\Omega$.

Provided $\de\Omega$ is regular enough, we also let $\nu_\Omega(y)$ denote the outer unit normal to $\de\Omega$ at $y$, so that
\begin{equation}
    \nu_\Omega(y) = -\lim_{\substack{x\in\Omega\\x\to y}}\nabla d_\Omega(x).
\end{equation}
We gather some straightforward remarks on properties of $d_\Omega$.

\begin{lemma}[Technical lemma on the distance function]\label{lemma:tech_distance}
    $d_\Omega$ is differentiable at $x\in\Omega$ if and only if there exists a unique $y\in\de\Omega$ such that $d_\Omega(x)=|x-y|$, and in that case $\nabla d_\Omega(x)=\frac{x-y}{|x-y|}$.
    Furthermore, if $d_\Omega$ is differentiable at $x$, $y$ is as above and $\de\Omega$ has a tangent plane at $y$, then $\nabla d_\Omega(x)=-\nu_\Omega(y)$.
\end{lemma}
Notice that, since $d_\Omega$ is $1$-Lipschitz, it is differentiable at $\Lc^n$-almost every point in $\RR^n$.

As previously stated, throughout the paper we will work with domains $\Omega$ that are locally $C^{1,\alpha}$.
We refer the reader to \Cref{ass:boundaryof_Omega} for the definition of $\varkappa$-flat domain.
In passing, we notice that if $\Omega$ is $\varkappa$-flat, then for every $R>0$, $\frac{1}{R}\Omega$ is $\varkappa R^\alpha$-flat.

\begin{lemma}\label{lemma:notes_on_distance}
    Let $\Omega$ be $\varkappa$-flat.
    Then:
    \begin{enumerate}
        \item\label{item:varkappa_normals} For every $y\in\de\Omega\cap B_1$,
        \begin{equation}
            |\nu_\Omega(y)+e_n|\le \varkappa|y'|^\alpha,\qquad|y\cdot \nu_\Omega(y)|\le 2\varkappa |y|^{1+\alpha}.
        \end{equation}
        \item\label{item:varkappa_gradients} For $\Lc^n$-almost every $x\in\Omega\cap B_{1/2}$,
        \begin{equation}
            |d_\Omega(x)-x_n|\le 3\varkappa,\qquad |\nabla d_\Omega(x)-e_n|\le\varkappa
        \end{equation}
        and the first inequality above holds true everywhere in $B_{1/2}\cap\Omega$.
    \end{enumerate}
\end{lemma}
\begin{proof}
    \Cref{item:varkappa_normals} is a straightforward consequence of $\nu_\Omega(y) = \frac{(\nabla'g(y'),-1)}{\sqrt{1+|\nabla'g(y')|^2}}$.
    For \Cref{item:varkappa_gradients}, if $x\in\Omega\cap B_{1/2}$ is a point of differentiability for $d_\Omega$, then for some $y\in\de\Omega\cap B_1$ it holds $\nabla d_\Omega(x)=-\nu_\Omega(y)$, hence $|\nabla d_\Omega(x)-e_n|\le\varkappa$ and
    \begin{equation}
        |d_\Omega(x)-x_n|\le|(x-y)\cdot\nu_\Omega(y)+e_n\cdot x|\le 3\varkappa
    \end{equation}
    by \Cref{item:varkappa_normals}. The latter inequality holds true by continuity everywhere in $\Omega\cap B_{1/2}$.
    
\end{proof}

\subsection{Weighted perimeters} 
We start by recalling the relevant definitions and notation, most of which were already introduced in \Cref{sec:intro}.

Throughout the section, we assume that $\Omega\subset\RR^n$ is an open set with Lipschitz boundary and that $w$ satisfies \Cref{ass:weights} with the additional assumption $w\in W^{1,1}_{loc}(\RR^n)$
(see the end of this subsection for the case where $w\notin W^{1,1}$).
\begin{definition}[Sets of finite $w$-perimeter]
    We say that a set $E\subset\RR^n$ has \textit{locally finite $w$-perimeter} if for every $R>0$ we have
    \begin{equation}\label{eq:def_weightedperimeter}
        \sup\left\{\int_{E} \dive(wX)\dif\Lc^n\colon X\in C_c^1(B_R;\RR^n)\mbox{ and }|X|\le1\right\}<\infty.
    \end{equation}
\end{definition}
By Riesz's theorem, if $E$ has locally finite $w$-perimeter, then there exists a vector-valued Radon measure $\mu_{E;w}$ such that, for every $X\in C^1_c(\RR^n;\RR^n)$ it holds
\begin{equation}\label{eq:div_theo}
    \int_E\dive(wX)\dif\Lc^n = \int X\cdot\dif\mu_{E;w}.
\end{equation}
We define the $w$-perimeter of $E$ as the Radon measure
\begin{equation}
    \Per_{w}(E;A) = |\mu_{E;w}|(A),
\end{equation}
for every $A\subset\RR^n$ Borel.
We can also introduce the weighted Lebesgue and Hausdorff measures as
\begin{equation}
    \Lc_w^n(A)=\int_Aw\dif\Lc^n,\qquad\Hc_w^s(A) = \int_Aw\dif\Hc^s,
\end{equation}
for $s>0$. 
As for the classical perimeter, a structure theorem that allows us to describe the perimeter as a $\Hc_{\dista{}}^{n-1}$ measure holds true:

\begin{lemma}\label{lemma:structure}
    If $E$ is a set of locally finite $w$-perimeter, then $E$ is a set of locally finite perimeter in $\Omega$.
	In particular, the reduced boundary $\de^*E$ is well defined in $\Omega$ and it holds
    \begin{equation}\label{eq:weighted_structure}
        \Per_w(E;A) = \Hc^{n-1}_w(A\cap \de^* E)
    \end{equation}
    for every $A$ Borel. 
    Moreover, for all $X\in C^1_c(\RR^n;\RR^n)$, it holds
    \begin{equation}\label{eq:weighted_divergence}
        \int_E\dive(wX)\dif\Lc^n = \int_{\de^*E}wX\cdot\nu_E\dif\Hc^{n-1}
    \end{equation}
    where $\nu_E$ is defined below.
\end{lemma}

Before giving the proof of the \Cref{lemma:structure}, we recall and fix some notation. 
If $E$ is a set of locally finite perimeter in $\Omega$, we define its reduced boundary as the set $\de^*E$ of points $x\in\Omega$ for which 
\begin{equation}
	\nu_E(x) \coloneqq \lim_{r\to0} \frac{\mu_E(B_r(x))}{|\mu_E|(B_r(x))} \quad \mbox{exists and belongs to } \sphere,
\end{equation}
where $\mu_E\coloneqq -D\chi_E$ is the Gauss-Green measure of $E$ (see \cite[Chapter 12]{maggi12}).
Additionally, we define $E^{(1)}$ the set of points in the measure theoretic interior of $E$. 
Specifically
\begin{equation}
    E^{(1)} = \left\{x \in\RR^n : \frac{\Lc^{n}(E\cap B_r(x))}{|B_r|} = 1, \mbox{ for some } r>0 \right\},
\end{equation}
and with this definition it is not difficult observe that $E^{(1)}$ is an open set.

\begin{proof}
    \textbf{Step 1. }We first prove that $E$ is a set of locally finite perimeter in $\Omega$, namely that $\Per(E;D)<+\infty$ for all $D\Subset\Omega$.
    To this end, given some compact set $D\subset\{d_\Omega\ge\delta\}$ for some $\delta>0$ and given $X\in C^1_c(D;\RR^n)$, let $Y = \frac{1}{w}X$.
    Since $w\in W^{1,1}$ and $w\ge C^{-1}d_\Omega^a$, it holds $Y\in W^{1,1}(D;\RR^n)$ and $|Y|\le C\delta^{-a}$ for some $C$ depending on $w$ and $X$.
    By mollifying and using \eqref{eq:def_weightedperimeter}, we obtain
    \begin{equation}
		C\delta^{-a} \Per_{\dista{}}(E;D) \ge \int_{E}\dive(\dista{}Y) \dif\Lc^{n} = \int_{E}\dive(X) \dif\Lc^{n},
	\end{equation}
	and taking the supremum in $|X|\le1$ we get
    \begin{equation}\label{eq:per_awayfromdOmega}
    	\Per(E;D) \le C\delta^{-a} \Per_{\dista{}}(E;D)
    \end{equation}
    as claimed.

    \textbf{Step 2. }We show that \eqref{eq:weighted_structure} holds true for all but countably many $A = B_R$.
    First, notice that, by Step 1, \eqref{eq:weighted_structure} holds true under the additional assumption $E\Subset\Omega$.
    Otherwise, we define \begin{equation}
    	\Omega_j = \left\{x\in\Omega : d_\Omega(x) \ge \delta_j \right\},
    \end{equation}
    where $\{\delta_j\}_{j\in\NN}$ is a decreasing sequence such that $\delta_j\to0$ and for all $j\in\NN$ it holds
    \begin{equation}\label{eq:divergence_partition}
        \Hc^{n-1}(\de^*E\cap \{d_\Omega = \delta_j\})= \Hc^{n-1}(\de^*E \cap \de B_R \cap \Omega_j)= \Hc^{n-1}(\de B_R \cap \{d_\Omega = \delta_j\})=0,
    \end{equation}
    for all but countably many $R>0$.
    Such a sequence exists because $d_\Omega$ is Lipschitz and, by \eqref{eq:per_awayfromdOmega}, $E\cap \left\{ d_\Omega \ge 2^{-j}\right\}$ has locally finite perimeter.
    In particular, it follows that $E_{j,R} := E \cap B_R \cap \Omega_j$ has finite perimeter with 
    \begin{equation}
    	\de^* E_{j,R} = 
    	\Bigl( \de^*E \cap (\Omega_j \cap B_R)\Bigl) 
    	\cup
    	\Bigl( \{ d_\Omega = \delta_j\} \cap (E^{(1)}\cap B_R) \Bigl)
    	\cup 
    	\Bigl(\de B_R \cap (E^{(1)}\cap \Omega_j) \Bigl).
    \end{equation}
\begin{center}
    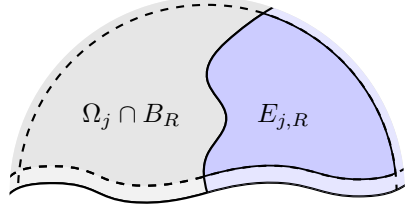
\begin{figure}
	\begin{tikzpicture}[use Hobby shortcut, thick,even odd rule]
		\begin{scope}
		 \clip (-2.6,-0.15) rectangle (2.6,2.6);
		 \clip	(0,0) circle (2.6);
		 	
		 \draw[fill = black!5, yscale=0.25] (-4,12) -- (-4,-0.5) .. (-3,-1) .. (-2,0.5) .. (-1,-0.5) .. (0,0) .. (1,0.5) .. (2,-0.2) .. (3,1) .. (4,0.5) -- (4,12) -- cycle;
		 
		 \draw[fill = black!10, yscale=0.25,shift={(0,0.9)},dashed] (-4,12) -- (-4,-0.5) .. (-3,-1) .. (-2,0.5) .. (-1,-0.5) .. (0,0) .. (1,0.5) .. (2,-0.2) .. (3,1) .. (4,0.5) -- (4,12) -- cycle;

		 \begin{scope}
		 	\clip[yscale=0.25] (-4,12) -- (-4,-0.5) .. (-3,-1) .. (-2,0.5) .. (-1,-0.5) .. (0,0) .. (1,0.5) .. (2,-0.2) .. (3,1) .. (4,0.5) -- (4,12) -- cycle;
		 	\clip (0,0) circle (2.6);


		 		\draw[fill = blue!10] (0,6) -- (3,3) .. (0.25,2) .. (0,1.5) .. (0.25,1) .. (0,0.5) .. (0,0) .. (1,-1) -- (1,-6) -- (6,-6) -- (6,6) -- cycle;
		 		
		 		\draw[yscale=0.25] (-4,12) -- (-4,-0.5) .. (-3,-1) .. (-2,0.5) .. (-1,-0.5) .. (0,0) .. (1,0.5) .. (2,-0.2) .. (3,1) .. (4,0.5) -- (4,12) -- cycle;

		 \begin{scope}
		 	\clip[yscale=0.25,shift={(0,0.86)}] (-4,12) -- (-4,-0.5) .. (-3,-1) .. (-2,0.5) .. (-1,-0.5) .. (0,0) .. (1,0.5) .. (2,-0.2) .. (3,1) .. (4,0.5) -- (4,12) -- cycle;
		 	\clip (0,0) circle (2.5);
		 		\draw[fill = blue!20] (0,6) -- (3,3) .. (0.25,2) .. (0,1.5) .. (0.25,1) .. (0,0.5) .. (0,0) .. (1,-1) -- (1,-6) -- (6,-6) -- (6,6) -- cycle;
		 	\begin{scope}
		 		\clip (0,6) -- (3,3) .. (0.25,2) .. (0,1.5) .. (0.25,1) .. (0,0.5) .. (0,0) .. (1,-1) -- (1,-6) -- (6,-6) -- (6,6) -- cycle;
		 			
		 			\draw[very thick, yscale=0.25,shift={(0,0.87)}] (-4,12) -- (-4,-0.5) .. (-3,-1) .. (-2,0.5) .. (-1,-0.5) .. (0,0) .. (1,0.5) .. (2,-0.2) .. (3,1) .. (4,0.5) -- (4,12) -- cycle;
		 			\draw (0,0) circle (2.49);

		 	\end{scope}

		 \end{scope}
		    \draw[ yscale=0.25,shift={(0,0.9)},dashed] (-4,12) -- (-4,-0.5) .. (-3,-1) .. (-2,0.5) .. (-1,-0.5) .. (0,0) .. (1,0.5) .. (2,-0.2) .. (3,1) .. (4,0.5) -- (4,12) -- cycle;

		 	\draw[dashed] (0,0) circle (2.49);

		 \end{scope}

		\end{scope}
		\node at (-1,1) {$\Omega_j\cap B_R$};
		\node at (1,1) {$E_{j,R}$};
		
	\end{tikzpicture}
    \caption{The set $E_{j,R}$ constructed in the proof of \Cref{lemma:structure}}
    \end{figure}
\end{center}
By \eqref{eq:divergence_partition} and monotone convergence it follows
    \begin{align}
    	\Per_{\dista{}}(E;B_R) = \lim_{j\to +\infty} \Per_{\dista{}}&(E; B_R \cap\Omega_j) \\
    	&= \lim_{j\to+\infty} \left(\Per_{\dista{}}(E_{j,R}) - \Per_{\dista{}}(\Omega_j\cap B_R;E^{(1)}) \right).
    \end{align}
    Again from \eqref{eq:divergence_partition} and direct computation we get
    \begin{equation}
    	\Per_{\dista{}}(E_{j,R}) - \Per_{\dista{}}(\Omega_j\cap B_R;E^{(1)}) = \Hc_{\dista{}}^{n-1}(\de^*E\cap B_R \cap \Omega_j).
    \end{equation}
    Lastly, by monotone convergence it follows that 
	\begin{equation}
		\Per_{\dista{}}(E;B_R) = \lim_{j\to+\infty} \Hc^{n-1}_{\dista{}}(\de^*E \cap B_R \cap \Omega_j)  = \Hc^{n-1}_{\dista{}}(\de^*E \cap B_R \cap \Omega),
	\end{equation}
	that concludes Step 2.

    \textbf{Step 3. }We now prove \eqref{eq:weighted_divergence}.
    Consider $X\in C^1_c(\RR^n;\RR^n)$ such that $\supp X\subset B_R$ and observe that, for all $j\in\NN$, it holds
    \begin{equation}
        \int_{E_{j,R}} \dive(wX) \dif\Lc^{n}
            = \int_{\de^*E\cap \Omega_j\cap B_R} w X\cdot \nu_E\dif \Hc^{n-1}
            - \int_{E^{(1)}\cap B_R\cap\{d_\Omega=\delta_{j}\}}w X\cdot \nabla d_\Omega\dif\Hc^{n-1},
    \end{equation}
    because $X = 0$ on $\de B_R$.
    Since $X\in C^{1}_c(\RR^n;\RR^n)$, the left hand side term converges as $j\to+\infty$.
    On the right hand side, the first term converges by the dominated convergence theorem, since $|X|\le1$ and $E$ is of locally finite $\dista{}$-perimeter.
    The last term converges to zero, since $|X|\le1$ and by \Cref{ass:weights} it holds
    \begin{equation}
    	 \Hc^{n-1}_{\dista{}}(E^{(1)}\cap B_R \cap \{d_\Omega = \delta_j\})
    	 \le C\delta_j^{a} \Hc^{n-1}(B_R \cap \{d_\Omega = \delta_j\})
    	 \le C\delta_j^a R^{n-1},
    \end{equation}
    where in the last inequality the fact that $\de\Omega$ is Lipschitz was used.
   Taking all together, we finally conclude that
    \begin{equation}
    	\int_{E}\dive(wX)\dif\Lc^{n} = \int_{\de^*E} w X\cdot\nu_E \dif\Hc^{n-1}
    \end{equation}
    as desired.

    \textbf{Step 4. }\eqref{eq:weighted_structure} for generic Borel sets $A$ is derived from \eqref{eq:weighted_divergence} by classical rectifiability results, such as \cite[Theorem 2.83]{ambrosio_fusco_pallara00}.
    
\end{proof}

\begin{remark}
    We shall always assume that a set $E$ with locally finite $w$-Perimeter satisfies
    \begin{equation}
        \de E = \{x\in\RR^n\colon 0<\Lc_w^n(E\cap B_r(x))<\Lc_w^n(B_r(x))\mbox{ for all }r>0\},
    \end{equation}
    since we can always replace $E$ with any set $E'$ such that $\Lc^n_w(E\Delta E')=0$ without affecting the behavior of $\Per_w(E;\cdot)$.
    With the above assumption, it also holds
    \begin{equation}
        \de E\cap \Omega = \supp\mu_{E;w}\cap\Omega.
    \end{equation}
\end{remark}

With a more solid background on weighted perimeters, we now recall for the reader's convenience the following 
\begin{definition}[Almost-minimizers of $\Per_w$]
    We say that $E\subset\overline{\Omega}$ is a $(\vartheta,\beta)$-minimizer of $\Per_w$ in $D\subset\RR^n$ if
    \begin{equation}
        \Per_w(E;B_r(x))\le(1+\vartheta r^\beta)\Per_w(F;B_r(x))
    \end{equation}
    for every set of locally finite perimeter $F$ such that $E\Delta F\Subset B_r(x)\subset D$.
    When $D=\RR^n$ or when the indication of $D$ is unnecessary, we simply say that $E$ is a $(\vartheta,\beta)$-minimizer of $\Per_w$.
    Moreover, if $E$ is a $(0,\beta)$-minimizer of $\Per_w$, we simply say that it is a minimizer of $\Per_w$.
\end{definition}
The assumption $E\subset\overline{\Omega}$ is purely technical. However, it will be convenient later on and it is not restrictive since for any set $E\subset\RR^n$ it holds $\Per_w(E;\cdot)=\Per_w({E\cap\overline{\Omega}};\cdot)$

\begin{remark}\label{remark:deficit} (Obtaining an additive deficit)
    Let $E$ be a $(\vartheta,\beta)$-minimizer of $\Per_w$ in $D\subset\RR^n$.
    By almost-minimality, for every $B_r(y)\subset D$ with $r\le1$ and every $0<s<r$, it holds
    \begin{equation}
        \Per_w(E;B_r(y))\le (1+\vartheta r^\beta)\left(\Per_w\big(E;B_r(y)\setminus \overline{B_s(y)}\big)+\Hc_w^{n-1}(\partial B_s(y))\right).
    \end{equation}
    Letting $s\nearrow r$, we obtain
    \begin{equation}\label{eq:perboundonqmin}
        \Per_w(E;B_r(y))\le C(1+\vartheta)\left(\sup_{B_r(y)}w\right)r^{n-1}
    \end{equation}
    for some constant $C$ depending only on $n$.
    In particular, for every $F$ such that $E\Delta F\Subset B_r(y)\subset D$, provided $\vartheta\le1$ and $r\le1$, it holds
    \begin{equation}\label{eq:percompetitorbound}
        \Per_w(E;B_r(y))\le \Per_w(F;B_r(y))+ C\left(\sup_{B_r(y)}w\right)\vartheta r^{n-1+\beta}
    \end{equation}
\end{remark}

\begin{remark}[Scaling]\label{remark:scaling}  
    If $E$ is a $(\vartheta,\beta)$-minimizer of $\Per_w$ in $D$ then, for any $r>0$ and $x_0\in\RR^n$, $\frac{E-x_0}{r}$ is a $(\vartheta r^\beta,\beta)$-minimizer of $\Per_{\tilde w}$ in $\frac{D-x_0}{r}$, where $\tilde w(x) \coloneqq w(x_0+rx)$.
    Notice that a $(\vartheta,\beta)$-minimizer of $\Per_w$ is also a $(\vartheta,\beta)$-minimizer for $\Per_{cw}$ for any $c>0$.
    Therefore, in the particular case where the weight is $\dista{\Omega}\coloneqq d_\Omega^a$, if $E$ is a $(\vartheta,\beta)$-minimizer of $\Per_{\dista{\Omega}}$, then $\frac{E}{r}$ is a $(\vartheta r^\beta,\beta)$-minimizer of $\Per_{\dista{\Omega/r}}$.
\end{remark}

\begin{proposition}[Lower Semicontinuity]\label{rmk:lscperimeter}
    Let $\{\Omega_j\}_{j\in\NN}$ be a family of open sets and $\{w_j\}_{j\in\NN}$ be a family of weights satisfying \Cref{ass:weights}, and let $\{E_j\}_{j\in\NN}$ be a family of sets of finite $w_j$-perimeter.
    Assume there exist an open set $\Omega$, $E\subset\overline{\Omega}$ and $w$ satisfying \Cref{ass:weights} such that 
    \begin{equation}
        E_j \longrightarrow E\; \mbox{in } L^1_{loc}(\RR^n),\qquad \mbox{and}\qquad w_j \longrightarrow w \;\in W^{1,1}_{loc}(\RR^n).
    \end{equation}
    Then
    \begin{equation}
\Per_{\dista{}}(E;D)\le\liminf_{j\to\infty}\Per_{w_j}(E_j;D)
\end{equation}
for all open sets $D\subset\RR^n$ and  
\begin{equation}\label{eq:weakstarconv}
\nu_{E_j} \Hc^{n-1}_{\dista{j}} \llcorner \partial^*E_j \weakstar \nu_{E} \Hc^{n-1}_{\dista{}} \llcorner \partial ^*E.
\end{equation}
\end{proposition}

\begin{proof}
    Since $\chi_{E_j}w_j$ and $\chi_{E_j}\nabla w_j$ converge weakly in $L^1_{loc}$, for all $X\in C^{1}_c(\RR^n;\RR^n)$ it holds
    \begin{equation}
        \int_{E} \dive(wX) = \lim_{j\to+\infty} \int_{E_j} \dive(w_jX)\dif\Lc^{n} \le \liminf_{j\to+\infty} \Per_{w_j}(E),
    \end{equation}
    and we conclude taking the supremum in $X$.
    Since $E$ is of finite $w$-perimeter and $w$ satisfies \Cref{ass:weights}, \Cref{lemma:structure} applies and for all $X\in C^1_c(\RR^n)$
    \begin{equation}
        \int_{\de^*E} w X\cdot \nu_E \dif \Hc^{n-1} = \lim_{j\to +\infty} \int_{E_j} \dive (w_jX)\dif\Lc^{n} =\lim_{j\to+\infty} \int_{\de^*E_j} w_j X\cdot \nu_{E_j} \dif \Hc^{n-1}.
    \end{equation}
\end{proof}

\begin{remark}
    In the special case $\dista{j} = \dista{\Omega_j}$, $\dista{} = \dista{\Omega}$, and if $\de\Omega_j$ are $\varkappa$-flat, with $\varkappa\le 1$, we can replace the assumption on $\{\dista{j}\}_{j\in\NN}$ with the geometric assumption
    \begin{equation}
        \de \Omega_j \xrightarrow[j\to+\infty]{} \de \Omega \; \mbox{ in the Hausdorff distance,}
    \end{equation}
    meaning that 
    \begin{equation}
    	d_{\Hc}(\de\Omega_j,\de\Omega) := \max \left\{ \sup_{x\in\de\Omega_j}\dist(x,\de\Omega),\sup_{y\in\de\Omega} \dist(x,\de\Omega_j)  \right\} \longrightarrow 0.
    \end{equation}

    Indeed, $\de\Omega_j \to \de\Omega$ implies that $\dista{\Omega_j} \to \dista{\Omega}$ uniformly, and their gradients are equi-bounded in $L^{p}(B_1)$, for some $p>1$, since     \begin{equation}
        \int_{\Omega_j\cap B_1}|\nabla \dista{\Omega_j}|^p \dif\Lc^{n} \le C_p\left(1 +  \int_{\Omega_j \cap B_1 \cap \{d_{\Omega_j} < 1/10\}} d_{\Omega_j}^{p(a-1)} \dif\Lc^{n}\right),
    \end{equation}
    where $C_p>0$ is universal.    
    The sets $\de\Omega_j$ are $\varkappa$-flat, with $\varkappa\le1$ and thus
    \begin{equation}
        \int_{\Omega_j \cap B_1 \cap \{d_{\Omega_j} < 1/10\}} d_{\Omega_j}^{p(a-1)} \dif\Lc^{n}\le C_p \int_{0}^{1/10} t^{p(a-1)} \dif t,
    \end{equation}
    which is uniformly bounded if $a\ge1$ or $p< \frac{1}{1-a}$.

    Therefore $\nabla w_j$ are uniformly bounded in $L^p$ and thus there exists $v\in L^p_{loc}(B_1;\RR^n)$ and a subsequence $\{\nabla w_{j(\ell)}\}_{\ell\in\NN}$ such that $\nabla w_{j(\ell)}\weakly v$ in $L^p_{loc}$ (and in particular in $L^1_{loc}$). 
    
    Looking at the behavior with smooth test functions, it is not difficult to show that $v = \nabla w$ and since the limit of subsequences is unique, the whole sequence of $\{w_j\}_{j\in\NN}$ converges weakly to $w$ in $W^{1,1}_{loc}(B_1)$. 
\end{remark}

We conclude this subsection with a quick observation about the case where $w$ satisfies \Cref{ass:weights}, but $w\notin W^{1,1}_{loc}(\RR^n)$.
In this case, we cannot define $\Per_{\dista{}}$ as in \eqref{eq:def_weightedperimeter}, since $\int_E \dive(wX)\dif\Lc^{n}$ is not well defined.
On the other hand, if $\dista{}$ satisfies the \Cref{ass:weights}, and $E$ is a set of locally finite $\dista{\Omega}$-perimeter, then we can define 
\begin{equation}\label{eq:def_moregeneralperimeter}
	\Per_{\dista{}}(E;B_R) := \int_{\de^*E \cap B_R}  \dista{} \dif\Hc^{n-1}.
\end{equation}
Thanks to \Cref{lemma:structure}, this expression is well defined.
From \Cref{ass:weights}, we know that $\frac{1}{C}\le\frac{\dista{}}{\dista{\Omega}}\le C$, therefore
\begin{equation}\label{eq:comparable_weights}
  	\frac{1}{C} \Per_{\dista{\Omega}}(E;B_R) \le \Per_{\dista{}}(E;B_R) \le  C\Per_{\dista{\Omega}}(E;B_R).
\end{equation}

Thus $\Per_{\dista{\Omega}}$ and $\Per_{\dista{}}$ are comparable, and we can simply define the sets of finite $\dista{}$-perimeter as of finite $\dista{\Omega}$-perimeter.
This definition is sufficient for \Cref{prop:changing_weight} to hold, thus we recover the regularity properties of almost-minimizers of $\Per_{\dista{}}$ from the analogous result for $\Per_{\dista{\Omega}}$.

\subsection{Compactness}

From now on we turn our attention to the special case $\dista{\Omega}(x) = d_\Omega(x)^a$.

\begin{proposition}[Compactness from bounds on the $\dista{\Omega}$-perimeter]\label{prop:compactness_frombound}
    Let $\{\Omega_j\}_{j\in\NN}$, $\Omega_j\subset B_1$ a sequence of $\varkappa$-flat open sets, with $\varkappa\le 1$, and $\{E_j\}_{j\in\NN}$ with $E_j\subset \overline{\Omega_j}\cap B_1$ be a sequence of sets of locally finite $\dista{\Omega_j}$-perimeter such that 
    \begin{equation}
        \liminf_{j\to+\infty}\Per_{\dista{\Omega_j}}(E_j;B_1) < +\infty.
    \end{equation}
    Then there exists {$\Omega = \{x_n > g(x') \}\subset\RR^n$,  with $\|g\|_{C^{1,\alpha/2}}\le 1$}, $E\subset\Omega$ of locally finite $\dista{\Omega}$-perimeter and a sequence $j(\ell) \to +\infty$ such that
    $\de\Omega_{j(\ell)} \to \de\Omega$ locally in $B_1$ with respect to the Hausdorff distance,
    \begin{equation}
         E_{j(\ell)} \xrightarrow[]{L^{1}_{loc}(B_1)} E,
         \qquad \mbox{and}\qquad 
         \Hc^{n-1}_{\dista{\Omega_{j(\ell)}}} \llcorner\de^*E_{j(\ell)} \weakstar \Hc^{n-1}_{\dista{\Omega}}\llcorner \de^*E \mbox{ in } B_1.
    \end{equation}
\end{proposition}

\begin{proof}
    The convergence of $\de\Omega_j$ is a classical application of Arzelà-Ascoli Theorem, thus in the rest of the proof we assume that $\de\Omega_j \to \de\Omega$ in the Hausdorff distance.
    
    For the compactness of the $E_j$ we rely on the compactness theorem for the usual perimeter (see for instance \cite[Theorem 12.26]{maggi12}):
    for all $\delta >0$ and $j\gg 1$ such that $d_{\Hc}(\de\Omega,\de\Omega_j) < \delta$, it holds 
    \begin{equation}
        \Per(E_j;B_1\cap\{d_{\Omega}\geq 3\delta\}) \leq C\delta^{-a} \Per_{\dista{\Omega}}(E_j;B_1\cap\{d_{\Omega}\geq 3\delta\}).
    \end{equation}
    Therefore for every $\delta>0$ there exists $\{E_{j(\ell)}\}_{\ell\in\NN}$ that converges in $L^1(B_R\cap\{\dista{\Omega}\geq 3\delta\})$ to some set $E_\delta$. 
    Taking $\delta = 1/m$, for $m\in\NN$, a diagonal argument shows that 
    \begin{equation}
        E_{j(\ell)} \xrightarrow[\ell\to +\infty]{} E \;\;\mbox{in }L^1_{loc}(B_1\cap\{\dista{\Omega}\geq1/m\}),\quad \mbox{where } E = \bigcup_{m\in\NN} E_{1/m}.
    \end{equation}
    Since $\de\Omega$ is the graph of a $C^{1}$-function with bounded derivative, then $\Lc^n(\{d_{\Omega}\le 1/m\}) \leq \frac{C}{m}$, and
    \begin{equation}
        \limsup_{\ell\to+\infty} \Lc^n(E\triangle E_{j(\ell)}) 
        \leq \frac{C}{m} + \limsup_{\ell\to+\infty} \Lc^n(E\triangle E_{j(\ell)} \cap \{\dista{\Omega}\geq 1/m\})\leq \frac{C}{m}.
    \end{equation}
    We conclude by taking $m\to+\infty$ and using \Cref{rmk:lscperimeter}.
\end{proof}

\begin{proposition}[Compactness for almost-minimizers]\label{prop:compactnessqmin}
    Let $\{\Omega_j\}_{j\in\NN}$ be a sequence of $\varkappa$-flat open set, with $\varkappa\le1$, and $\{E_j\}_{j\in\NN}$, $E_j\subset \Omega_j$ a sequence of $(\vartheta_j,\beta)$-minimizers of the $\dista{\Omega_j}$-perimeter.
    If 
    \begin{equation}
        \vartheta \coloneqq \liminf \vartheta_j < +\infty
    \end{equation}
    then there exist a subsequence $\{j(\ell)\}_{\ell\in\NN}$, $\Omega\subset\RR^n$, and a $(\vartheta,\beta)$-minimizer $E\subset \Omega$ of $\Per_\dista{\Omega}$ such that $\de\Omega_{j(\ell)}\to \de\Omega$ locally in $B_1$ with respect to the Hausdorff distance and
    \begin{equation}
        E_{j(\ell)}\xrightarrow[]{L^1_{loc}(B_1)} E;\qquad  \Hc^{n-1}_{\dista{\Omega_j}}\llcorner\de^*{E_{j(\ell)}}\weakstar\Hc^{n-1}_{\dista{\Omega}}\llcorner\de^*E
        \mbox{ in $B_1$};
    \end{equation}
    Moreover, $\de E_{j(\ell)}\to \de E$ locally in $B_1$ with respect to the Hausdorff distance.
\end{proposition}

\begin{proof}
    The convergence of $\de\Omega_j$ follows as in \Cref{prop:compactness_frombound}.
    Without loss of generality we therefore assume $\de\Omega_j\to\de\Omega$ and $\vartheta = \lim \vartheta_j < +\infty$.

    From \eqref{eq:perboundonqmin}, it follows that 
    \begin{equation}
        \liminf\Per_{\dista{\Omega_j}}(E_j;B_1) \le \liminf C(1+\vartheta_j) \le \vartheta C,
    \end{equation}
    and thus from \Cref{prop:compactness_frombound} there exists a subsequence $\{E_{j(\ell)}\}_{\ell\in\NN}$ and a set $E\subset\Omega$ of locally finite $\dista{\Omega}$-perimeter such that $E_{j(\ell)} \to E$. 
    
    We only need to prove that if $E$ is the limit of a sequence of $(\vartheta_j,\beta)$-minimizers, then $E$ is $(\vartheta,\beta)$-minimizer where $\vartheta = \lim \vartheta_j$.

    If $F$ is such that $E\triangle F \Subset B_r(x)$ for some $x\in\RR^n$ and $r>0$, then for all except at most countably many $s<r$ it holds 
    \begin{equation}\label{eq:goodsplitcond}
        \Hc^{n-1}_{\dista{\Omega_j}}(\partial^*E_{j} \cap \partial B_s(x) ) =0 \quad \mbox{ and } \quad \Hc^{n-1}_{\dista{\Omega_j}}(\partial^*F \cap \partial B_s(x) ) =0.
    \end{equation}
    Let now $F_{s,j} = \left(F \cap \overline{B_s(x)}\right) \cup \left(E_{j}\cap (B_r(x)\setminus \overline{B_s(x)})\right)$ so that
    \begin{equation}
        \begin{split}
            \Per_{\dista{\Omega}}(E;B_r(x)) \leq \liminf_{j\to +\infty} \Per_{\dista{\Omega_j}}(E_{j};B_r(x)) \leq \liminf_{j \to+\infty} (1+\vartheta_{j} r^\beta) \Per_{\dista{\Omega_j}}(F_{s,j};B_r(x)).
        \end{split}
    \end{equation}
    Since we chose $s<r$ such that \eqref{eq:goodsplitcond} holds, we get
    \begin{equation}
        \Per_{\dista{\Omega_j}}(F_{s,j}; B_r(x)) \leq \Per_\dista{\Omega_j}(F;B_s(x)) + \Per_\dista{\Omega_j}(E_{j};B_r(x)\setminus\overline{B_s(x)}),
    \end{equation}
    and taking $s\nearrow r$ from \Cref{rmk:lscperimeter} we get 
    \begin{equation}
        \Per_{\dista{\Omega}}(E;B_r(x)) 
        \leq \liminf_{j \to+\infty} (1+\vartheta_{j} r^\beta) \Per_{\dista{\Omega_j}}(F;B_r(x)) = (1+\vartheta r^\beta) \Per_{\dista{\Omega}}(F;B_r(x)) \qedhere
    \end{equation}
    
\end{proof}

In the previous propositions we lose information on $\de\Omega$, since the Arzelà-Ascoli theorem does not prevent some loss of regularity of the limit.
This is no issue, because in the applications we always consider sequences defined by a blow-up procedure.
Indeed, taking into account \Cref{prop:compactnessqmin} and \Cref{remark:scaling} we get that blow-up of $(\vartheta,\beta)$-minimizers are $\Per_a$-minimizers, as stated next.
Before proceeding, we introduce the notation 
\begin{equation}
    \RR^n_+ = \left\{ x \in\RR^n : x_n\ge 0 \right\} \quad \mbox{and} \quad D^+ = D \cap \RR^{n}_+,
\end{equation}
for all sets $D\subset \RR^n$.
We also define $\Per_a$, and write the $a$-perimeter, in a Borel set $A$ as
\begin{equation}
    \Per_a(E;A) = \Per_{\dista{\RR^n_+}}(E;A) =\int_{\de^*E\cap A}(x_n^+)^a \dif\Hc^{n-1},
\end{equation}
where $x_n^+ = \max\{0,x_n\}$. 
In the same way, we also define $\Lc^{n}_a$ and $\Hc^{n-1}_a$.

\begin{corollary}[Blow-ups]\label{cor:blowupqm}
    Let $\Omega$ be a $\varkappa$-flat open set, $E$ a $(\vartheta,\beta)$-minimizer, $x_0\in \de E \cap \de \Omega$, and $\Phi$ be be a linear isometry that maps $\{ x\cdot\nu_{\Omega}(x_0) \ge 0 \} \mapsto \{ x_n\ge 0\}$
    Then there exist a sequence $r_j\to0$ and a local $\Per_a$-minimizer $E_0$ such that 
    \begin{equation}
        \frac{1}{r_j}\Phi(E-x_0)\xrightarrow[]{L^1_{loc}} E_0,\quad \mbox{and}\quad \frac{1}{r_j}\Hc^{n-1}_{\dista{\Omega_j}}\llcorner \partial^*\Phi(E-x_0) \weakstar \Hc^{n-1}_a\llcorner \partial^*E_0.
    \end{equation}
    More in general, whenever there exists $r_j\to0$ and $E_0\subset\RR^n_+$ such that $\frac{1}{r_j}\Phi(E-x_0)\xrightarrow[]{L^1_{loc}} E_0$, then $E_0$ is a $\Per_a$-minimizer.
\end{corollary}

\subsection{Density estimates}
Here we introduce two density estimates, one holding at points in the interior of $\Omega$ and one up to the boundary of $\Omega$.

\begin{proposition}[Density estimates at the boundary]\label{prop:density_a}
    There exists a universal constant $C$ such that,
    if $\Omega\subset\RR^n$ is $\varkappa$-flat with $\varkappa\le\frac{1}{C}$, and $E$ is a $(\vartheta,\beta)$-minimizer of $\Per_{\dista{\Omega}}$ in $B_1$, then for all $x_0\in\de E\cap \de\Omega \cap B_{1/2}$ and $r\in(0,1/4)$ such that $\vartheta r^\beta\le 1$ it holds
    \begin{equation}\label{eq:densityest}
        \frac{1}{C}\le\frac{\Lc^n_{\dista{\Omega}}(E\cap B_r(x_0))}{r^{n+a}}\le C.
    \end{equation}
\end{proposition}

\begin{proof}
    The proof relies on a weighted isoperimetric inequality, which is guaranteed by \cite{CabreRosOtonSerra2016}. 
    Once we reduce ourselves to the assumptions of that result, the proof is based on classical arguments, that we briefly recollect, referring to \cite[Theorem 16.14]{maggi12} for further details.

    \textbf{Step 1.}
    We call $g$ the function introduced in the definition of $\varkappa$-flatness (\Cref{ass:boundaryof_Omega}).
    By elementary computations analogous to the ones in \Cref{lemma:notes_on_distance}, we get that for some universal $C>0$, it holds
    \begin{equation}
    \dista{\Omega}(x) \ge (1-C\varkappa^2) |x_n - g(x')|^a \ge \frac{1}{2}|x_n - g(x')|^a,
    \end{equation}
    for all $x\in\Omega$, provided $\varkappa>0$ is sufficiently small.
    Thus for all $x_0\in \de \Omega \cap B_{1/2}$ and $r<1/4$ it holds
    \begin{multline}
                \Per_{\dista{\Omega}}(E;B_r(x_0)) = \int_{\de^*E \cap B_r(x_0)} \dista{\Omega}(x) \dif \Hc^{n-1} 
        \ge \frac{1}{2}\int_{\de^*E\cap B_r(x_0)} |x_n - g(x')|^a\dif \Hc^{n-1}.
    \end{multline}
    We now consider the diffeomorphism $\Psi(x',x_n) = (x',x_n - g(x'))$, that sends $\{x_n = g(x')\}$ in $\{x_n=0\}$.
    It holds
    \begin{equation}
        \nabla \Psi  = Id - \nabla_{x'} g \otimes e_n,\quad |\nabla \Psi | \le 1 + 2\varkappa,\quad \mbox{and}\quad J\Psi =1.
    \end{equation}
    Gathering these estimates and exploiting the isoperimetric inequality for the flat case proved in \cite[Theorem 1.3]{CabreRosOtonSerra2016}, we find 
    \begin{equation}\label{eq:curvedisoperimetric}
        \frac{\Per_{\dista{\Omega}}(E\cap B_r(x_0))}{\Lc_{\dista{\Omega}}^n(E\cap B_r(x_0))^{\frac{n+a-1}{n+a}}} 
        \ge \frac{1}{2}  \frac{\Per_{a}(\Psi(E\cap B_r(x_0))}{\Lc_{a}^n(\Psi(E\cap B_r(x_0))^{\frac{n+a-1}{n+a}}}
        \ge c.
    \end{equation}

    \textbf{Step 2.}
     Let $E$ be a $(\vartheta,\beta)$-minimizer of $\Per_{\dista{\Omega}}$ in $B_1$, let $x_0\in\partial E\cap\de\Omega\cap B_{1/2}$ and consider $0<t<s<1/4$ for which $\Per_{\dista{\Omega}}(E;\de B_t(x_0)) = \Per_{\dista{\Omega}}(E;\de B_s(x_0)) = 0$.
     By $(\vartheta,\beta)$-minimality it holds  
     \begin{equation}
         \begin{split}
             \Per_{\dista{\Omega}}(E;B_s(x_0)) &\leq (1+\vartheta s^\beta)\Per_{\dista{\Omega}}(E\setminus B_t(x_0);B_s(x_0)) \\
             &= (1+\vartheta s^\beta)\left(\Per_{\dista{\Omega}}(E;B_s(x_0)\setminus\overline{B_t(x_0)}) + \Hc^{n-1}_{\dista{\Omega}}(\de B_t(x_0)\cap E^{(1)}) \right)\\
         \end{split}
     \end{equation}
     thus, by letting $s\searrow t$ and using Fatou's lemma, we get
     \begin{equation}
         \Per_{\dista{\Omega}}(E;B_t(x_0)) \leq (1+\vartheta t^\beta)\Hc^{n-1}_{\dista{\Omega}}(\partial B_t(x_0) \cap E^{(1)}).
     \end{equation}
     Adding $\Hc^{n-1}_{\dista{\Omega}}(\partial B_t(x_0) \cap E^{(1)})$ to both sides, we then get
     \begin{equation}\label{eq:isoperimetricdensity}
         c[\Lc^{n}_{\dista{\Omega}}(E\cap B_t(x_0))]^{\frac{n+a-1}{n+a}} 
         \leq \Per_{\dista{\Omega}}(E\cap B_t(x_0)) 
         \leq (2+\vartheta t^\beta)\Hc^{n-1}_{\dista{\Omega}}(\partial B_t(x_0) \cap E^{(1)}),
     \end{equation}
     where the first inequality comes from \eqref{eq:curvedisoperimetric}.
     Calling $m(t) = \Lc_{\dista{\Omega}}^{n}(E\cap B_t(x_0))$, the assumption that $\vartheta r^\beta\le1$ and \eqref{eq:isoperimetricdensity} above yield the differential inequality 
     \begin{equation}
         c m(t)^{\frac{n+a-1}{n+a}} \leq 3 m'(t)
     \end{equation}
     for all $t\in(0,r)$.
     Since $x\in\partial E$, $m(t) >0$ for all positive radius, hence integrating from $0$ to $r$ and using the fact that $m(0^+)=0$,  we obtain
     \begin{equation}\label{eq:lowerdensityestimate}
         c r^{n+a} \leq \Lc_{\dista{\Omega}}^{n}(E\cap B_r(x_0)).
     \end{equation}
     The complementary set $E^c$ is a $(\vartheta,\beta)$-minimizer as well, hence
     \begin{equation}
         c r^{n+a} \le \Lc^{n}_{\dista{\Omega}}(E^c\cap B_r(x_0)) 
         = \Lc^{n}_{\dista{\Omega}}(B_r(x_0)) - \Lc^{n}_{\dista{\Omega}}(E\cap B_r(x_0)),
     \end{equation}
     and observing that for $x_0\in\de\Omega$, $\Lc^{n}_{\dista{\Omega}}(B_r(x_0)) \le \left(\frac{1}{2}+ C\varkappa\right) C r^{n+a}$ for $r\le 1/4$, we conclude
     \begin{equation}
         \Lc_{\dista{\Omega}}^{n}(E\cap B_r(x_0)) \le C r^{n+a}.
     \end{equation}
\end{proof}

\begin{remark}
	In the previous proposition, we need the assumption $x_0\in\de\Omega$ only to achieve the upper density estimate.
	Indeed we need it to prove that the mass of the balls centered in $x_0$ are comparable with $r^{n+a}$, but we do not need it to achieve \eqref{eq:lowerdensityestimate}.
	
\end{remark}

\begin{proposition}[Density estimates away from $\de\Omega$]\label{lemma:density_estimates_away}
	There exists a universal constant $C>0$ such that, if $\Omega\subset\RR^n$ is $\varkappa$-flat with $\varkappa\le \frac{1}{C}$ and $E$ is a $(\vartheta,\beta)$-minimizer of $\Per_{\dista{\Omega}}$ in $B_1$, then for all $x_0\in\de E \cap B_{1/2}\cap \Omega$ and $r < \frac{d_{\Omega}(x_0)}{2}$ such that $\vartheta r^\beta \le 1$ it holds
	\begin{equation}
		\frac{d_\Omega(x_0)^a}{C} \le \frac{\Lc^{n}_{\dista{\Omega}}(E\cap B_r(x_0)) }{r^n} \le C d_\Omega(x_0)^a.
	\end{equation}
\end{proposition}

\begin{proof}
    As we observed in \Cref{sec:propalmostminimizers}, $E$ is of finite perimeter away from $\de\Omega$. 
    Therefore the Euclidean isoperimetric inequality applies giving 
    \begin{equation}\label{eq:equclideanisoper}
        \frac{1}{C} \le \frac{\Per(E\cap B_r(x_0))}{\Lc^{n}(E\cap B_r(x_0))^{\frac{n-1}{n}}}.
    \end{equation}
    Using $(d_\Omega(x_0)-r)^a\le \dista{\Omega}\le(d_\Omega(x_0)+r)^a$ in $B_r(x_0)$, from \eqref{eq:equclideanisoper} we get
    \begin{equation}
        \frac{1}{C} \le \frac{\Per(E\cap B_r(x_0))}{\Lc^{n}(E\cap B_r(x_0))^{\frac{n-1}{n}}} 
        \le \frac{(d_\Omega(x_0)+r)^{a\frac{n-1}{n}}}{(d_\Omega(x_0)-r)^a}
        \frac{\Per_{\dista{\Omega}}(E\cap B_r(x_0))}{\Lc^{n}_{\dista{\Omega}}(E\cap B_r(x_0))^{\frac{n-1}{n}}}.
    \end{equation}
    Since $r< \frac{d_{\Omega}(x_0)}{2}$, it follows that 
    \begin{equation}
        \frac{(d_\Omega(x_0)+r)^{a\frac{n-1}{n}}}{(d_\Omega(x_0)-r)^a} 
        = d_{\Omega}(x_0)^{-\frac{a}{n}}
        \frac{\left( 1 + \frac{r}{d_\Omega(x_0)} \right)^{a\frac{n-1}{n}}}
        {\left(1-\frac{r}{d_{\Omega}(x_0)}\right)^{a}}
        \le  C d_\Omega(x_0)^{-\frac{a}{n}},
        \end{equation}
    that implies 
    \begin{equation}\label{eq:isoper1_away}
    	\frac{d_{\Omega}(x_0)^{\frac{a}{n}}}{C} \le 
        \frac{\Per_{\dista{\Omega}}(E\cap B_r(x_0))}{\Lc^{n}_{\dista{\Omega}}(E\cap B_r(x_0))^{\frac{n-1}{n}}}.
    \end{equation}
    Next, carefully following Step 2 of the proof of \Cref{prop:density_a}, from \eqref{eq:isoper1_away} we deduce
    \begin{equation}
    	 \frac{d_\Omega(x_0)^{a}}{C} r^n \le \Lc^{n}_{\dista{\Omega}}(E \cap B_r(x_0)).
    \end{equation}
    We then conclude as in \Cref{prop:density_a} after observing that
    \begin{equation}
        \Lc_{\dista{\Omega}}^{n}(B_r(x_0)) \le (d_{\Omega}(x_0)+r)^a C r^n \le C d_\Omega(x_0)^ar^{n},
    \end{equation}
    whence   
    \begin{equation}
    	\Lc^{n}_{\dista{\Omega}}(B_r(x_0) \cap E) \le \Lc^{n}_{\dista{\Omega}}(B_r(x_0)) - \Lc^{n}_{\dista{\Omega}}(E^c\cap B_r(x_0)) \le C d_\Omega(x_0)^a r^n
    \end{equation}
\end{proof}

From the density estimates, classical properties of almost minimizers of the perimeter follow also for almost minimizers of the weighted perimeter. In particular:

\begin{corollary}
	Let $\Omega$ and $E$ be as in the hypotheses of \Cref{prop:density_a} and \Cref{lemma:density_estimates_away}. Then, up to a set of measure $\Lc^{n}_{\dista{\Omega}}$-negligible, $E^{(1)}=\interior E$ and 
	\begin{equation}
		\de E = \left\{ x\in\RR^n:  0 < \frac{\Lc^{n}_{\dista{\Omega}}(B_r(x)\cap E)}{\Lc^{n}_{\dista{\Omega}}(B_r(x))} < 1 \mbox{ for all }r>0 \right\}.
	\end{equation}
\end{corollary}

\begin{corollary}\label{cor:convergence_Hausdorff}
	Let $\{\Omega_j\}_{j\in\NN}$ a sequence of $\varkappa$-flat open sets, with $\varkappa\le1$, and $\{E_j\}_{j\in\NN}$ be a sequence of $(\vartheta_j,\beta)$-minimizers of $\Per_{\dista{\Omega_j}}$ in $B_1$ with
	\begin{equation}
		\vartheta = \sup_{j\in\NN}\vartheta_j < +\infty.
	\end{equation}
	If there exist a $\varkappa$-flat set $\Omega$ and a $(\vartheta,\beta)$-minimizer $E\subset \overline\Omega$ such that $\de\Omega_j \to \de\Omega$ with respect to the Hausdorff distance and $E_j \to E$ in $L^1_{loc}(B_1)$, then
		$\de E_j \longrightarrow \de E$ in the Hausdorff distance.
\end{corollary}

\begin{proof}
    If not, then there exist $\delta>0$ and a subsequence $\{x_{j(\ell)}\}_{\ell\in\NN}\subset \overline{B_r}$ (for some $r<1$) such that $x_{j(\ell)}\in\partial E_{j(\ell)}$ and $\dist(x_{j(\ell)},\partial E) \geq \delta$ (or $\dist(x_{j(\ell)},\partial E) \geq \delta$ and $x_{j(\ell)}\in\partial E$).
    Without loss of generality we argue for the first case. 
    By taking a fixed $\rho < \delta$ such that $\vartheta_j \rho^\beta \le 1$ for all $j\ge 1$, then by \eqref{eq:lowerdensityestimate}, for all $\ell\in\NN$ it holds
    \begin{equation}
        \Lc^{n}(E\Delta E_j) \ge \Lc^{n}(E\cap B_\rho(x_{j(\ell)})) \ge \Lc^{n}_{\dista{\Omega_j}}(E\cap B_\rho(x_{j(\ell)})) \ge \frac{1}{C}\rho^{n+a},     \end{equation}
    where the second inequality follows since $\dista{\Omega_j}\le 1$ in $B_1$.
    This is a contradiction, because we assumed $E_j\xrightarrow[]{} E$ in $L^{1}_{loc}(B_1)$.
    
\end{proof}

\section{Variational and viscosity properties of (almost)-minimizers}\label{sec:var_visc_properties}

\subsection{First variation and monotonicity formula for minimizers}

In this subsection, we explore some properties that minimizers (rather than \textit{almost} minimizers) of the weighted perimeter enjoy.
Notice that minimizer arise as blow-up limits of almost minimizers: as we are mostly interested in the properties of such blow-up limits, the results in this subsection are about minimizers of $\Per_a$ in the flat domain $\RR^n_+$.

Before stating the next result, we introduce the following notation: if $E$ is a set of locally finite perimeter and $X\in C^1_c(\RR^n;\RR^n)$, then the quantity
\begin{equation}
    \dive_E X(x) = \dive X - \langle\nabla X(x)\nu_E(x),\nu_E(x)\rangle
\end{equation}
is well defined for $\mu_E$-almost every $x$.

\begin{proposition}[First variation]\label{prop:first_var}
    Let $E\subset\RR^n_+$ be a minimizer of $\Per_{a}$ in $A\subset\RR^n$.
    For every $X\in C^1_c(A\cap\RR^{n}_+;\RR^n)$ such that $X\cdot e_n=0$ on $\{x_n=0\}$, it holds
    \begin{equation}\label{eq:stationarity}
        \int x_n^a\left(\dive_E X + a\frac{X\cdot e_n}{x_n}\right)\dif\mu_E = 0.
    \end{equation}
\end{proposition}
\begin{proof}
    Let $X\in C^1_c(A\cap\RR^{n}_+;\RR^n)$ be such that $X\cdot e_n=0$ on $\{x_n=0\}$ and, for $t\in\RR$, let $f_t(x) = x+tX(x)$.
    For $t$ small enough, $f_t(E)\Delta E\Subset A$ and standard computations (see, for instance, \cite[Chapter 17]{maggi12}) give
    \begin{align}
        \Per_{a}(f_t(E),A) &= \int_{f_t^{-1}(A)}((f_t)_n)^a Jf_t\,|(\nabla(f_t^{-1})\circ f)^*\nu_E|\dif\mu_E\\        
        &=\Per_a(E,A) + t\left(\int x_n^a\dive_E X\dif\mu_E + \int a x_n^{a-1} e_n\cdot X\dif\mu_E\right) + O(t^2)
    \end{align}
    where we have used the fact that, due to $X\cdot e_n=0$ on $\{x_n=0\}$, it holds $((f_t(x))_d)^a=x_n^a\left(1+at\frac{X(x)\cdot e_n}{x_n}\right)$.
    The minimality of $E$ yields that the first-order term in the right-hand side above must vanish, hence the desired result.
    
\end{proof}
We call a set $E$ that satisfies \eqref{eq:stationarity} for every $X\in C^1_c(A;\RR^n)$ such that $X\cdot e_n=0$ on $\{x_n=0\}$ a \textit{stationary set for $\Per_{a}$ in $A$}.

We introduce the $a$-density of a set $E$ at $x$ and at scale $r$:
\begin{equation}
    \Theta_a(E;x,r) \coloneqq\frac{\Per_a(E;B_r(x))}{r^{n-1+a}} = \frac{1}{r^{n-1+a}}\int_{\de^*E\cap B_r(x)}(x_n^+)^a\dif\Hc^{n-1}.
\end{equation}

\begin{proposition}[Monotonicity formula at the boundary]\label{prop:monotonicity}
    Let $E\subset\RR^n_+$ be a stationary set for $\Per_{a}$ in $B_1$. For all $x_0\in\{x_n=0\}$ and every $0<r\le s<1-|x|$, it holds
    \begin{equation}
    \Theta_a(E;x_0,s) - \Theta_a(E;x_0,r) =
    C\int_{B_s(x_0)\setminus B_r(x_0)} \frac{x_n^a}{|x-x_0|^{n-1+a}} \left( 1-\frac{|(x-x_0)^T|^2}{|x-x_0|^2} \right) \dif \mu_E(x)
    \end{equation}
    where $x^T = x-(x\cdot \nu_E(x))\nu_E(x)$ is the projection of $x$ onto the approximate tangent space to $\de^*E$ at $x$ and $C$ is a universal constant.
\end{proposition}
\begin{proof}
We assume without loss of generality that $x_0=0$
We introduce the truncated fundamental solution
\begin{equation}
	h(x) = \frac{1}{(n-1)(n+a-3)}
	\begin{cases}
		\frac{n-1+a}{2}-\frac{n+a-3}{2}|x|^2\qquad&\mbox{if $|x|<1$}\\
		\frac{1}{|x|^{n+a-3}}&\mbox{if $|x|\ge1$}
	\end{cases}
\end{equation}
and, for $0<r<s$ fixed:
\begin{equation}
	g_{r,s}(x) = r^{3-n-a}h\left(\frac{x}{r}\right)-s^{3-n-a}h\left(\frac{x}{s}\right).
\end{equation}
Notice that
\begin{multline}\label{eq:divSg}
	\dive_E\nabla g_{r,s}=\frac{\chi_{B_s}}{s^{n-1+a}}-\frac{\chi_{B_r}}{r^{n-1+a}}\\
    +\frac{1}{(n-1)|x|^{n-1+a}}\left((n-1+a)\frac{|x^T|^2}{|x|^2}-(n-1)\right)\chi_{B_s\setminus B_r}.
\end{multline}
Since the vector field $\nabla g$ is tangent to $\{x_n=0\}$, by \Cref{prop:first_var} we have
\begin{align}\label{eq:first_var}
	0 &= \int x_n^a\left(\dive_E(\nabla g_{r,s})+a\frac{e_n\cdot\nabla g_{r,s}}{x_n}\right)\dif \mu_E.
\end{align}
Now, for the first summand in the integrand of \eqref{eq:first_var}, we have
\begin{multline}\label{eq:first}
	\int x_n^a\dive_E(\nabla g_{r,s})\dif \mu_E
    = \frac{1}{s^{n-1+a}}\int_{B_s}x_n^a\dif \mu_E
    -\frac{1}{r^{n-1+a}}\int_{B_r}x_n^a\dif \mu_E \\
  + \frac{1}{n-1}\int_{B_s\setminus B_r}\frac{x_n^a}{|x|^{n-1+a}}\left((n-1+a)\frac{|x^T|^2}{|x|^2}-n+1\right)\dif\mu_E.
\end{multline}
On the other hand, for the second summand in \eqref{eq:first_var}, we remark that
\begin{equation}
	\nabla g_{r,s}(x) = \frac{1}{n-1}\left(-\frac{\chi_{B_r}}{r^{n-1+a}}+\frac{\chi_{B_s}}{s^{n-1+a}} - \frac{\chi_{B_s\setminus B_r}}{|x|^{n-1+a}}\right)x
\end{equation}
and that $x\cdot \nabla [x_n^a] = a x_n^a$, thus
\begin{equation}\label{eq:second}
	a x_n^{a-1}e_n\cdot\nabla g_{r,s}(x) = \frac{a}{n-1}\left(-\frac{\chi_{B_r}}{r^{n-1+a}}+\frac{\chi_{B_s}}{s^{n-1+a}} - \frac{\chi_{B_s\setminus B_r}}{|x|^{n-1+a}}\right) x_n^a.
\end{equation}
Using \eqref{eq:first} and \eqref{eq:second} in \eqref{eq:first_var}, we obtain
\begin{multline}
	\left(1+\frac{a}{n-1}\right)\left(\frac{1}{s^{n-1+a}}\int_{B_s}x_n^a\dif \mu_E - \frac{1}{r^{n-1+a}}\int_{B_r}x_n^a\dif \mu_E \right) \\
    = \frac{n-1+a}{n-1}\int_{B_s\setminus B_r} \frac{x_n^a}{|x|^{n-1+a}}\left(1-\frac{|x^T|^2}{|x|^2}\right)\dif\mu_E
\end{multline}
which is the desired conclusion.

\end{proof}

As corollaries, we get the two following standard results:

\begin{corollary}\label{cor:densityusc}
    If $E\subset \RR^n_+$ is a stationary point for $\Per_{a}$, then for all $x\in\{x_n =0\}$ there exists the density 
        \begin{equation}
            \Theta_a(E;x) := \lim_{r\searrow0} \Theta_a(E;x,r),
        \end{equation}
    and it is upper-semicontinuous on $\{x_n=0\}$.
\end{corollary}
\begin{proof}
    See, for instance, \cite[Proposition 2.2]{DeLellis_2018}.
\end{proof}

\begin{corollary}\label{cor:blowupsarecones}
Let $E$ be a minimizer for $\Per_{a}$ (resp. a stationary point) and $x_0\in\de E \cap \{x_n=0\}$.
Then there exists a blow-up sequence $E_{x_0,r_i}\coloneqq \frac{E-x_0}{r_i}$ and a minimizing (resp. stationary) cone $K$ such that
\begin{equation}
    E_{x_0,r_i} \xrightarrow[]{L^1_{loc}} K,\quad \mbox{and}\quad \Hc^{n-1}_{a}\llcorner\de^*E_{x_0,r_i} \weakstar \Hc^{n-1}_a\llcorner K.
\end{equation}
\end{corollary}
\begin{proof}
    The existence of a limit $K$ is given by \Cref{prop:compactnessqmin}.
    By \Cref{prop:monotonicity}, it holds
    \begin{equation}
        \Theta_a(K;0,s)\equiv \Theta_a(E;x_0)
    \end{equation}
    for every $s>0$.
    By \Cref{prop:monotonicity} again, we have
    \begin{equation}
        \int_{B_s\setminus B_r} \frac{x_n^a}{|x|^{n-1+a}} \left( 1-\frac{|x^T|^2}{|x|^2} \right) \dif \mu_K(x)=0
    \end{equation}
    for every $0<r<s$.
    Therefore, for $\Hc^{n-1}$-almost every $x\in \de^*K\cap\{x_n>0\}$, it holds $x\cdot \nu_K(x)=0$.
    Since $\de^*K$ is rectifiable, we also have that $x\cdot\nu_K(x)=0$ for $\Hc^{n-1}$-almost every $x\in\de^*K\cap\{x_n=0\}$.
    The two fact above yield that $K$ is, indeed, a cone (see, for instance, \cite[Proposition 28.8]{maggi12}).
    
\end{proof}

\subsection{Viscosity properties}\label{subsec:viscosity_prop}
As explained in the introduction, the proof of the improvement of flatness exploits the fact that an almost minimizer is, in a very weak sense, a viscosity solution to an elliptic equation with Neumann boundary condition.
On one hand, the Neumann boundary condition holds true in a fairly standard viscosity sense, as we show in \Cref{prop:boundary_max} below.
On the other hand, the condition
\begin{equation}\label{eq:target_viscosity_eq}
    H_E+a\frac{\nu_E\cdot\nabla d_\Omega}{d_\Omega}=0
\end{equation}
would hold true for a minimizer, but is false, in general, for almost minimizers.
However, we will use some \textit{interior maximum principles} that correspond, in a certain sense, to a viscosity formulation of \eqref{eq:target_viscosity_eq}.
Specifically, we will do so in \Cref{prop:weak_viscosity}, in the proof of \Cref{prop:harnack_boundary} and in the proof of \Cref{thm:improvement_flatness}.
Those three proofs rely on the same technical geometric construction, which is carried out in \Cref{lemma:technical_geom} below.

We start by introducing some terminology.
We say that \textit{a set $F$ touches another set $E$ from outside at $z$ in a neighborhood $B_r(z)$} if $E\cap B_r(z)\subset F$ and $z\in\de E\cap\de F$.
The usual situation we will encounter in the rest of this paper is that an almost minimizer $E$ is touched from outside by some smooth set $F$.
The technical assumption that an almost minimizer of $\Per_{\dista{\Omega}}$ is a subset of $\overline{\Omega}$ plays a non-trivial role here, in that we allow any set $F$ to touch $E$ even at points on $\de\Omega$ regardless of the behavior of $F$ outside $\Omega$.
\begin{center}
\begin{figure}[ht]
	\begin{tikzpicture}[use Hobby shortcut, thick,even odd rule]
	\begin{scope}
			\clip(1,-0.4) rectangle (8,3.2);
	\begin{scope}
		\clip 	(4.2,0) circle (3);
			\draw[fill=black!5, fill opacity =1, draw opacity =0] (1,-0.2) .. (3,0.1) .. (4,0) .. (5,0.1) .. (7,-0.1) .. (7.4,0.1) -- (7.4,3.2) -- (1,3.2) -- (1,0.-0.2);

	\begin{scope}
		\clip(1,-0.4) rectangle (8,3.2);

			  \draw[closed, fill= blue!20, fill opacity =1] (1,-2) .. (4,0) ..  (5,3.5) .. (6,6) .. (6,7) .. (6,8) .. (0,10);
	\end{scope}
	\begin{scope}
   		\clip (1,-0.2) .. (3,0.1) .. (4,0) .. (5,0.1) .. (7,-0.1) .. (7.4,0.1) -- (7.4,3.2) -- (1,3.2) -- (1,0.-0.2);

         	\draw[closed, fill= red!20, fill opacity =1] (4.2,-6) .. (4,0) .. (3.45,3) .. (3,3.5) .. (2.7,6) .. (2.5,7) .. (2,8) .. (0,10);
            \draw[dashed] (4.2,0) circle (3);
	\end{scope}
	\end{scope}
	\end{scope}
		\draw (1,-0.2) .. (3,0.1) .. (4,0) .. (5,0.1) .. (7,-0.1) .. (7.4,0.1);

		\node at (5.7,1) {$\Omega\cap B_r(x_0)$};
		\node at (3,1.5) {$E$};
		\node at (4.2,2.6) {$F$};
		\node at (4.2,-0.3) {$x_0$};
	\end{tikzpicture}
    \caption{A set $F$ touching $E$ from outside at $x_0\in\de\Omega$. Notice that the assumption $E\subset\overline{\Omega}$ allows, in the smooth setting, that the tangent spaces at $x_0$ to $\de E$ and $\de F$ to differ when the touching point is at $\de\Omega$}
\end{figure}
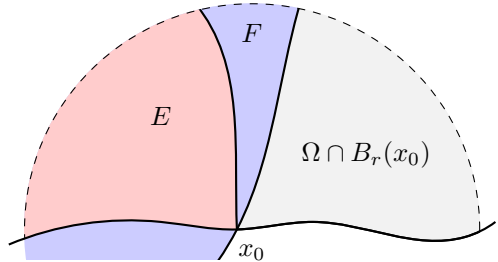
\end{center}
\begin{proposition}[Viscosity property on $\de\Omega$]\label{prop:boundary_max}
    Let $\Omega$ be a set with $C^1$ boundary and let $E\subset\overline{\Omega}$ be a $(\vartheta,\beta)$-minimizer of $\Per_\dista{\Omega}$ in $B_1$. If a smooth set $F$ touches $E$ from outside at $x_0\in\de\Omega\cap B_1$ in a neighborhood $B_r(x_0)$, then
    \begin{equation}\label{eq:viscosity_boundary}
            \nu_F(x_0)\cdot \nu_\Omega(x_0) \ge 0
    \end{equation}
    where $\nu_F$ and $\nu_\Omega$ denote the outer unit normals to $\de F$ and $\de\Omega$, respectively.
\end{proposition}
\begin{proof}
    By contradiction, assume (up to a change of coordinates) that there is a smooth set $F$ that touches $E$ from outside at $0\in\de\Omega$ in a neighborhood $B_r$, that $\nu_\Omega(0)=-e_n$ and that 
    \begin{equation}
        \nu_F(x_0)\eqqcolon \nu_0 = b e_n + \sqrt{1-b^2}\,e_1,
    \end{equation}
    for some $b\in(0,1)$.
    
    By \Cref{cor:blowupqm}, we can find a sequence $r_j\searrow 0$ and a set $E_0$ such that
    \begin{equation}
    \frac{E}{r_j} \longrightarrow E_0, \quad\mbox{and}\quad \frac{F}{r_j} \longrightarrow F_0 \coloneqq \{x\in\RR^n\colon\nu_0\cdot x < 0 \}
    \end{equation}
    locally in $L^1$.
    Furthermore, $E_0$ is a minimizer of $\Per_{a}$ and $F_0$ touches $E_0$ from outside at $0$ in any neighborhood $B_r(0)$.

    Applying \Cref{cor:blowupqm} again, we find another sequence $s_j\searrow0$ and $E_{00}$ such that
    \begin{equation}
        \frac{E_0}{s_j} \longrightarrow E_{00}
    \end{equation}
    as above, $E_{00}$ is a minimizer of $\Per_{a}$ in any ball and $F_0$ touches $E_{00}$ from outside at $0$.
    Moreover, by \Cref{cor:blowupsarecones}, $E_{00}$ is a cone. 
    Let $\Gamma_0\coloneqq\partial F_0\cap\{x_n=0\}=\{x_n=x_1=0\}$.

    \textbf{We first assume that $\partial E_{00}\cap\Gamma_0 = \{0\}$.}
    Notice that, since $E_{00}\subset\RR^n_+\cap\{x\cdot \nu_0\le0\}$, it holds $E_{00}\subset\{x_1\le0\mbox{ and }x_n\ge0\}$.
    Furthermore, the fact that $\overline E_{00}$ is a cone implies that there exists $q>0$ such that
    \begin{equation}\label{eq:esistegamma}
        \overline {E_{00}}\subset \{x_1\le-q|x|\},
    \end{equation}
    because otherwise it would be $\Gamma_0\cap\partial E_{00}\cap\sphere[n-1]\neq\emptyset$.

    Consider the vector field $X(x) = f(x_1)e_1$, where $f$ is a smooth non-decreasing function such that $f\equiv0$ in $(-\infty,-1]$ and $f'\equiv1$ in $[-1/2,0]$.
    Since $E_{00}$ is a cone, for $\mu_{E_{00}}$-a.e. $x$ it holds $\frac{x}{|x|}\cdot \nu_{E_{00}}(x)=0$,
    thus \eqref{eq:esistegamma} gives 
    \begin{math}
        |\Pi_xe_1|^2\ge \left|e_1\cdot\frac{x}{|x|}\right|^2\ge q^2,
    \end{math}
    at $\mu_{E_{00}}$-a.e. $x$, where $\Pi_x$ denotes the orthogonal projection onto $T_x\de^*E_{00}$.
    Therefore
    \begin{equation}
        \dive_{E_{00}} X = f'(x_1)|\Pi_xe_1|^2 \ge q^2f'(x_1)\ge0
    \end{equation}
    at $\mu_{E_{00}}$-a.e. $x$.
    
    Now, since
    \begin{equation}
        \supp X\cap \overline {E_{00}}\subset\{-1\le x_1\le-q|x|\}\subset B_{2/q},
    \end{equation}
    we may use $X$ as a test vector field in \eqref{eq:first_var} and obtain that
    \begin{align}
        0 &= \int x_n^a\left(\dive_{E_{00}} X + a\frac{X\cdot e_n}{x_n}\right)\dif\mu_{E_{00}}\\
        &\ge q^2\int x_n^a f'(x_1)\dif\mu_{E_{00}}\\
        &\ge q^2\int_{B_{1/2}}x_n^a\dif\mu_{E_{00}}.
    \end{align}
    This yields $x_n=0$ for $\mu_{E_{00}}$-a.e. $x\in B_{1/2}$, thus either $E_{00}\cap B_{1/2}=B_{1/2}$ (contradicting the fact that $F_0$ touches from outside) or $E_{00}\cap B_{1/2}=\emptyset$ (contradicting $0\in\de E$).

    \textbf{If, on the other hand, $\partial E_{00}\cap\Gamma_0$ contains more than one point, }then the fact that $E_{00}$ is a cone implies that, up to a rotation of $\{x_n=0\}$, $\partial E_{00}\cap\Gamma_0\supset\{t e_2:t\ge 0\}$.
    
    We let $E_{000}$ be a further blow-up of $E_{00}$ at $e_2$: namely, $E_{000} = \lim\frac{E_{00}-e_2}{t_j}$ for some $t_j\searrow0$. Then $E_{000}$ is a stationary cone which is invariant under translations by $te_2$ for all $t\in\RR$. In particular, $E_{000}=E_{000}'\times\RR$ for some stationary cone $E_{000}'\subset\RR^{n-1}$ that is included in an acute wedge of the form $\{0\le x_{n-1}\le -q x_1\}$ for some $q\in(0,+\infty)$.
    At this point, either $\partial E_{000}'\cap\{x_1=x_{n-1}=0\}=\{0\}$, which gives a contradiction by the discussion above, or we can apply the dimension-reduction argument until that is the case, which will happen in at most $(n-1)$ iterations.
    
\end{proof}

We now state and prove the technical geometric lemma that will allow us, later, to prove the interior maximum principles.

\begin{lemma}[Technical geometric lemma]\label{lemma:technical_geom}
    Let $\phi:\RR^{n-1}\to\RR$ be a $C^2$ function with $||\nabla \phi||_\infty\le 1/4$ and assume that $\{x\colon x_1\le \phi(x'')\}$ touches $E$ from outside at $z$ in a neighborhood $B_r(z)$, with $r\le1$.
    Given $0<b\le \frac{1}{4}$, let $p(x'') = b\left(\frac{r^2}{16}-|x''-z''|^2\right)$ and let
    \begin{equation}
        F\coloneqq E\setminus\left(B_{r}(z)\cap\left\{x\colon x_1\le \phi(x'')-p(x'')\right\}\right).
    \end{equation}
    Then:
    \begin{enumerate}
        \item $E\setminus F\Subset B_{r}(z)$;
        \item \label{item:smallballinside} there exists $c>0$ universal such that $B_{cbr^2}(z)\subset\RR^n\setminus F$;
        \item if $z\in\mathrm{Int}\,\Omega$ and $E\subset\Omega$ is a $(\vartheta,\beta)$-minimizer of $\Per_{\dista{\Omega}}$ in $B_r(z)$, then letting $G(x)\coloneqq x_1-\phi(x'')+p(x'')$ it holds
        \begin{equation}\label{eq:technical_geom_integral}
            C\vartheta(d_\Omega(z)+r)^ar^{n-1+\beta}\ge\int_{E\setminus F}\dive\left(\dista{\Omega}\frac{\nabla G}{|\nabla G|}\right)\dif\Lc^n,
        \end{equation}
        where $C$ is a universal constant.
    \end{enumerate}    
\end{lemma}

\begin{center}
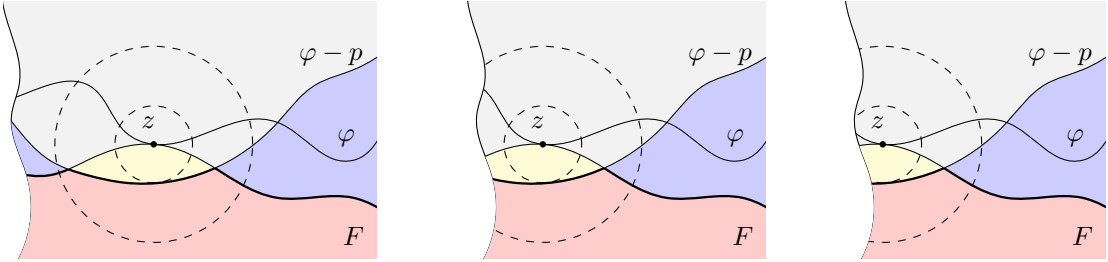
\begin{figure}[ht]
	
\begin{minipage}{\textwidth}
\begin{center}
	\begin{tikzpicture}[use Hobby shortcut]
		\begin{scope}
			\clip (-0.61,-1.3) rectangle (4.3,2.1);
			
			\draw[fill = black!5] (-0.6,-1.6) .. (-0.3,-0.2) .. (-0.5,0.6) .. (-0.4,1) .. (-0.6,2) .. (-0.5,3) -- (10,3) -- (10,-3) -- (-0.6,-3);
			
		\begin{scope}
				\clip[shift={(0,0)}] (-0.6,-1.6) .. (-0.3,-0.2) .. (-0.5,0.6) .. (-0.4,1) .. (-0.6,2) .. (-0.5,3) -- (10,3) -- (10,-3) -- (-0.6,-3);

			\draw[fill = yellow!20] (-15,-6) .. (-7,-1) .. (-5,2) .. (-4,-1) .. (-3,-1) .. (-2,-0.1) .. (-1,0) .. (0,-0.2) .. (0.5,0) .. (1.37,0.21) .. (2,0) .. (3,-0.5) .. (4,-0.5) .. (5,-1) .. (6,-1) .. (7,-9);	
			
			\draw[fill=blue!20] (-15,-6) .. (-7,1) .. (-5,7) .. (-4,6) .. (-3,1.5) .. (-2,1.2) .. (-1,1) .. (0,0) .. (0.5,-0.2) .. (1,-0.3) .. (3,0.5) .. (3.5,1) .. (4,1.2) .. (5,1.8) .. (6,2) .. (7,-9);
			\draw (-15,-6) .. (-7,1) .. (-5,7) .. (-4,6) .. (-3,1.5) .. (-2,1) .. (-1,0.7) .. (0,1) .. (0.5,1) .. (1,0.37) .. (1.37,0.21) .. (3,0.5) .. (4,0) .. (5,1) .. (5.5,0.8) .. (6,0.9) .. (7,-9);

			\begin{scope}
			\clip(-15,-6) .. (-7,1) .. (-5,7) .. (-4,6) .. (-3,1.5) .. (-2,1.2) .. (-1,1) .. (0,0) .. (0.5,-0.2) .. (1,-0.3) .. (3,0.5) .. (3.5,1) .. (4,1.2) .. (5,1.8) .. (6,2) .. (7,-9);

			 \draw[fill = red!20,line width=0.3mm] (-15,-6) .. (-7,-1) .. (-5,2) .. (-4,-1) .. (-3,-1) .. (-2,-0.1) .. (-1,0) .. (0,-0.2) .. (0.5,0) .. (1.37,0.21) .. (2,0) .. (3,-0.5) .. (4,-0.5) .. (5,-1) .. (6,-1) .. (7,-9);

			\end{scope}

                \begin{scope}
			\clip(-15,-6) .. (-7,-1) .. (-5,2) .. (-4,-1) .. (-3,-1) .. (-2,-0.1) .. (-1,0) .. (0,-0.2) .. (0.5,0) .. (1.37,0.21) .. (2,0) .. (3,-0.5) .. (4,-0.5) .. (5,-1) .. (6,-1) .. (7,-9);

			 \draw[line width=0.3mm] (-15,-6) .. (-7,1) .. (-5,7) .. (-4,6) .. (-3,1.5) .. (-2,1.2) .. (-1,1) .. (0,0) .. (0.5,-0.2) .. (1,-0.3) .. (3,0.5) .. (3.5,1) .. (4,1.2) .. (5,1.8) .. (6,2) .. (7,-9);

			\end{scope}

			\end{scope}

			\draw[dashed] (1.37,0.21) circle (1.3);
			\draw[dashed] (1.37,0.21) circle (0.51);
			\filldraw (1.37,0.21) circle (1pt);
			
			\node at (4,-1) {$F$};			
			\node at (1.3,0.5) {$z$};
			\node at (3.9,0.3) {$\phi$};
			\node at (3.7,1.4) {$\phi-p$};

		\end{scope}
	\end{tikzpicture}
\hfill
	\begin{tikzpicture}[use Hobby shortcut,scale=1]
		\begin{scope}
			\clip (0.2,-1.3) rectangle (4.3,2.1);
			
			\draw[fill = black!5, shift={(1,0)}] (-0.6,-1.6) .. (-0.3,-0.2) .. (-0.5,0.6) .. (-0.4,1) .. (-0.6,2) .. (-0.5,3) -- (10,3) -- (10,-3) -- (-0.6,-3);
			
		\begin{scope}
				\clip[shift={(1,0)}] (-0.6,-1.6) .. (-0.3,-0.2) .. (-0.5,0.6) .. (-0.4,1) .. (-0.6,2) .. (-0.5,3) -- (10,3) -- (10,-3) -- (-0.6,-3);

			\draw[fill = yellow!20] (-15,-6) .. (-7,-1) .. (-5,2) .. (-4,-1) .. (-3,-1) .. (-2,-0.1) .. (-1,0) .. (0,-0.2) .. (0.5,0) .. (1.37,0.21) .. (2,0) .. (3,-0.5) .. (4,-0.5) .. (5,-1) .. (6,-1) .. (7,-9);	
			
			\draw[fill=blue!20] (-15,-6) .. (-7,1) .. (-5,7) .. (-4,6) .. (-3,1.5) .. (-2,1.2) .. (-1,1) .. (0,0) .. (0.5,-0.2) .. (1,-0.3) .. (3,0.5) .. (3.5,1) .. (4,1.2) .. (5,1.8) .. (6,2) .. (7,-9);

			\draw (-15,-6) .. (-7,1) .. (-5,7) .. (-4,6) .. (-3,1.5) .. (-2,1) .. (-1,0.7) .. (0,1) .. (0.5,1) .. (1,0.37) .. (1.37,0.21) .. (3,0.5) .. (4,0) .. (5,1) .. (5.5,0.8) .. (6,0.9) .. (7,-9);

			\begin{scope}
			\clip(-15,-6) .. (-7,1) .. (-5,7) .. (-4,6) .. (-3,1.5) .. (-2,1.2) .. (-1,1) .. (0,0) .. (0.5,-0.2) .. (1,-0.3) .. (3,0.5) .. (3.5,1) .. (4,1.2) .. (5,1.8) .. (6,2) .. (7,-9);

                    \draw[fill = red!20,line width=0.3mm] (-15,-6) .. (-7,-1) .. (-5,2) .. (-4,-1) .. (-3,-1) .. (-2,-0.1) .. (-1,0) .. (0,-0.2) .. (0.5,0) .. (1.37,0.21) .. (2,0) .. (3,-0.5) .. (4,-0.5) .. (5,-1) .. (6,-1) .. (7,-9);

			\end{scope}

                \begin{scope}
			\clip(-15,-6) .. (-7,-1) .. (-5,2) .. (-4,-1) .. (-3,-1) .. (-2,-0.1) .. (-1,0) .. (0,-0.2) .. (0.5,0) .. (1.37,0.21) .. (2,0) .. (3,-0.5) .. (4,-0.5) .. (5,-1) .. (6,-1) .. (7,-9);

			 \draw[line width=0.3mm] (-15,-6) .. (-7,1) .. (-5,7) .. (-4,6) .. (-3,1.5) .. (-2,1.2) .. (-1,1) .. (0,0) .. (0.5,-0.2) .. (1,-0.3) .. (3,0.5) .. (3.5,1) .. (4,1.2) .. (5,1.8) .. (6,2) .. (7,-9);

			\end{scope}

			\draw[dashed] (1.37,0.21) circle (1.3);
			\draw[dashed] (1.37,0.21) circle (0.51);

			\end{scope}

			\filldraw (1.37,0.21) circle (1pt);
			
			\node at (4,-1) {$F$};			
			\node at (1.3,0.5) {$z$};
			\node at (3.9,0.3) {$\phi$};
			\node at (3.7,1.4) {$\phi-p$};

		\end{scope}
	\end{tikzpicture}
\hfill
	\begin{tikzpicture}[use Hobby shortcut]
		\begin{scope}
			\clip (0.85,-1.3) rectangle (4.3,2.1);
			
			\draw[fill = black!5,shift={(1.5,0)}] (-0.6,-1.6) .. (-0.3,-0.2) .. (-0.5,0.6) .. (-0.4,1) .. (-0.6,2) .. (-0.5,3) -- (10,3) -- (10,-3) -- (-0.6,-3);
			
		\begin{scope}
				\clip[shift={(1.5,0)}] (-0.6,-1.6) .. (-0.3,-0.2) .. (-0.5,0.6) .. (-0.4,1) .. (-0.6,2) .. (-0.5,3) -- (10,3) -- (10,-3) -- (-0.6,-3);

			\draw[fill = yellow!20] (-15,-6) .. (-7,-1) .. (-5,2) .. (-4,-1) .. (-3,-1) .. (-2,-0.1) .. (-1,0) .. (0,-0.2) .. (0.5,0) .. (1.37,0.21) .. (2,0) .. (3,-0.5) .. (4,-0.5) .. (5,-1) .. (6,-1) .. (7,-9);	
			
			\draw[fill=blue!20] (-15,-6) .. (-7,1) .. (-5,7) .. (-4,6) .. (-3,1.5) .. (-2,1.2) .. (-1,1) .. (0,0) .. (0.5,-0.2) .. (1,-0.3) .. (3,0.5) .. (3.5,1) .. (4,1.2) .. (5,1.8) .. (6,2) .. (7,-9);
			\draw (-15,-6) .. (-7,1) .. (-5,7) .. (-4,6) .. (-3,1.5) .. (-2,1) .. (-1,0.7) .. (0,1) .. (0.5,1) .. (1,0.37) .. (1.37,0.21) .. (3,0.5) .. (4,0) .. (5,1) .. (5.5,0.8) .. (6,0.9) .. (7,-9);

			\begin{scope}
			\clip(-15,-6) .. (-7,1) .. (-5,7) .. (-4,6) .. (-3,1.5) .. (-2,1.2) .. (-1,1) .. (0,0) .. (0.5,-0.2) .. (1,-0.3) .. (3,0.5) .. (3.5,1) .. (4,1.2) .. (5,1.8) .. (6,2) .. (7,-9);

			 \draw[fill = red!20,line width=0.3mm] (-15,-6) .. (-7,-1) .. (-5,2) .. (-4,-1) .. (-3,-1) .. (-2,-0.1) .. (-1,0) .. (0,-0.2) .. (0.5,0) .. (1.37,0.21) .. (2,0) .. (3,-0.5) .. (4,-0.5) .. (5,-1) .. (6,-1) .. (7,-9);

			\end{scope}

                \begin{scope}
			\clip(-15,-6) .. (-7,-1) .. (-5,2) .. (-4,-1) .. (-3,-1) .. (-2,-0.1) .. (-1,0) .. (0,-0.2) .. (0.5,0) .. (1.37,0.21) .. (2,0) .. (3,-0.5) .. (4,-0.5) .. (5,-1) .. (6,-1) .. (7,-9);

			 \draw[line width=0.3mm] (-15,-6) .. (-7,1) .. (-5,7) .. (-4,6) .. (-3,1.5) .. (-2,1.2) .. (-1,1) .. (0,0) .. (0.5,-0.2) .. (1,-0.3) .. (3,0.5) .. (3.5,1) .. (4,1.2) .. (5,1.8) .. (6,2) .. (7,-9);

			\end{scope}
			
			\draw[dashed] (1.37,0.21) circle (1.3);
			\draw[dashed] (1.37,0.21) circle (0.51);

			\end{scope}

			\filldraw (1.37,0.21) circle (1pt);
			
			\node at (4,-1) {$F$};			
			\node at (1.3,0.5) {$z$};
			\node at (3.9,0.3) {$\phi$};
			\node at (3.7,1.4) {$\phi-p$};

		\end{scope}
	\end{tikzpicture}
\end{center}
\end{minipage}
\caption{In red $F$, in blue $\{x_1 < \phi - p\}$ and in yellow $E\setminus F$. 
The two dashed balls represent, respectively, $B_r(z)$ and $B_{cbr^2}(z)$.
Three possible configurations based on the relation between $d_\Omega(z)$ and $r$ are represented.}
\end{figure}

\end{center}

\begin{proof}
    Without loss of generality, we may assume $z=0$.
    \begin{enumerate}
        \item For every $x\in E\setminus F$, it holds $x\in E\cap B_r\subset\{x_1\le \phi(x'')\}$ and $x\in\{x_1\ge \phi-p\}$. 
        Therefore $p(x'')\ge0$, hence $|x''|\le \frac{r}{4}$.
        Moreover, we have $|\phi(x'')|\le \frac{r}{4}$ and $p(x'')\le b\frac{r^2}{16}\le \frac{r}{4}$ due to $|\nabla \phi|,b\le\frac{1}{4}$ and $r\le1$.
        Therefore
        \begin{equation}
            |x_1|\le |\phi(x'')|+|p(x'')|\le\frac{r}{2},
        \end{equation}
        wich together with $|x''|\le\frac{r}{4}$ proves the first item.
        \item For every $x\in B_{cbr^2}$, it holds
        \begin{equation}
            x_1 \ge -cbr^2, \quad \phi(x'') \le \frac{1}{4}cbr^2,\quad \mbox{and}\quad p(x'') \ge br^2\frac{1-c^2}{16},
        \end{equation}
        where we used $|\nabla \phi|,b\le\frac{1}{4}$ once again.
        Therefore taking $c$ smaller than, say, $\frac{1}{64}$ it holds
        \begin{equation}
            x_1-\phi(x'')+p(x'')\ge br^2\left(-c-\frac{c}{4}+\frac{1-c^2}{16}\right) > 0 
        \end{equation}
        hence $x\notin F$, as claimed.
        \item Let $T(x)\coloneqq \frac{\nabla G(x)}{|\nabla G(x)|}$.
        Notice that, due to $|\nabla(\phi-p)|\le\frac{1}{2}$, $T$ is well defined. Moreover, $|T|=1$ everywhere and $T|_{\{x_1=\phi-p\}\cap E}$ coincides with the outer unit normal to $\de F$. Therefore, by the divergence theorem:
        \begin{equation}
            \Per_{\dista{\Omega}}(E;B_r)-\Per_{\dista{\Omega}}(F;B_r)\ge\int_{E\setminus F}\dive(\dista{\Omega} T)\dif\Lc^n.
        \end{equation}
        The inequality then follows from \eqref{eq:percompetitorbound}.
    \end{enumerate}
\end{proof}

\section{Epsilon-regularity}\label{sec:eps_regularity}

The goal of this Subsection is proving \Cref{thm:improvement_flatness}. 
As explained in \Cref{sec:intro}, the proof of \Cref{thm:improvement_flatness} is based on some Harnack-type estimates (Propositions \ref{prop:interior_harnack} and \ref{prop:harnack_boundary}) and on their iteration (\Cref{cor:full_decay}), following the scheme developed in \cite{Savin_2007}.

This section is structured as follows.
\begin{itemize}
    \item In \Cref{subsec:interior_harnack}, we prove a Harnack-type inequality away from $\de\Omega$ (\Cref{prop:interior_harnack}), using the results in \cite{DeSilva_Savin_2021}. We then iterate that result to obtain a $C^{0,\sigma}$ decay of oscillations away from $\de\Omega$ up to a scale that depends on the initial flatness of the set (\Cref{cor:osc_decay_interior}).
    \item Exploiting the results from \Cref{subsec:interior_harnack}, in \Cref{subsec:boundary_harnack} we prove a Harnack inequality near $\de\Omega$ (\Cref{prop:harnack_boundary}), with the techniques developed in \cite{DeSilva_2011}. As above, we then iterate it to obtain a $C^{0,\sigma}$ decay of oscillations up to $\de\Omega$ (\Cref{cor:full_decay}).
    \item In \Cref{sec:IOF}, we use \Cref{cor:full_decay} to finally prove \Cref{thm:improvement_flatness}.
\end{itemize}
 
Before proceeding, we fix some notation.
Let $\Omega\subset\RR^n$ be a given open set. For a unit vector $\nu\in\RR^n$, a set of finite $\dista{\Omega}$-perimeter $E$ and $U\subset\RR^n$, we let
\begin{equation}
    \osc_{\nu}(\de E; U) \coloneqq \frac{1}{2}\sup\{|(x-y)\cdot\nu|\colon x,y\in\de E\cap\Omega\cap U\}
\end{equation}
so that, if $\osc_{\nu}(\de E;U)=h$ then there exists $c\in\RR$ such that $\de E\cap\Omega\cap U\subset\{x\colon |x\cdot\nu-c|\le h\}$.
Notice that this entails no restriction on $\de E\cap\de\Omega\cap U$.
We record the following technical result concerning sets with small enough oscillations: 
\begin{lemma}[Infiltration Lemma]\label{lemma:tech_oscill}
	There exist dimensional constant $C,\delta>0$ with the following property.
	Let $\Omega\subset\RR^n$, $\Omega = \{(x',x_n) : x_n > g(x') \}$, where $g$ is $\frac{1}{CR}$-Lipschitz, $g(0)=0$, $x_0\in B_{R}\cap\de E\cap \overline{\Omega}$, $R\ge 1$, and $E$ a $(\vartheta,\beta)$-minimizer of $\Per_\dista{\Omega}$ in $B_{2R}(x_0)$, with $\vartheta\le 1$. 
	If there exists $\nu\in\sphere$ such that
	\begin{equation}
		\de E \cap\Omega\cap B_{R}(x_0) \subset \left\{ |(x-x_0)\cdot\nu| \le \delta R\right\},
	\end{equation}
	then 
	\begin{equation}
		\left\{ (x-x_0)\cdot \nu \le -\delta R \right\}\cap B_R(x_0)\cap \Omega
        \,\subset\,
        E \cap B_R(x_0)
        \,\subset\,
        \left\{ (x-x_0)\cdot \nu \le \delta R \right\}.
	\end{equation}
	
	\end{lemma}
\begin{proof}
    This Lemma is the analogous of \cite[Lemma 22.10]{maggi12}.

	We argue by contradiction, and we assume without loss of generality $R=1$.
	By contradiction, for all $\delta>0$, there exist a $(\vartheta,\beta)$-minimizer $E\subset\overline{\Omega}$ and $x,y \in B_1(x_0)$ such that 
	\begin{equation}
		(x-x_0)\cdot\nu < -\delta \qquad \mbox{and}\qquad (y-x_0)\cdot\nu > \delta,
	\end{equation}
	and either $x,y\in E$ or $x,y\in B_1(x_0) \setminus E$.
	Without loss of generality, we us assume the second case holds (we can simply exchange the conditions by taking $B_1(x_0)\setminus E$ instead of $E$).
	
	From \cite[Proposition 7.5]{maggi12}, it follows that, since $\de E = \supp \mu_E \subset \{|(x-x_0)\cdot\nu|\le \delta\}$, then $\chi_E$ is constant outside that strip.
	In particular, the contradiction assumption yields $E\subset \{|(x-x_0)\cdot\nu|\le \delta\}$.
	We now show that this condition implies that we can define a competitor for $E$ that falsifies the $(\vartheta,\beta)$-minimality assumption.
	
	For $r\in(1/2,1)$ such that $\Hc^{n-1}(\de B_r(x_0) \cap \de^* E)=0$, we define $F_r := (E\setminus \overline{B_r(x_0)} ) \cap B_1(x_0)$. 
	By $(\vartheta,\beta)$-minimality, for $s>r$ we get
	\begin{align}
		\Per_{\dista{\Omega}}(E;B_s(x_0)) 
		&\le (1+\vartheta s^\beta)\Per_{\dista{\Omega}}(F_r;B_s(x_0)) \\
		&\le (1+\vartheta s^\beta) \Big( \Per_{\dista{\Omega}}(E;B_s(x_0)\setminus\overline{B_r(x_0)}) + \Per_{\dista{\Omega}}(\de B_r(x_0) \cap E) \Big).
	\end{align} 
	Moreover, since $\Hc^{n-1}(\de B_r(x_0) \cap \de^*E)=0$ and $\Per_{\dista{\Omega}}(E\cap B_r(x_0)) = \Per_{\dista{\Omega}}(E;B_r(x_0)) + \Hc^{n-1}_{\dista{\Omega}}(\de B_r(x_0) \cap E)$, 
	taking the limit as $s\to r$ we get
	\begin{equation}\label{eq:contr_filtration}
		\Per_{\dista{\Omega}}(E \cap B_r) \le \Hc^{n-1}_{\dista{\Omega}}(\de B_r \cap E) + (1+\vartheta r^\beta) \Per_{\dista{\Omega}}(\de B_r \cap E).
	\end{equation}
		
	We now show that \eqref{eq:contr_filtration} leads to a contradiction. 
	The right hand side can be bounded by above by
	\begin{equation}\label{eq:techflat_est}
		\Hc^{n-1}_{\dista{\Omega}}(\de B_r(x_0) \cap E) + (1+\vartheta r^\beta) \Per_{\dista{\Omega}}(\de B_r(x_0) \cap E) \le \delta C(d_{\Omega}(x_0)+r)^a r^{n-2}.
	\end{equation} 
	
	For the lower bound in \eqref{eq:contr_filtration}, we use the lower density estimates presented in \Cref{prop:density_a} and \Cref{lemma:density_estimates_away}. 
    We need to consider two separated cases, depending on the value of $d_\Omega(x_0)$.
        
      \textbf{Case $d_\Omega(x_0) \le 2$}.
        From \eqref{eq:lowerdensityestimate} we already know that 
	\begin{equation}
		\Per_{\dista{\Omega}}(E\cap B_r(x_0)) 
		\ge \left[ \Lc^{n}_{\dista{\Omega}}(E\cap B_r(x_0)) \right]^{\frac{n-1+a}{n+a}}
		\ge c r^{n-1+a}.
	\end{equation}
	Together with $d_\Omega(x_0) \le 2$, $r\ge 1/2$ and \eqref{eq:techflat_est}, we obtain
	\begin{equation}
		\delta \ge \frac{c r^{1+a}}{(d_{\Omega}(x_0)+r)^a} \ge c_1.
	\end{equation}
	
	\textbf{Case $d_\Omega(x_0) > 2$}. The interior density estimates imply 
	\begin{equation}
		\Per_{\dista{\Omega}}(E\cap B_r(x_0)) 
		\ge \left[ \Lc^{n}_{\dista{\Omega}}(E\cap B_r(x_0)) \right]^{\frac{n-1}{n}}
		\ge c d_\Omega(x_0)^a \,r^{n-1}.
	\end{equation}
	Since $4r \ge d_\Omega(x_0)$, going back to \eqref{eq:techflat_est} we get
	\begin{equation}
		\delta \ge \frac{c\,r}{\left(1+\frac{r}{d_\Omega(x_0)}\right)^a} \ge c_2.
	\end{equation}

    In both cases, taking $\delta = \frac{1}{2}\min\{c_1,c_2\}$, we get a contradiction.
\end{proof}

As a straightforward consequence, we get the following result, whose proof we omit since it is a direct consequence of the previous proposition and of the area formula.

\begin{corollary}\label{cor:lowerboundperimeter}
	Let $\pi : \RR^n \to \RR^{n-1}$ be the projection onto $\{x_1=0\}$ given by $\pi(x_1,x'') = x''$.
	If $E$ satisfies the hypotheses of \Cref{lemma:tech_oscill}, then for all open set $U\subset B_{1-2\delta}$ it holds $\pi(\de E\cap U) = \pi(U)$.
	In particular for all $U\subset B_{1-2\delta}$ it holds
	\begin{equation}
		\Hc^{n-1}(\de E\cap U) \ge \Hc^{n-1}(\pi(U)).
	\end{equation}
\end{corollary}

\subsection{Interior Harnack inequality}\label{subsec:interior_harnack}
The goal of the present Subsection is proving the following
\begin{proposition}\label{prop:interior_harnack}
    There exist positive universal constants $\lambda_2,\eta_2,\eps_2$ (small) and $C_2$ (large) with the following property.
    Let $E$ be a $(\vartheta,\beta)$-minimizer of $\Per_{\dista{\Omega}}$ in $B_1$ and assume that $d_\Omega(0)\ge C_2$ and that, for some $\nu\in\sphere$,
    \begin{equation}
        C_2\left(\frac{\left\|\nabla d_\Omega\cdot \nu\right\|_{L^\infty(B_1)}}{d_\Omega(0)}+\vartheta^{\lambda_2}\right)\le\osc_{\nu}(\de E;B_1)\le\eps_2.
    \end{equation}
    Then
    \begin{equation}
        \osc_{\nu}(\de E;B_{\eta_2})\le (1-\eta_2)\osc_\nu(\de E;B_1).
    \end{equation}
\end{proposition}

As explained in \Cref{subsec:sketch}, in order to prove \Cref{prop:interior_harnack} above we use the fact that an almost-minimizer is a solution of a particular partial differential equation in a weak viscosity sense.
Let us elaborate further.
The appropriate class of solution we should use was introduced in \cite{DeSilva_Savin_2021} to define a notion of 
\textit{supersolutions of size $I$ at scale $r$}.
We start by defining the following class of standard test paraboloids:
\begin{equation}\label{eq:test_paraboloid0}
    \Fc_\Lambda\coloneqq\left\{\phi(x'')\coloneqq \frac{1}{2}|x''|^2 - \frac{\Lambda}{2}(x''\cdot \nu'')^2 + \xi''\cdot x'' + b\colon \nu''\in\sphere[n-2],\xi''\in B_1^{n-1},b\in\RR\right\}.
\end{equation}
Notice that, if $\phi\in \Fc_\Lambda$, then $\Delta\phi \equiv n-1-\Lambda$.
We now give the following
\begin{definition}
    Given an interval $I\subset\RR$, $\Lambda>0$ and $r>0$, we say that $E\in \Pc_{\Lambda}^I(r)$ in $B_R(x_0)$ if $E$ cannot be touched from outside at any point $z\in B_R(x_0)$ in a neighborhood $B_r(z)$ by $\{x_1\le \sigma \phi(x''-z'')\}$ for any $\sigma\in I$ and any $\phi\in\Fc_\Lambda$.    
\end{definition}
We are now ready to state and prove that an almost minimizer belongs, in fact, to the class of viscosity solutions we have just introduced:
\begin{proposition}[Weak viscosity property]\label{prop:weak_viscosity}
    Given $\delta>0$, there exist positive constants $c,\lambda$ (small) and $C$ (large) with the following property.
    Let $E$ be a $(\vartheta,\beta)$-minimizer of $\Per_{\dista{\Omega}}$ in $B_1$ and assume that $d_\Omega(0)\ge C$.
    For every $\eps>0$ such that
    \begin{equation}
        C\left(\frac{\|\nabla d_\Omega\cdot e_1\|_{L^\infty(B_1)}}{d_\Omega(0)}+\vartheta^\lambda\right)\le \eps\le c
    \end{equation}
    it holds
    \begin{equation}
        E\in\Pc_{4n}^{[\delta\eps,c]}(\eps) \mbox{ in }B_{1/2}.
    \end{equation}
\end{proposition}
\begin{proof}
    We argue by contradiction: assume $\sigma\in[\delta \eps,c]$ and $\tilde{\phi}\in\Fc_{4n}$ are such that $\{x_1\le \sigma \tilde{\phi}(x''-z'')\}$ touches $\de E$ from outside at $z\in B_1$ in a neighborhood $B_\eps(z)$. 
    Let $p(x'') = \sigma\left(\frac{\eps^2}{16}-|x''-z''|^2\right)$, $\phi(x'') = \sigma \tilde{\phi}(x''-z'')$, and, as in \Cref{lemma:technical_geom}:
    \begin{gather}
        F=E\setminus (B_\eps(z)\cap\{x_1\le\phi-p\}),\qquad
        G(x)=x_1-\phi(x'')+p(x'').
    \end{gather}
    Then, for some universal constant $C$,    \begin{multline}\label{eq:interior_max_princ_estimate}
        C\vartheta(d_\Omega(z)+\eps)^a \eps^{n-1+\beta}
        \ge \int_{E\setminus F} \dive\left(\dista{\Omega}\frac{\nabla G}{|\nabla G|}\right)\dif\Lc^n\\
        \ge(d_\Omega(z)-\eps)^a\int_{E\setminus F}\left(\dive\frac{\nabla G}{|\nabla G|}+a\frac{\nabla G\cdot \nabla d_\Omega}{|\nabla G|d_\Omega}\right)\dif\Lc^n.
    \end{multline}
    Notice that $\nabla G = e_1-\nabla''(\phi-p)$ and that $|\nabla''(\phi-p)|\le C\sigma$ (for some $C>0$ universal).
    Thus assuming that $\sigma$ is smaller than some universal constant, it holds $\frac{1}{2}\le 1-C\sigma^2\le|\nabla G|^{-1}\le1$ for some universal constant $C>0$.
    Straightforward computations then give
    \begin{equation}
        \dive \frac{\nabla G}{|\nabla G|}\ge -\Delta \phi - 2(n-1)\sigma-C\sigma^3\ge(n+1)\sigma,
    \end{equation}
    for some universal constant $C>0$, provided $\sigma>0$ is smaller than some universal constant.
    Next, we compute
    \begin{align}
        a\frac{\nabla G\cdot \nabla d_\Omega}{d_\Omega}\ge -a\frac{|\nabla d_\Omega\cdot e_1|+|\nabla''(\phi-p)|}{d_\Omega(0)-2}\ge -\sigma
    \end{align}
    provided $d_\Omega(0)\ge C$ for some $C$ large and $\frac{|e_1\cdot\nabla d_\Omega|}{d_\Omega(0)}\le c\eps$ for some $c$ small, both $C$ and $c$ depending on $\delta$.

    Going back to \eqref{eq:interior_max_princ_estimate} and using $d_\Omega(0)\ge C$ again, we find
    \begin{equation}
        C\vartheta \eps^{n-1+\beta}\ge n\sigma\Lc^n(E\setminus F)\ge n\sigma\Lc^n(E\cap B_{c\sigma \eps^2}(z)),
    \end{equation}
    where the last inequality is given by \Cref{item:smallballinside} in \Cref{lemma:technical_geom}.
    Using \Cref{lemma:density_estimates_away} and rearranging terms, we obtain
    \begin{equation}
        C\vartheta \ge \eps^{n+1-\beta}\sigma^{n+1}\ge \delta^{n+1}\vartheta^{\lambda(2n+2-\beta)},
    \end{equation}
    for some $C>0$ large universal, which fails if $\lambda$ is small enough and $\vartheta^\lambda\le c$ for some $c$ small enough.
    
\end{proof}

As explained above, the next result (\Cref{prop:weakharnack}), borrowed from \cite{DeSilva_Savin_2021}, states that elements of $\Pc_\Lambda^{I}(r)$ satisfy a weak Harnack inequality.
Before stating it, we set some notation.

We shall need a Calderon-Zygmund-type decomposition of $\RR^{n-1}$.
To this end, we introduce the following notation for cubes:
\begin{gather}
    Q''_\rho(x_0'') = \left\{x''\in\RR^{n-1}\colon |(x''-x_0'')\cdot e_j|\le \frac{\rho}{2}\mbox{ for all }j=2,\dots,n\right\},\\
    Q_\rho^r(x_0) = \left[(x_0)_1-\frac{r}{2},(x_0)_1+\frac{r}{2}\right]\times Q_\rho''(x_0'');
\end{gather}
we also write $Q''_\rho\coloneqq Q''_\rho(0)$ and $Q_\rho(x_0)\coloneqq Q_\rho^1(x_0)$.
Next, we introduce the family of dyadic cubes of side length $2^{-\ell}$:
\begin{equation}
    \Qc_\ell\coloneqq\left\{Q''_{2^{-\ell}}(x'')\colon x''\in 2^{-\ell}\ZZ^{n-1}\right\}
\end{equation}
and, given any set $A\subset\RR^{n-1}$, we let
\begin{equation}\label{eq:dyadic_exp}
    A_\ell\coloneqq \bigcup_{\substack{Q''\in\Qc_\ell\\Q''\cap A\neq\emptyset}}Q''.
\end{equation}

Lastly, for $\sigma>0$ and $y\in\RR^n$, we let $p^\sigma_y(x'')\coloneqq \frac{\sigma}{2}|x''-y''|^2+ y_1$, $F_y^\sigma = \{x_1 < p_y^\sigma(x'') \}$.
We then define the upper contact set
\begin{equation}\label{eq:definition_Asigma}
    A^\sigma(E) = \left\{ x'' \in B_1'' \middle| 
        \begin{array}{l} \exists\; x =(x_1,x'') \in E\mbox{ and } F^\sigma_\Yy \mbox{ such that }
        y'' \in B_1''\\
		 \mbox{ and  $F^\sigma_\Yy$ touches }E \mbox{ from outside at }
         x \mbox{ in }B_1
	\end{array}
	 \right\}.
\end{equation}
\begin{proposition}[Weak Harnack Inequality]\label{prop:weakharnack}
    There exist universal constants $\bar C$ and $\mu$ with the following property.
    Let $E\in\Pc^{[\tau,T]}_\Lambda(r)$ in $B_{1/2}$ for some $0<\tau<T$ and $r>0$ and assume that there exists $x_0''$ and $\rho>0$ such that 
    \begin{equation}
        A^\tau(E)\cap Q''_{3\cdot 2^{-M}}\neq\emptyset
    \end{equation}
    for some integer $M\ge \bar C$.
    Then for all $\ell\in\NN$ such that $\bar Cr\le 2^{-\ell}\le2^{-M}$ and all $h\in\NN$ such that $\bar C^h\tau\le T$ it holds
    \begin{equation}
        \Lc^{n-1}\big(A^{\bar C^h\tau}_\ell(E)\cap Q''_{2^{-M}}\big)\ge (1-\mu^h)\Lc^{n-1}(Q''_{2^{-M}}).
    \end{equation}
\end{proposition}

For the proof of \Cref{prop:weakharnack} we refer to \Cref{sec:appendix_DeSilvaSavin}, where we adapt the main ideas presented in \cite{DeSilva_Savin_2021} to our setting.

\begin{proof}[Proof of \Cref{prop:interior_harnack}]
    Up to a change of coordinates, we may assume that $\nu=e_1$. 
    We also call $\eps=\osc(\de E; B_1)$, and without loss of generality we can assume the oscillations to be centered in $0$, that together with \Cref{lemma:tech_oscill}, implies 
    \begin{equation}
        \{x_1 \le -\eps \}\cap B_1 \subset E \cap B_1 \subset \{x_1 \le \eps \}.
    \end{equation}
    Lastly, we prove \Cref{prop:interior_harnack} for $Q_\eta$ rather than $B_\eta$, since it leads to no loss of generality and rectangles are easier to manage with the statement of \Cref{prop:weakharnack}.
    
    We argue by contradiction, assuming that for some $\eta:=2^{-M}>0$ small to be chosen later both
    \begin{equation}\label{eq:contr_assumption_interior_H}
        \de E\cap Q_\eta\cap \{x_1>(1-\eta^3)\eps\}\neq\emptyset
    \end{equation}
    and
    \begin{equation}\label{eq:contr_assumption_interior_H2}
        \de E\cap Q_\eta\cap \{x_1<-(1-\eta^3)\eps\}\neq\emptyset.
    \end{equation}
    
    \textbf{{ Step 1.}} Firstly, from \eqref{eq:contr_assumption_interior_H}, \eqref{eq:contr_assumption_interior_H2}, and \Cref{prop:weakharnack} we derive a measure estimate for (discrete) superlevel and sublevel sets.
    
    We begin by fixing some (universal) constants we will use later.
    Let $\mu$ and $\bar C$ be the universal constants given by \Cref{prop:weakharnack}.
    Choose $h\in\NN$ large such that $\mu^h\le\frac{1}{4}$, and let $\delta$ be so small that $\bar C^h\delta\le\frac{1}{8}$.
    Corresponding to $\delta$, we let $c,\lambda$ and $C$ denote the constants given by \Cref{prop:weak_viscosity}.
    Notice that $\mu,\bar C,h,\delta,c,\lambda$ and $C$ are all independent of $\eta$.
    By \Cref{prop:weak_viscosity}, if $E$ satisfies the assumptions of \Cref{prop:interior_harnack}, $\lambda_2\le\lambda$, $\eps_2\le c$, and $C_2\ge C$, then $E\in\Pc_{4n}^{[\delta\eps,c]}$ in $B_{1/2}$.
    By \eqref{eq:contr_assumption_interior_H} and \eqref{eq:definition_Asigma}, if $\eta$ is smaller than some universal constant,
    \begin{equation}
        A^{\delta\eps}(E)\cap Q''_{3\eta}\neq\emptyset.
    \end{equation}
    
    By the discussion above, we may apply
    \Cref{prop:weakharnack} with the choices $\tau=\delta\eps$, $T=c$, $r=\eps$, and $\rho=\eta$.
    Therefore, up to choosing $\eps_2$ even smaller, if needed, so that $\bar C\eps_2\le \eta$ and $\bar C^h\delta\eps_2\le c$, we obtain
    \begin{equation}\label{eq:weak_H_applied}
        \Lc^{n-1}(A_{\ell}^{\bar C^{h}\delta\eps}(E)\cap Q''_{\eta})\ge\frac{3}{4}\Lc^{n-1}(Q''_{\eta})
    \end{equation}
    where $\ell\in\NN$ satisfies $2^{-\ell-1}\le \bar C\eps\le 2^{-\ell}\le\eta$.

    By definition, if $x''\in A^{\bar C^{h}\delta\eps}(E)$, then there exists $F_y^{\bar C^h\delta \eps}$ that touches $E$ from outside at $(x_1,x'')$. 
    Then, by \eqref{eq:contr_assumption_interior_H}, it must be $x_1\ge y_1\ge(1-\eta^3-\bar C^{h}\delta)\eps\ge\frac{3}{4}\eps$ by our assumptions on $\delta$ and provided $\eta^3\le\frac{1}{8}$.
    Therefore, letting
    \begin{equation}
        A^+(E)\coloneqq \left\{x''\in B''_1\colon \exists x_1\ge\frac{3}{4}\eps\mbox{ such that }(x_1,x'')\in\de E\right\}
    \end{equation}
    we have $A^{\bar C^{h}\delta\eps}(E)\subset A^+(E)$.
    Thus by \eqref{eq:weak_H_applied},
    \begin{equation}
        \Lc^{n-1}(A_{\ell}^+(E)\cap Q''_{\eta})\ge\frac{3}{4}\Lc^{n-1}(Q''_{\eta}),
    \end{equation}
    (with the notation $A^+_\ell(E)$ introduced in \eqref{eq:dyadic_exp}).
    We get the same conclusion considering $E^c$, using $A^{-}(E)=\{x''\colon (x_1,x'')\in\de E\mbox{ for some }x_1\le-3\eps/4\}$ in place of $A^+(E)$ and \eqref{eq:contr_assumption_interior_H2} in place of \eqref{eq:contr_assumption_interior_H}.
    Therefore
  	\begin{equation}\label{eq:densityofcubes}
		\Lc^{n-1}(A^{-}_\ell(E) \cap A^{+}_\ell(E) \cap Q''_{\eta}) \geq \frac{1}{2}\Lc^{n-1}(Q''_{\eta}).
	\end{equation}
    
	We underline that $A^{-}_\ell(E)$ and $A^+_\ell(E)$ are not disjoint, thus \eqref{eq:densityofcubes} is not a contradiction, but we now show that it implies a perimeter excess that is not compensated by the almost-minimality deficit.

    \textbf{{ Step 2.}} We show that there exists a (small) universal constant $c_2>0$ such that, for all $Q''\in\Qc_\ell$ such that $Q''\subset A^{+}_\ell(E) \cap A^{-}_\ell(E)$, it holds
	\begin{equation}\label{eq:tilt_excess}
		\Per(E;Q) \ge (1+c_2) \Lc^{n-1}(Q''),
	\end{equation}
    where $Q\coloneqq[-1,1]\times Q''$. 
	
	By contradiction, let $\{E_j\}_{j\in\NN}$ a sequence of $(\vartheta_j,\beta)$-minimizers in $B_1$ as in the hypotheses of \Cref{prop:interior_harnack} and $\{Q''(j)\}_{j\in\NN}$ a sequence of dyadic cubes such that
	\begin{gather}
		Q''(j) \subset A^{+}_\ell(E_j) \cap A^{-}_\ell(E_j),\\
        \limsup_{j\to\infty}\Per(E_j;[-1,1]\times Q''(j)) - \Lc^{n-1}(Q''(j))\le 0.\label{eq:harnack_smaller_mass_contradict}
	\end{gather}
    Since there are finitely many dyadic cubes in $Q_\eta''$, by the pigeonhole principle we may assume $Q''(j)=Q''_{2^{-\ell}}(0)$ for all $j\in\NN$.

    We now consider the rescaled sets $\tilde{E}_j\coloneqq 2^\ell E_j$, which are $(2^{-\beta\ell}\vartheta_j,\beta)$-minimizers in $Q_2^1$ and satisfy
    \begin{equation}
        \de \tilde E_j\cap Q_1\cap \left\{x_1\ge 2^{\ell}\frac{3}{4}\eps\right\}\neq\emptyset,\qquad\de \tilde E_j\cap Q_1\cap \left\{x_1\le-2^{\ell}\frac{3}{4}\eps\right\}\neq\emptyset.
    \end{equation}
    Notice that, by our choice of $\ell$, $2^{\ell}\frac{3}{4}\eps\ge\frac{3}{4\bar C}>0$.

    Since $\limsup\vartheta_j\le 1$, we can apply \Cref{prop:compactnessqmin}, and thus there exists a $(2^{-\beta\ell}\vartheta,\beta)$-minimizer $\tilde E$ in $Q_2$ such that
	\begin{equation}
		\tilde E_j\rightarrow \tilde E \mbox{ in } L^1_{loc}(Q_2), \qquad \mbox{and} \qquad \de^* \tilde E_j \weakstar \de^* \tilde E \mbox{ in } Q_2.
	\end{equation}
	By $L^1$ lower semicontinuity of the (classical) perimeter and by \eqref{eq:harnack_smaller_mass_contradict}, we get 
	\begin{equation}
    \Per(\tilde E;U) \le \Lc^{n-1}(\pi(U))\quad\mbox{for all }U\Subset Q_2,
	\end{equation}
    where $\pi(x_1,x'')=x''$;
    moreover, by \Cref{cor:lowerboundperimeter} (which holds true for $\tilde E$ in $Q_2$ up to choosing $\bar C$ greater than some universal constant), the converse inequality holds true.
	Then the area formula (see, for instance, \cite[\S 12]{simonGMT}) yields that the tangential Jacobian $J^{\de^*\tilde E}_\pi(x) = 1$ for $\Hc^{n-1}$-almost every $x\in\de^*\tilde E$, and therefore $\de^*\tilde E\cap Q_2$ is a hyperplane parallel to $\{x_1=0\}$. 
	
	On the other hand, \Cref{cor:convergence_Hausdorff} implies that $\de \tilde E_j \cap Q_{2}\to \de \tilde E\cap Q_2$ in the Hausdorff distance. 
	But for all $j\in\NN$, $\de \tilde E_j\cap Q_2$ intersects both $\{ x_1 \ge \frac{3}{4\bar C}\}$ and $\{x_1 \le -\frac{3}{4\bar C}\}$, thus they cannot converge toward a horizontal hyperplane, which is a contradiction.
	
    \textbf{{ Step 3.}} 
	We show that \eqref{eq:tilt_excess}, \eqref{eq:densityofcubes} allow us to define a competitor for $E$ that contradicts the almost-minimality.
    We preliminarily recall that for all $Q''\in\Qc_\ell$, we defined for any $s>0$ $Q^{2s} = [-s,s]\times Q''$, and $Q = Q^{1}$.

    Firstly, we deduce a lower bound on the classical perimeter of $E$ in $Q_\eta$. 
    Namely, letting $\Bc =\{Q''\in\Qc_\ell:  Q''\not\subset A^+_\ell(E)\cap A^-_\ell(E)\}$ and $\Gc = \{Q''\in\Qc_\ell :  Q''\notin \Bc\}$, it holds
    \begin{equation}
        \Per(E;Q_\eta) \ge \sum_{\substack{Q''\in\Qc_\ell \\Q''\subset Q_\eta''}} \Per(E;Q) 
        = \sum_{\substack{Q'' \in \Bc \\ Q'' \subset Q_\eta''}}\Per(E;Q)  + \sum_{\substack{Q'' \in \Gc \\ Q'' \subset Q_\eta''}}\Per(E;Q),
    \end{equation}
    and by \eqref{eq:tilt_excess}, \eqref{eq:densityofcubes} and \Cref{cor:lowerboundperimeter} we get
    \begin{equation}\label{eq:tilt_excess_eta}
        \Per(E;Q_\eta) \ge \left(1+\frac{c_2}{2}\right) \Lc^{n-1}(Q''_\eta).
    \end{equation}

    From this, we deduce a lower bound on $\Per_{\dista{\Omega}}(E;Q_\eta)$. 
    Since $\de E \cap Q_\eta\subset Q_\eta^{2\eps}$, where $\eps\ll\eta$, for all $x,y\in \de E \cap  Q_\eta$, it holds $\dista{\Omega}(x) \le (1+ C\eta)\dista{\Omega}(y)$.
    Thus, using \eqref{eq:tilt_excess_eta}, we deduce that 
    \begin{multline}
            \Hc^{n-1}_{\dista{\Omega}}(\{0\}\times Q''_\eta) 
                \le \dista{\Omega}(0) (1 + C\eta) \Lc^{n-1}(Q''_\eta) \\
                \le  \dista{\Omega}(0)\frac{1+C\eta}{1+\frac{c_2}{2}} \Per(E;Q_\eta)
                \le \frac{1+C\eta}{1+\frac{c_2}{2}}\Per_{\dista{\Omega}}(E;Q_\eta).
    \end{multline}
    Taking $\eta \le \frac{c_2}{8C}$, we get
    \begin{equation}\label{eq:lowerweightedmeasureest}
       \Per_{\dista{\Omega}}(E;Q_\eta) \ge \left(1+\frac{c_2}{4}\right)\Hc^{n-1}_{\dista{\Omega}}(\{0\}\times Q''_\eta).
    \end{equation}
	
	We are now in a position to define our competitor. 
    Let $s\in(0,1)$ be such that $\Hc^{n-1}(\de^*E \cap  \de Q_{s\eta})=0$, and define $F_s$ as
    \begin{equation}\label{eq:harnack_competitor}
        F_{s} := (E \setminus Q_{s\eta}^{4\eps}) \cup (\{x_1\le 0\} \cap Q_\eta)
    \end{equation}
    so that $F_s \Delta E \Subset Q_\eta^{4\eps} \Subset Q_{\eta}$.
    \begin{center}
    	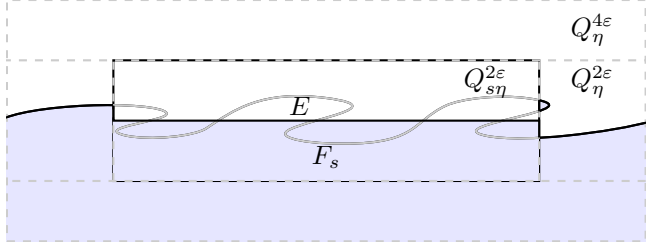
\begin{figure}[ht]
    		\begin{tikzpicture}[use Hobby shortcut, thick,even odd rule,yscale =0.4,xscale=0.7]

\begin{scope}
	\clip(-6,-4) rectangle (6,4);
	
		\draw[yscale=0.5,xscale=1.5,fill = blue!10] (-4.1,0) .. (-3,1) .. (-2,0.5) .. (-2.5,-1) .. (-1.5,-0.5) .. (0,1.5) .. (-0.5,-1) .. (0,-1.5)  .. (1.5,1) .. (2.5,1.5) .. (2,-1) .. (3,-1) .. (4.1,0) -- (4.1,-10) -- (-4.1,-10) ;
		
		\draw[fill = white, draw opacity =0] (-4,-2) rectangle (4,2);

		\fill[blue!10] (-4,-2) rectangle (4,0);

		\draw[dashed, color=black!20] (-4,-2) -- (-4,2);
		\draw[dashed, color=black!20] (4,-2) -- (4,2);
		
		\draw[dashed, color=black!20] (-6.2,-2) rectangle (6.2,2);
		\draw[dashed, color=black!20] (-6,-4) rectangle (6,4);

		\draw[yscale=0.5,xscale=1.5, color=black] (-4.1,0) .. (-3,1) .. (-2,0.5) .. (-2.5,-1) .. (-1.5,-0.5) .. (0,1.5) .. (-0.5,-1) .. (0,-1.5)  .. (1.5,1) .. (2.5,1.5) .. (2,-1) .. (3,-1) .. (4.1,0) -- (4.1,-10) -- (-4.1,-10) ;
		\begin{scope}
			\clip (-4,-2) rectangle (4,2);
					\draw[yscale=0.5,xscale=1.5, color=black!20] (-4.1,0) .. (-3,1) .. (-2,0.5) .. (-2.5,-1) .. (-1.5,-0.5) .. (0,1.5) .. (-0.5,-1) .. (0,-1.5)  .. (1.5,1) .. (2.5,1.5) .. (2,-1) .. (3,-1) .. (4.1,0) -- (4.1,-10) -- (-4.1,-10) ;
		\end{scope}
		
		\draw (-4,0.5) -- (-4,0) -- (4,0) -- (4,-0.56);
		\draw (4,0.305) -- (4,0.69);

		\node at (0,-1.2) {$F_s$};
		\node at (-0.5,0.5) {$E$};

		\node at (5,1.3) {$Q_\eta^{2\eps}$};
		\node at (3,1.3) {$Q_{s\eta}^{2\eps}$};
		\node at (5,3) {$Q_\eta^{4\eps}$};
		
\end{scope}
\end{tikzpicture}
        \caption{The competitor $F_s$ defined in \eqref{eq:harnack_competitor}}
    	\end{figure}
    \end{center}

    Since $\eps\ll\eta$, the $(\vartheta,\beta)$-minimality ensures that 
	\begin{multline}
		\Per_{\dista{\Omega}}(E; Q_{\eta}) 
            \leq (1+\vartheta \eta^\beta ) \Per_{\dista{\Omega}}(F_s; Q_{\eta}) \\
            \le (1+C\vartheta \eta^\beta)\Big( \Hc^{n-1}_{\dista{\Omega}}
            (\{0\}\times Q_\eta'')\\
             + C\dista{\Omega}(0)\eta^{n-2}\eps  + \Per_{\dista{\Omega}}(E;Q_{\eta}\setminus Q_{s\eta}^{2\eps})\Big).
	\end{multline}
        Taking $s\to 1$ we get 
        \begin{equation}
            \Per_{\dista{\Omega}}(E; Q_{\eta}) 
            \le (1+\vartheta \eta^\beta)\left( \Hc^{n-1}_{\dista{\Omega}}(\{0\}\times Q_\eta'') + C\dista{\Omega}(0)\eta^{n-2}\eps \right).
        \end{equation}
                Finally, let us assume that $\eps_2$ is even smaller than previously specified, such that $\eps_2 \leq \eta^2$. 
                Since $\vartheta \leq 1$, the previous estimate, together with \eqref{eq:lowerweightedmeasureest}, implies
        \begin{equation}
            \left(1 + \frac{c_2}{4}\right) \Hc^{n-1}_{\dista{\Omega}}(\{0\} \times Q_{\eta}'') 
            	\le (1+C \eta^{\beta}) \Hc^{n-1}_{\dista{\Omega}}(\{0\}\times Q_\eta''),
        \end{equation}
        and taking $\eta$ small enough depending on $c_2$, we achieve the contradiction.
        \qedhere

\end{proof}

\begin{corollary}[Decay of oscillations in the interior]\label{cor:osc_decay_interior}
    There exist universal constants $\sigma_2$ (small) and $C'_2$ (large) with the following property.
    Let $E$ be a $(\vartheta,\beta)$-minimizer of $\Per_\dista{\Omega}$ in $B_R(x_0)$, where $x_0\in\Omega$ and $R$ are such that $d_\Omega(x_0)\ge C'_2R$, and let
    \begin{equation}
        \eps\coloneqq \osc_{\nu}(\de E;B_R(x_0)) + R\left((\vartheta R^\beta)^{\lambda_2}+R\frac{||\nabla d_\Omega\cdot \nu||_{L^\infty(B_R(x_0))}}{d_\Omega(x_0)}\right) \le \eps_2
    \end{equation}
    for some $\nu\in\sphere$, where $\lambda_2$ is as in \Cref{prop:interior_harnack}.
    Then, for every $r\in[C'_2\eps R,R]$, it holds
    \begin{equation}
        \osc_{\nu}(\de E;B_r(x_0))\le C'_2\eps \left(\frac{r}{R}\right)^{\sigma_2}.
    \end{equation}
\end{corollary}
\begin{proof}
    Without loss of generality, we may assume $x_0=0$ and $R=1$.
    Let
    \begin{equation}
        F(r)\coloneqq\osc_{\nu}(\de E;B_r) + Ar\left(r\frac{||\nabla d_\Omega\cdot \nu||_{L^\infty(B_r)}}{d_\Omega(x_0)}+(\vartheta r^\beta)^{\lambda_2}\right),
    \end{equation}
    where $A$ is a large constant we will determine later. We claim the following:
    \begin{equation}\label{eq:int_harn_claim}
        \mbox{if }F(r)\le\eps_2 r,\quad\mbox{then}\quad F(\eta r)\le(1-\eta) F(r)
    \end{equation}
    for any $r>0$, where $\eta\coloneqq
    \min\{\eta_2,1/2\}$ and $\eta_2,\eps_2$ are as in \Cref{prop:interior_harnack}. 
    Assuming \eqref{eq:int_harn_claim}, by induction we obtain
    \begin{equation}\label{eq:interior_harnack_iteration}
        F(\eta^k)\le (1-\eta)^kF(1)\mbox{ for all }k\in\NN\mbox{ such that }(1-\eta)^{k-1}F(1)\le\eps_2\eta^{k-1}.
    \end{equation}
    Then, given $r\in [C\eps,1]$, we let $k\in\NN$ be such that $\eta^{k+1}<r\le\eta^{k}$. If $C$ and $A$ are chosen large enough, with $C\ge A$, then we have
    \begin{equation}
        (1-\eta)^{k-1}F(1)\le F(1)\le A\eps\le \eps_2 r\le\eps_2 \eta^k,
    \end{equation}
    hence by \eqref{eq:interior_harnack_iteration} we have
    \begin{equation}
        \osc_{\nu}(\de E;B_r)\le F(\eta^k)\le (1-\eta)^k F(1) \le A\eps r^\sigma
    \end{equation}
    where $\sigma\in(0,1)$ satisfies $1-\eta\le\eta^\sigma$. This proves the desired result.

    We are left with the proof of \eqref{eq:int_harn_claim}.
    \begin{itemize}
        \item If 
        \begin{equation}\label{eq:osc_decay_vince_flat}
            C_2\left(r\frac{||\nabla d_\Omega\cdot \nu||}{d_\Omega(0)}+(\vartheta r^\beta)^{\lambda_2}\right)\le\frac{1}{r}\osc_{\nu}(\de E;B_r),
        \end{equation}
        then by \Cref{prop:interior_harnack} it holds
        \begin{equation}
            \osc_{\nu}(\de E;B_{\eta r})\le (1-\eta)\osc_{\nu}(\de E;B_r),
        \end{equation}
        thus
        \begin{align}
            F(\eta r)&\le(1-\eta)\osc_{\nu}(\de E;B_r) \\
            &\qquad+ A\eta r (||\nabla d_\Omega\cdot \nu||+(\vartheta (\eta r)^\beta)^{\lambda_2})\\
            &\le (1-\eta) F(r),
        \end{align}
        where the last inequality holds true since $\eta\le\frac{1}{2}$.
        \item On the other hand, if \eqref{eq:osc_decay_vince_flat} fails, then we trivially have
        \begin{align}
            F(\eta r)&\le \osc_{\nu}(\de E;B_r)+A\eta r\left((\vartheta(\eta r)^\beta)^{\lambda_2}+\frac{\eta r}{d_\Omega(0)}||\nabla d_\Omega\cdot \nu||\right)\\
            &\le Ar\left(r\frac{||\nabla d_\Omega\cdot \nu||}{d_\Omega(0)}+(\vartheta r^\beta)^{\lambda_2}\right)\left(\frac{C_2}{A}+\eta\right)\\
            &\le (1-\eta)F(r)
        \end{align}
        provided $A$ is large enough so that $C_2/A+\eta\le 1-\eta$.
    \end{itemize}
\end{proof}

\subsection{Boundary Harnack inequality and decay of oscillations}\label{subsec:boundary_harnack}
We now turn our attention to the decay of oscillations near the boundary of $\Omega$.

\begin{proposition}[Harnack inequality at the boundary]\label{prop:harnack_boundary}
    There exist small universal constants $\lambda_3,\eps_3$ and $\eta_3$ with the following property.
    Let $\Omega$ be $\varkappa$-flat in the sense of \Cref{ass:boundaryof_Omega} and let $E$ be a $(\vartheta,\beta)$-minimizer of $\Per_\dista{\Omega}$ in $B_1$.
    If
    \begin{equation}
        (\varkappa+\vartheta)^{\lambda_3}\le\osc_{e_1}(\de E;B_1)\le\eps_3
    \end{equation}
    then
    \begin{equation}
        \osc_{e_1}(\de E;B_{\eta_3})\le(1-\eta_3)\osc_{e_1}(\de E;B_1).
    \end{equation}
\end{proposition}
\begin{proof}
    This proof follows the idea introduced in \cite{DeSilva_2011}.
    Let $\eps\coloneqq\osc_{e_1}(\de E;B_1)$ . By \Cref{lemma:tech_oscill} we shall assume, without loss of generality, that
    \begin{equation}\label{eq:harnack_boundary_osc_initial}
			\{x_1\le -\eps\}\cap B_1\cap\Omega\subset E\cap B_1\cap\Omega\subset \{x_1\le \eps\}.
	\end{equation}
    
    Throughout the present proof, we will use four small, universal constants $\tau_1,\tau_2,\tau_3,\tau_4$ such that $\tau_1\ll \tau_2\ll \tau_3\ll \tau_4$ and whose value will be (implicitely) specified later.
    Finally, $\eps_3$ and $\eta_3$ will be chosen much smaller than $\tau_1$.

    Let $x_0 = (x_0)_1 e_1 +\tau_2 e_n\in\de E$.
    Such a point exists by \Cref{lemma:tech_oscill}. We assume that $-\eps\le (x_0)_1\le0$, since the other case can be handled by replacing $E$ with $\Omega\setminus E$.
    By \Cref{lemma:notes_on_distance} (assuming $\varkappa\ll \tau_2$), $d_\Omega(x_0)\ge \tau_2/2$ and, in a small neighborhood of $x_0$, $|\nabla d_\Omega\cdot e_1|\le C\varkappa\le \eps$ for some $C$ large, universal.
    By \Cref{cor:osc_decay_interior} applied in $B_{cd_\Omega(x_0)}(x_0)$, provided $\lambda_3\le\lambda_2$ and $\eps_3$ is small enough, we find $\tau_4$ universal so that
    \begin{equation}\label{eq:boundary_harn_oscill_int}
        \osc_{e_1}(\de E;B_{2\tau_2\tau_4}(x_0))\le \frac{\eps}{4}.
    \end{equation}
    For brevity, we let $r\coloneqq \tau_2\tau_4$ for the rest the proof.
    Notice that \eqref{eq:boundary_harn_oscill_int} above yields
    \begin{equation}\label{eq:expansion_of_positivity}
        E\cap B_{2r}(x_0)\subset\{x_1\le\eps/2\}.
    \end{equation}
    
    We now \enquote{slide from above} the family of sets $\{F_t\}_{t\in\RR}$ defined as
	\begin{equation}
		F_t = \inset{(x_1,x'')\in B_1 : x_1 \le -t - \left(\frac{1}{2}+\tau_3\right)\eps\Phi(x'')},
    \end{equation}
	where
	\begin{equation}
		\Phi(x'') = \min\left\{\left(\frac{r}{|x''-x_0''|}\right)^{n+a-2},1\right\}.
	\end{equation}
    Notice that, since $n\ge2$ and $a>0$, $n+a-2>0$; the choice of this particular exponent will be made clear in case \ref{item:bdry_contradiction} below.
    Notice also that $F_t\supset E\cap B_1^+$ for all $t\le-3\eps$.
    We claim that, actually, 
    \begin{equation}\label{eq:claim_bdry_harnack}
        F_t\supset E\cap B_1^+\quad\mbox{for all }t< t^*\coloneqq-\eps\left(1+\tau_3\right).
    \end{equation}
    The claim yields the desired result:
    indeed, assuming $\eta_3\le \tau_2$, for every $x\in \de E\cap B_{\eta_3}\cap\Omega$ it holds $|x''-x_0''|\le 3\tau_2$,
    hence $\Phi(x'')\ge (\tau_4/3)^{n+a-2}$ and
    \begin{align}
        x_1
        &\le \eps\left(1+\tau_3-\left(\frac{1}{2}+\tau_3\right)\left(\frac{\tau_4}{3}\right)^{n+a-2}\right)\\
        &\le\eps\left(1+\tau_3 - \frac{1}{2}\left(\frac{\tau_4}{3}\right)^{n+a-2}\right)\\
        &\le\eps\left(1 - \frac{1}{4}\left(\frac{\tau_4}{3}\right)^{n+a-2}\right)\label{eq:bdry_harnack_desired}
    \end{align}
    where the last inequality holds true provided $\tau_3$ is chosen much smaller than $\tau_4$.
    \eqref{eq:bdry_harnack_desired} gives the desired result provided $\eta_3 \le \min\{\tau_2,\frac{1}{4}(\tau_4/3)^{n+a-2}\}$.
    
    To prove \eqref{eq:claim_bdry_harnack}, we argue by contradiction: if not, then there is $\bar t\in[-3\eps,t^*)$ and a point $\bar x\in\overline{B_1\cap{\Omega}}$ such that $\bar x\in \de F_{\bar t}\cap\de E$ and $F_{\bar t}\supset E\cap B_1^+$.
    We show that this cannot be the case.
    \begin{enumerate}[(i)]
        \item First of all, by \eqref{eq:expansion_of_positivity} we can exclude the case $|\bar x''-x_0''|\le2 r$.
        \item Next, we exclude the case $|\bar x|\ge 1/2$. Indeed, in that case, provided $\eps_3$ and $\tau_2$ are smaller than some universal constant, it holds $|\bar x''-x_0''|\ge 1/3$, thus $\Phi(\bar x'')\le (3r)^{n+a-2}\le \tau_2^{n+a-2}$.
        Using $\bar x\in \de F_{\bar t}$, we find
        \begin{align}
            \bar x_1
            &> -t^* - \eps\left(\frac{1}{2}+\tau_3\right)\Phi(\bar x'')\\
            &\ge \eps\left(1+\tau_3 - \tau_2^{n+a-2} \right)
        \end{align}
        which is greater than $\eps$ since $\tau_2$ is much smaller than $\tau_3$.
        This, however, contradicts \eqref{eq:harnack_boundary_osc_initial}.
        \item If $|\bar x|\le \frac{1}{2}$, $|\bar x''-x_0''|\ge 2 r$ and $\bar x\in\de\Omega$, we first remark that,  since $|\bar x''-x_0''|\le 1$ and $    \bar x_n\le \varkappa\le\eps^{1/\lambda}\le \tau_2/2$, $\de_n\Phi(\bar x'')\ge \bar c$ for some $\bar c$ small universal.
        Therefore, using the fact that $\Omega$ is $\varkappa$-flat, we compute
        \begin{align}
            \nu_{F_{\bar t}}(\bar x)\cdot \nu_\Omega(\bar x)
            &\ge-\nu_{F_{\bar t}}(\bar x)\cdot e_n-\varkappa\\
            &\ge \frac{1}{2}\eps\left(\frac{1}{2}+\tau_3\right)\de_n\Phi(\bar x'') - \varkappa\\
            &\ge \frac{1}{4}\bar c\eps-\varkappa.\label{eq:normali_contradiction}
        \end{align}
        Recalling $\varkappa\le\eps^{1/\lambda}$ and assuming $\eps$ smaller than some universal constant, we find
        \begin{equation}
            \nu_{F_{\bar t}}(\bar x)\cdot \nu_\Omega(\bar x)>0,
        \end{equation}
        which contradicts \Cref{prop:boundary_max}.
        
        \item\label{item:bdry_contradiction} Lastly, we consider the case $|\bar x|\le \frac{1}{2}$, $|\bar x''-x_0''|\ge 2 r$ and $\bar x\in\Omega$.
        For brevity, let $\phi(x'')\coloneqq-\bar t-\left(\frac{1}{2}+\tau_3\right)\eps\Phi(x'')$.               
        Then $F_{\bar t}=\{x_1\le \phi(x'')\}$ touches $E$ from outside at $\bar x$ in a neighborhood $B_r(\bar x)$,
        and there exists $C>0$ large universal such that $|\nabla \phi|\le C\eps$ in $B_r''(\bar x'')$.
        Let $p(x'') = \tau_1\eps\left(\frac{r^2}{16}-|x''-\bar x''|^2\right)$
        and let 
        \begin{equation}
            F\coloneqq E\setminus\bigg(B_r(\bar x)\cap \{x_1\le\phi-p\}\bigg),\quad G(x) \coloneqq x_1 - \phi(x'')+p(x''),
        \end{equation}
        as in \Cref{lemma:technical_geom}.
        By \eqref{eq:technical_geom_integral}, we have
        \begin{equation}\label{eq:bdry_harnack_integral}
            C\vartheta \ge \int_{E\setminus F}\dista{\Omega}\left(\dive\left(\frac{\nabla G}{|\nabla G|}\right)+a\frac{\nabla G\cdot\nabla d_\Omega}{d_\Omega|\nabla G|}\right)\dif\Lc^n,
        \end{equation}
        where $C$ is a large universal constant.
        By computations similar to those in the proof of \Cref{prop:weak_viscosity}, for $x\in B_r(\bar x)$ we find
        \begin{equation}
            \dive\left(\frac{\nabla G}{|\nabla G|}\right)(x) \ge \Delta p(x'')-\Delta\phi(x'')-C\eps^3
        \end{equation}
        for some $C$ universal.
        For the second summand in the integrand on the right-hand side of \eqref{eq:bdry_harnack_integral}, using \Cref{lemma:notes_on_distance}
        we find
        \begin{align}
            \frac{\nabla G\cdot \nabla d_\Omega}{|\nabla G|d_\Omega}
            &\ge \frac{\nabla G\cdot e_n}{d_\Omega|\nabla G|} - C\frac{\varkappa}{d_\Omega}\\
            & = \frac{1}{d_\Omega|\nabla G|}(\de_n(p-\phi))- C\frac{\varkappa}{d_\Omega}.
        \end{align}
        Next, we use \Cref{lemma:notes_on_distance} again and the facts that $\bar x_n\ge-\varkappa$ and $(x_0)_n=\tau_2>0$ to estimate
        \begin{align}
            \frac{\de_np(x'')}{d_\Omega(x)|\nabla G(x)|}
            =-2\tau_1\eps\frac{x_n-\bar x_n}{d_\Omega|\nabla G|}
            \ge -2\tau_1\eps-C\frac{\eps\varkappa}{d_\Omega}
        \end{align}
        and
        \begin{align}
            \frac{\de_n\phi(x'')}{d_\Omega(x)|\nabla G(x)|}
            &=(n+a-2)\frac{r^{n+a-2}}{|x''-x''_0|^{n+a}}\left(\frac{1}{2}+\tau_3\right)\eps\frac{x_n-(x_0)_n}{d_\Omega|\nabla G|}\\
            &\le C\frac{\eps\varkappa}{d_\Omega} + (n+a-2)\frac{r^{n+a-2}}{|x''-x''_0|^{n+a}}\left(\frac{1}{2}+\tau_3\right)\eps.
        \end{align}
        Gathering the above estimates, we find
        \begin{align}
            &\dive\left(\frac{\nabla G}{|\nabla G|}\right)+a\frac{\nabla G\cdot\nabla d_\Omega}{d_\Omega|\nabla G|}\\
            &\qquad\ge \left(-\Delta\phi - a(n+a-2)\frac{r^{n+a-2}}{|x''-x''_0|^{n+a}}\left(\frac{1}{2}+\tau_3\right)\eps\right)\\
            &\qquad\qquad+\left(\Delta p-2a\tau_1\eps\right)-C\eps^3-C\frac{\varkappa}{d_\Omega}\\
            &\qquad = (n+a-2)\frac{r^{n+a-2}}{|x''-x_0''|^{n+a}}\left(\frac{1}{2}+\tau_3\right)\eps\\
            &\qquad\qquad-2\tau_1\eps(n-1-a) - C\eps^3-C\frac{\varkappa}{d_\Omega}\\
            &\qquad\ge \frac{1}{C}\eps - C\frac{\varkappa}{d_\Omega}
        \end{align}
        where the last inequality holds true provided $\tau_1$ is chosen smaller than some universal constant and $\eps\le\eps_3$ is chosen even smaller.
        Going back to \eqref{eq:bdry_harnack_integral}, we find
        \begin{equation}\label{eq:bdry_harnack_almost}
            C\vartheta\ge \frac{1}{C}\eps\Lc^n_{\dista{\Omega}}(E\setminus F) - C\varkappa\int_{E\setminus F}d_\Omega^{a-1}\dif\Lc^n.
        \end{equation}
        Now, by \Cref{lemma:technical_geom} and \eqref{eq:lowerdensityestimate}, we have
        \begin{equation}
            \Lc^n_{\dista{\Omega}}(E\setminus F)\ge \Lc^n_{\dista{\Omega}}(E\cap B_{\eps^2/C}(\bar x))\ge \frac{1}{C}\eps^{2(n+a)}
        \end{equation}
        where $C$ is, as usual, a large universal constant.
        Since $a>0$, we also have
        \begin{equation}
            \int_{E\setminus F}d_\Omega^{a-1}\dif\Lc^n\le\int_{B_1\cap\Omega}d_\Omega^{a-1}\dif\Lc^n\le C.
        \end{equation}
        Therefore \eqref{eq:bdry_harnack_almost} yields
        \begin{equation}
            C\vartheta\ge \frac{1}{C}\eps^{2(n+a)+1}-C\varkappa
        \end{equation}
        which fails if
        \begin{equation}
            (\vartheta+\varkappa)^{\lambda_3}\le\eps
        \end{equation}
        and $\lambda_3$ is small enough.
        This excludes the last alternative and thus it concludes the proof of \eqref{eq:claim_bdry_harnack}.
    \end{enumerate}
\end{proof}

We now combine \Cref{prop:harnack_boundary} and \Cref{cor:osc_decay_interior} to obtain the following

\begin{corollary}[Decay of oscillations up to $\de\Omega$]\label{cor:full_decay}
    There exist positive universal constants $C_3$ (large) and $\sigma_3$ (small) with the following property.
    Let $\Omega$ be $\varkappa$-flat and let $E$ be a $(\vartheta,\beta)$-minimizer of $\Per_\dista{\Omega}$ in $B_1$.
    Let 
    \begin{equation}
        \eps\coloneqq\osc_{e_1}(\de E;B_{1}) + \left(\vartheta + \varkappa \right)^{\lambda_3},
    \end{equation}
    where $\lambda_3$ is given by \Cref{prop:harnack_boundary}.
    Then, for every $x\in B_{1/4}\cap\overline{\Omega}$ and every $r\in [C_3\eps,1/2]$, it holds
    \begin{equation}\label{eq:osc_dec_tot_thesis}
        \osc_{e_1}(\de E;B_r(x))\le C_3\eps r^{\sigma_3}.
    \end{equation}
\end{corollary}
\begin{proof}
    For some $A$ large to be chosen later and $\lambda_3$ as in \Cref{prop:harnack_boundary}, we let
    \begin{equation}
        F_x(r) \coloneqq \osc_{e_1}(\de E;B_r(x))+Ar\left(\vartheta r^\beta+\varkappa r^{\alpha}\right)^{\lambda_3}.
    \end{equation}
    
    \textbf{Case 1: $x\in\de\Omega\cap B_{1/2}$. }We reproduce the proof of \Cref{cor:osc_decay_interior}. 
    Notice that, after rescaling, translating and possibly rotating $\Omega$, the assumptions of \Cref{prop:harnack_boundary} are satisfied in $B_{1/2}(x)$. Arguing exactly as in the proof of \Cref{cor:osc_decay_interior}, but using \Cref{prop:harnack_boundary} instead of \Cref{prop:interior_harnack}, we prove that there exist two universal constants $C$ (large) and $\sigma$ (small) such that, for every $r\in[C\eps,1/2]$, it holds
    \begin{equation}\label{eq:osc_decay_primo}
        \osc_{e_1}(\de E;B_r(x))\le F_x(r)\le C\eps r^\sigma.
    \end{equation}

    If $x\in\Omega$, we prove \eqref{eq:osc_dec_tot_thesis} in two different cases, based on whether $d_\Omega(x)\ge \bar C r$ or not, where $\bar C$ is a large universal constant we will choose later.

    \textbf{Case 2: $x\in \Omega\cap B_{1/4}$ and $d_\Omega(x)\ge \bar Cr$. }
    We let $\rho\coloneqq\bar C^{-1}d_\Omega(x)$. 
        Since $B_{\rho}(x)\subset B_{(\bar C+1)\rho}(y)$ for some $y\in\de\Omega\cap B_{1/2}$, we may use \eqref{eq:osc_decay_primo} to obtain
        \begin{equation}
            \osc_{e_1}(\de E;B_\rho(x))\le\osc_{e_1}(\de E;B_{(\bar C+1)\rho}(y))\le C\eps \rho^\sigma
        \end{equation}
        and
        \begin{equation}
            \rho\left((\vartheta \rho^\beta)^{\lambda_2}+\rho\frac{||\nabla d_\Omega\cdot e_1||_{L^\infty
            (B_\rho(x))}}{d_\Omega(x)}\right)\le \rho(\vartheta^{\lambda_2}+\varkappa)\le \eps\rho^\sigma
        \end{equation}
        provided $\lambda_3\le\lambda_2$ given in \Cref{cor:osc_decay_interior} and $\sigma<1$.
        Therefore $F_x(\rho)\le C\eps\rho^\sigma$.
        Provided $\bar C$ is large enough and $\sigma$ is small enough, \Cref{cor:osc_decay_interior} yields
        \begin{equation}
            \osc_{e_1}(\de E;B_r(x))\le C_2' (C\eps\rho^\sigma)\left(\frac{r}{\rho}\right)^{\sigma_2}\le C\eps r^\sigma
        \end{equation}
        for all $r\ge C\eps\rho^{1+\sigma}$, which is the case since $r\ge C\eps$.

        \textbf{Case 3: $x\in \Omega\cap B_{1/2}$ and $d_\Omega(x)<\bar Cr$. }In this case, we choose $y\in \de\Omega\cap B_{1/2}$ such that $d_\Omega(x)=|x-y|$ and we estimate, using \eqref{eq:osc_decay_primo}:
        \begin{equation}
            \osc_{e_1}(\de E;B_r(x))\le \osc_{e_1}(\de E;B_{(\bar C+1)r}(y))\le C\eps r^\sigma
        \end{equation}
        up to choosing $C$ large enough, as desired.
        
\end{proof}

\subsection{Improvement of flatness}\label{sec:IOF}

In this section we prove \Cref{thm:improvement_flatness} and its counterpart for points away from $\de\Omega$ (\Cref{prop:interior_IOF}).
We anticipate that we will adopt the convention of identifying $\RR^{n-1}$ with $e_1^\perp$, so that $x''=(x''_2,\dots,x''_n)\in\RR^{n-1}$.
We will usually denote points in $\RR^{n-1}$ as $x'',y''$ and $B_r''(x'')=\{y''\in\RR^{n-1}\colon |x''-y''|<r\}$.
We also recall the notation 
\begin{equation}
    U^+\coloneqq\{x\in U\colon x_n\ge0\}, 
\end{equation}
where $U$ is either a subset of $\RR^n$ or of $\RR^{n-1}$.
Lastly, we introduce the notation
\begin{equation}
	\Cc_r(x'') = [-1,1]\times B''_r(x'')\subset\RR^{n}, \quad \Cc_r = \Cc_r(0'').
\end{equation}

\begin{proof}[Proof of \Cref{thm:improvement_flatness}]
    We argue by compactness.
    Assume there exist sequences $\Omega^j
    $ that are $\varkappa_j$ flat, $E^j$ that are $(\vartheta_j,\beta)$-minimizers of $\Per_{\dista{\Omega^j}}$ in $B_1$, and $\nu_j\in\sphere\cap e_n^\perp$, such that
    \begin{equation}
        \Big(\vartheta_j+\varkappa_j\Big)^{\lambda_{1}}\le\osc_{\nu_j}(\de E_j;B_1)\eqqcolon\eps_j
    \end{equation}
    for some $\eps_j\searrow0$, where $\lambda_{1}\coloneqq\lambda_3$ as in \Cref{cor:full_decay}.
    Without loss of generality, we assume $\nu_j=e_1$ for every $j$.

    Consider the rescaled sets
    \newcommand{\tE}{\tilde{E}}
    \newcommand{\tO}{\tilde{\Omega}}
    \begin{equation}
        \tE^j\coloneqq\{(x_1,x'')\in\Cc_1\colon (\eps_jx_1,x'')\in E^j\}
        \subset\tO^j
        \coloneqq\{(x_1,x'')\in\RR^n\colon (\eps_jx_1,x'')\in \Omega^j\}.
    \end{equation}
    Using \Cref{cor:full_decay} and arguing as in \cite{DeP_Gas_Schu24}, we prove that $\overline{\de \tE^j\cap \Cc_{3/4}\cap \tO^j}$ converge in the Hausdorff distance to the graph of some $C^{0,\sigma}$ function $u:(B_{3/4}'')^+\to[-1,1]$.

    Having defined $u$, we now prove that it solves
    \begin{equation}\label{eq:linearized}
        \begin{cases}
            \Delta u + a\frac{\de_n u}{x''_n}=0\quad&\mbox{in }B_{1/2}''\cap \{x_n>0\}\\
            \de_n u=0&\mbox{in }B_{1/2}''\cap \{x_n=0\}
        \end{cases}
    \end{equation}
    in the viscosity sense, meaning that whenever a smooth function $\phi:\RR^{n-1}\to\RR$ touches $u$ from above at some point $\bar y''\in (B''_{1/2})^+$ (that is $\phi(\bar y'')=u(\bar y'')$ and $\phi\ge u$ in some neighborhood $(B''_r(\bar y''))_+$) then:
    \begin{itemize}
        \item if $\bar y''_n>0$, then
    \begin{equation}
        \Delta\phi(\bar y'') + a\frac{\de_n \phi(\bar y'')}{\bar y''_n}\ge0;
    \end{equation}
        \item if $\bar y''_n=0$, then $\de_n\phi(\bar y'')\ge0$
    \end{itemize}
    and the opposite inequalities hold if $\phi$ touches $u$ from below.
    Towards the proof of the above claim, without loss of generality we may assume that $\phi$ is a paraboloid of the form $\phi(x'')\coloneqq\frac{1}{2}A(x''-\bar y'')\cdot (x''-\bar y'') +\xi''\cdot(x''-\bar y'')+c$
    and that $\phi(\bar y'')= u(\bar y'')$ and $\phi>u$ in $(B''_{2r}(\bar y''))^+\setminus\{\bar y''\}$ for some $r>0$.
    \begin{description}
        \item[Case $\bar y''_n>0$]
        Towards a contradiction, we assume that $r\le\frac{\bar y''_n}{8}$ and that 
        \begin{equation}\label{eq:IOF_proof_wrong_lapl}
            \Delta\phi(x'')+a\frac{\de_n\phi(x'')}{x''_n}\le-\Lambda<0
        \end{equation}
        for every  $x''\in B''_{2r}(\bar y'')$.
        By Hausdorff convergence, there exist sequences $y''_j\to \bar y''$ and $c_j\to0$ such that, for every $j$ large, $\{x_1<\eps_j\phi(x'')+c_j\}$ touches $E^j$ from outside at $y_j\coloneqq(\eps_j\phi(y''_j)+c_j,y''_j)$ in a neighborhood $B_r(y_j)$.

        We fix one of those $j$ (in order to simplify the notation, we drop its indication letting $\Omega\coloneqq\Omega^j, y\coloneqq y_j, \eps\coloneqq \eps_j$, and so on) and we let
        \begin{gather}
            p(x'')\coloneqq \frac{\Lambda}{4(n+a-1)}\left(\frac{r^2}{16}-|x''-y''|^2\right),\label{eq:def_pj}\\
            F\coloneqq E\setminus(B_r(y)\cap\{x_1\le\eps(\phi-p)\}),\label{eq:def_Fj}\\
            G(x)=x_1-\eps(\phi(x'')-p(x'')).\label{eq:def_Gj}
        \end{gather}
        By \Cref{lemma:technical_geom} and recalling the choice of $r$, it holds
        \begin{align}\label{eq:iof_proof_integral}
            C\vartheta r^{n+a+\beta-1}\ge \int_{E\setminus F}\dista{\Omega}\left(\dive\frac{\nabla G}{\nabla G}+a\frac{\nabla G\cdot\nabla d_\Omega}{|\nabla G|d_\Omega}\right)\dif\Lc^n.
        \end{align}
        In order to estimate from below of the right-hand side of the above inequality, we first remark that direct computations give
        \begin{equation}
            \Delta p(x'')+a\frac{\de_np(x'')}{x''_n}\ge-\frac{\Lambda}{2}
        \end{equation}
        in $B''_r(y'')$.
        Using the above inequality, \eqref{eq:IOF_proof_wrong_lapl} and the fact that $\left|\frac{\nabla d_\Omega}{d_\Omega}-\frac{e_n}{x''_n}\right|\le C\varkappa$ for some $C$ depending on $\bar y''_n$, we find
        \begin{align}
            &\dive\frac{\nabla G}{\nabla G}+a\frac{\nabla G\cdot\nabla d_\Omega}{|\nabla G|d_\Omega}\\
            &\qquad\ge\frac{\eps}{|\nabla G|}\left(\Delta(p-\phi)+a\frac{\de_n(p-\phi)}{x''_n}\right) - C||\phi||_{C^2}^3\eps^3 - C\varkappa\frac{1}{d_\Omega}\\
            &\qquad\ge\eps\frac{\Lambda}{4}-C\varkappa
        \end{align}
        provided $\eps$ is small enough, depending on $||\phi||_{C^2}$.
        We go back to \eqref{eq:iof_proof_integral} and find
        \begin{equation}
            C(\vartheta+\varkappa)\ge \eps\Lc^n(E\setminus F)
        \end{equation}
        for some $C$ large depending on $\phi$ and $\bar y''$.
        We reach a contradiction by remarking that, by \Cref{lemma:technical_geom} and \eqref{eq:lowerdensityestimate}, $\Lc^n(E\setminus F)\ge c$ (independently of $j$) and that $\varkappa+\vartheta\ll \eps$ by assumption.

        \item[Case $\bar y''_n=0$]As in the previous case, there exist sequences $y''_j\to \bar y''$, and $c_j\to0$ such that $\{x_1<\eps_j\phi(x'')+c_j\}$ touches $E$ from outside at $y_j\coloneqq(\eps_j\phi(y''_j)+c_j,y''_j)$ in a neighborhood $B_r(y_j)$.
        If $y_j\in\de\Omega^j$ for infinitely many $j$, then by \Cref{prop:boundary_max} 
        \begin{equation}
            0\le \nu_\Omega(y_j)\cdot(1,-\eps_j\nabla \phi(y''_j))\le -\eps_j\de_n\phi(y''_j)+C\varkappa_j
        \end{equation}
        Since $\varkappa_j^\lambda\le\eps_j$, as $j\to\infty$ we obtain $\de_n\phi(\bar y'')\ge0$.
        
        If, on the other hand, $y_j\in\Omega^j$ eventually, we argue as follows.
        As in the case $\bar y''_n>0$, we freeze some $j$ large enough and we drop its indication everywhere.
        Then, for some $r>0$ small (independent of $j$) to be determined later, we let 
        \begin{equation}
            p(x'') = \frac{1}{2}\left(\frac{r^2}{16}-|x''-y''|^2\right)
        \end{equation}
        and we let $F$ and $G$ be defined as in \eqref{eq:def_Fj} and \eqref{eq:def_Gj}.
        Then, by \Cref{lemma:technical_geom}, we have
        \begin{equation}\label{eq:iof_last_step}
            C\vartheta r^{n+\beta-1}\ge\int_{E\setminus F}\dista{\Omega}\left(\dive\frac{\nabla G}{|\nabla G|}+a\frac{\nabla G\cdot\nabla d_\Omega}{|\nabla G|d_\Omega}\right)\dif\Lc^n
        \end{equation}
        for some $C>0$ universal.
        We now estimate
        \begin{gather}
            \dive\frac{\nabla G}{|\nabla G|}\ge -\frac{|D^2G|}{|\nabla G|}\ge -\frac{C\eps d_\Omega}{|\nabla G|d_\Omega}.
        \end{gather}
        In order to estimate $a\frac{\nabla G\cdot \nabla d_\Omega}{|\nabla G|d_\Omega}$, we recall $\nabla G=e_1+\eps(0,\nabla''p-\nabla''\phi)$ and we compute
        \begin{gather}
            e_1\cdot \nabla d_\Omega\ge -C\varkappa,\\
            (0,\nabla''\phi)\cdot\nabla d_\Omega\le \de_n\phi+C\varkappa
        \end{gather}
        for some $C$ depending on $\phi$.        
        For the next computations, we recall that $|\nabla d_\Omega-e_n|\le C\varkappa$, that $x\cdot\nabla d_\Omega\ge d_\Omega-C\varkappa$ and that $y_n\ge -\varkappa$: therefore
        \begin{align}
            (0,\nabla''p)\cdot\nabla d_\Omega
            &\ge -x\cdot\nabla d_\Omega - |x_1||e_1\cdot\nabla d_\Omega| + y_n - |y||e_n-\nabla d_\Omega|\\
            &\ge -d_\Omega - C\varkappa
        \end{align}
        Gathering the above estimates, we find
        \begin{equation}
            \dive\frac{\nabla G}{|\nabla G|}+a\frac{\nabla G\cdot\nabla d_\Omega}{|\nabla G|d_\Omega}\ge-\frac{a}{|\nabla G|d_\Omega}\left(\eps\de_n\phi + C\varkappa + C\eps d_\Omega \right)
        \end{equation}
        Towards a contradiction, assume $\de_n\phi(\bar y'') = -4\delta<0$.
        Then, by choosing $r$ small (depending on $\delta$ and $||D^2\phi||_\infty$), using $\varkappa^\lambda\le\eps$ and taking $j$ large enough, we may assume $\de_n\phi\le-2\delta$ in $B_r(y)$ thus
        \begin{equation}
            \dive\frac{\nabla G}{|\nabla G|}+a\frac{\nabla G\cdot\nabla d_\Omega}{|\nabla G|d_\Omega}\ge\eps\delta\frac{a}{|\nabla G|d_\Omega}\ge \frac{a}{2}\eps\delta.
        \end{equation}
        Therefore, going back to \eqref{eq:iof_last_step} and using \Cref{prop:density_a}, we obtain
        \begin{equation}
            C\vartheta r^{n-1+\beta}\ge \delta \eps r^{n+a}
        \end{equation}
        which fails as $j\to\infty$, since $\vartheta^\lambda\le\eps$ and $\lambda<1$.
        This concludes the proof in the case $\bar y''_n=0$.        
    \end{description}
    
    Having established that $u$ is a viscosity solution to \eqref{eq:linearized}, we may apply \Cref{lemma:linearized} below, that exploits the results from \cite{Sire_Terracini_Vita_2021} to obtain Schauder estimates for $u$.
    In particular, it holds $u\in C^2((B_{1/4}'')^+)$ and 
	\begin{equation}
		||u||_{C^2((B''_{1/4})^+)}\le C
	\end{equation}
	for some $C$ universal. 
    Now, we may find $\eta$ so that
    \begin{equation}
        |u(x'')-u(0'')-\nabla u(0'')\cdot x''|\le \frac{1}{4}\eta
    \end{equation}
    for every $x''\in (B''_{2\eta})^+$.
	Therefore, by the Hausdorff convergence established previously, we obtain
	\begin{equation}
		\de E^j\cap B_{\eta}\cap\Omega^j\subset\left\{x\colon |x_1-\eps_j u(0'') - \eps_j \nabla u(0'')\cdot x''|\le \frac{1}{2}\eps_j \eta\right\}
	\end{equation}
    which is the desired result.
    
\end{proof}

The following result was used in the proof of \Cref{thm:improvement_flatness}:
\begin{lemma}[Regularity for the linearized problem]\label{lemma:linearized}
	There exists a universal constant $C$ such that, if $u$ is a viscosity solution to \eqref{eq:linearized} with $u\in C^{0,\sigma}(\overline{B_{1/2}^+})$ and $||u||_{L^\infty}\le 1$, then $u\in C^2((B_{1/4}'')^+)$ and
	\begin{equation}\label{eq:apriori_linearized}
		||u||_{C^2((B''_{1/4})^+)}\le C.
	\end{equation}
\end{lemma}
\begin{proof}
    The idea of the proof is to build a sequence of energetic solutions to \eqref{eq:linearized} that converges uniformly to $u$. We then conclude by propagating the a-priori estimates proved in \cite{Sire_Terracini_Vita_2021} along the sequence. 

    For the sake of discussion and in order to keep the notation as light as possible, we replace $\RR^{n-1}$ by $\RR^n$ (thus writing $x,B_r$ in place of $x'',B_r''$) and we replace $B_{1/2}$ and $B_{1/4}$ by $B_1$ and $B_{1/2}$, respectively.
    We also extend $u$ evenly to the whole $B_1$, letting $u(x',x_n)=u(x',-x_n)$ for all $(x',x_n)\in B_1\cap\{x_n<0\}$.

    Given $r>0$, we consider the inf-convolution $u_r\colon B_1\to\RR$ defined as
    \begin{equation}
    	u_r(x) := \inf_{y\in B_1}\left\{u(y) + \frac{1}{2r}|x-y|^2  \right\}.
    \end{equation}
    The fact that $u\in C^{0,\sigma}(B_1)$ and $||u||_{L^\infty}\le1$ yields 
    \begin{equation}\label{eq:conv_infconv}
    	u(x) - Cr^{\sigma/2} \le u_r(x) \le u(x),
    \end{equation}
    for all $x\in B_1$.
    Furthermore, $u_r$ is Lipschitz-continuous. We refer the readers to \cite[Section 5.1]{caffarelliCabre95} for the proofs of the above facts.
    Since $u_r$ is Lipschitz-continuous, there exists an energetic solution $v_r$ (in the sense of \cite{Sire_Terracini_Vita_2021}) to
    \begin{equation}
	\begin{cases}
		-\dive(|x_n|^a \nabla v_r) =0  & \mbox{in } B_{3/4},\\
		v_r = u_r  &\mbox{on }\de B_{3/4}.
	\end{cases}
    \end{equation}    
    Since $u_r$ is symmetric with respect to $\{x_n=0\}$, we may assume $v_r$ is symmetric as well.
    Moreover, by \cite{Sire_Terracini_Vita_2021}, $v_r\in C^{2,\theta}_{loc}(B_{3/4})$ and
    \begin{equation}\label{eq:linearized_apriori_est}
        ||v_r||_{C^{2,\theta}(B_{1/2})}\le C||v_r||_{L^\infty}\le C,
    \end{equation}
    where $\theta\in(0,1)$ and $C>0$ are universal constants.
    In particular, $\de_nv_r=0$ on $\{x_n=0\}$.

    Let $\delta>0$ small to be chosen later and consider $v_r'(x)=v_r(x)+\delta x_n$.
    We claim that $v'_r\le u+\delta$ in $\overline{B^+_{3/4}}$.
    If not, then
    \begin{equation}
        \max_{\overline{B^+_{3/4}}}(v'_r-u)=v'_r(z)-u(z)\eqqcolon m>\delta
    \end{equation}
    for some $z\in \overline{B_{3/4}^+}$, so that $v'_r-m$ touches $u$ from below at $z$.
    By $v_r=u_r < u$ on $\de B_{3/4}$, it must be $z\notin\de B_{3/4}$, for $\delta>0$ sufficiently small.
    Moreover, since $u$ is a viscosity solution of \eqref{eq:linearized}, we exclude both the option $z_n=0$ (because $\de_n v_r'(z)=\delta>0$ on $\{x_n=0\}$) and the option $z_n>0$, since in the latter case we would have
    \begin{equation}
        0\ge \Delta v_r'(z)+a\frac{\de_n v'_r(z)}{z_n}=\Delta v_r(z)+a\frac{\de_n v_r(z)}{z_n}+\frac{a\delta}{z_n}=\frac{a\delta}{z_n}>0.
    \end{equation}
    Since $\delta>0$ is arbitrary, we conclude $v_r\le u$ in $\overline{B^+_{3/4}}$. With analogous computations, we also find $v_r\ge u-Cr^{\sigma/2}$ in $\overline{B^+_{3/4}}$.

    By the above considerations, $v_r\to u$ uniformly as $r\to0$ in $B^+_{3/4}$.
    \eqref{eq:linearized_apriori_est} and the Arzelà-Ascoli theorem yield the desired conclusion for $u$.
    
\end{proof}

\begin{proposition}[Improvement of flatness at points in $\Omega$]\label{prop:interior_IOF}
    There exist universal constants $\eps_{4},\lambda_{4},\eta_4$ (small) and $C_4$ (large) with the following property.
    Let $E$ be a $(\vartheta,\beta)$-minimizer of $\Per_{\dista{\Omega}}$ in $B_R(x_0)$, where $x_0\in\Omega$ and $R$ are such that $d_\Omega(x_0)\ge C_4R$.
    Assume that, for some $\nu\in\sphere$,
    \begin{equation}\label{eq:iof_interior_hp}
        C_4\left((\vartheta R^\beta)^{\lambda_4}+R\frac{||\nabla d_\Omega\cdot\nu||_{L^\infty(B_R(x_0))}}{d_\Omega(x_0)}\right)\le \eps\le\eps_{4},
    \end{equation}
    where
    \begin{equation}
        \eps\coloneqq\frac{1}{R}\osc_{\nu}(\de E;B_R(x_0)).
    \end{equation}
    Then there exists $\tilde\nu\in\sphere$ such that $|\tilde \nu-\nu|\le C_4\eps$ and 
    \begin{equation}
        \osc_{\tilde\nu}(\de E;B_{\eta_4 R}(x_0))\le\frac{\eta_4}{2}\eps R. 
    \end{equation}
\end{proposition}
\begin{proof}
    Since the proof is very similar to the one of \Cref{thm:improvement_flatness}, we only sketch it.
    \begin{itemize}
        \item Without loss of generality, we assume $R=1$, $x_0=0$ and that $d_\Omega(0)\ge C$ for $C\ge C_2'$ given in \Cref{cor:osc_decay_interior}.
        \item We consider a sequence of sets $E^j$ that are $(\vartheta_j,\beta)$-minimizers of $\Per_{\dista{\Omega^j}}$ in $B_1$ and that satisfy \eqref{eq:iof_interior_hp} with $\eps_4$ replaced by $\eps_j$ for some $\eps_j\searrow0$. Without loss of generality, we assume $\nu_j=e_1$ for every $j$.
        \item Using \Cref{cor:osc_decay_interior}, we prove that the rescalings $\tilde E^j$ defined as in the proof of \Cref{thm:improvement_flatness} converge (up to a subsequence) in the Hausdorff distance to the graph of a $C^{0,\sigma}$ function $u$.
        \item Reproducing the argument used in the case $\bar y''_n>0$ of the proof of \Cref{thm:improvement_flatness}, we prove that $u$ is a viscosity solution to
        \begin{equation}
            \Delta u +a\frac{\de_n u}{x_n''+d_0}=0\quad\mbox{in }B''_{1/2}
        \end{equation}
        where $d_0\coloneqq\lim_{j\to\infty}d_{\Omega_j}(0)\in[C,+\infty]$ which exists up to extracting a further subsequence.
        \item By classical Schauder estimates (see, for instance, \cite{Gilbarg_Trudinger}), it holds $||u||_{C^2}\le C$ for some $C$ universal.
        \item As in the proof of \Cref{thm:improvement_flatness}, we conclude the proof by taking a second-order Taylor expansion of $u$ at $0''$ and exploiting the Hausdorff convergence proved above.
     \end{itemize}
\end{proof}

The last step towards the proof of \Cref{thm:main_theorem} is the following
\begin{corollary}[Iteration of the improvement of flatness]\label{cor:iof_iter}
    There exist universal constants $\eps_5,\lambda_5, \gamma_5$ (small) and $C_5$ (large) with the following property.
    Let $\Omega$ be $\varkappa$-flat and let $E$ be a $(\vartheta,\beta)$-minimizer of $\Per_\dista{\Omega}$ in $B_1$. If
    \begin{equation}
        (\vartheta+\varkappa)^{\lambda_5}\le \osc_{e_1}(\de E;B_1)\eqqcolon \eps\le\eps_5,
    \end{equation}
    then for every $x\in B_{1/4}\cap\overline{\Omega}$ there exists $\nu_x\in\sphere$ with $|\nu_x-e_1|\le C_5\eps$ and, for every $0<r\le1/4$:
    \begin{equation}\label{eq:iof_final_thesis}
        \osc_{\nu_x}(\de E;B_r(x))\le C\eps r^{1+\gamma_5}.
    \end{equation}
    Moreover, if $x\in\de\Omega$ then $\nu_x\perp\nu_\Omega(x)$.
\end{corollary}
\begin{proof}
    \textbf{Step 1: decay at the boundary. }
    For $x\in B_{1/2}\cap\de{\Omega}$ and $0<r\le\frac{1}{2}$, let
    \begin{equation}
    F_x(r)\coloneqq\inf_{\nu\in\sphere}\osc_\nu(\de E;B_r(x)) + Ar\big(\vartheta r^\beta+\varkappa r^\alpha\big)^\lambda,
    \end{equation}
    where $A$ is a large universal constant whose value will be specified alter.
    Notice that the infimum in the definition above is attained.

    We claim that, if for some $x\in\de\Omega$ and $r>0$ it holds $F_x(r)\le \eps_1r$, then
    \begin{equation}\label{eq:iof_iteration0}
        F_x(\eta r)\le\frac{\eta}{2}F_x(r).
    \end{equation}
    To prove the above claim, we distinguish two cases:
    \begin{itemize}
        \item If
        \begin{equation}\label{eq:deficit_grande}
            (\vartheta r^\beta+\varkappa r^\alpha)^\lambda\ge\frac{1}{r}\inf_{\nu}\osc_\nu(\de E;B_r(x))
        \end{equation}
        then we trivially have
        \begin{align}
            F_x(\eta r)\le Ar\big(\vartheta r^\beta+\varkappa r^\alpha\big)^\lambda\left(\frac{1}{A}+\eta^{1+\lambda(\beta\wedge\alpha)}\right)\le\frac{\eta}{2}F_x(r)
        \end{align}
        provided $\eta$ is small enough so that $\eta^{1+\lambda(\beta\wedge\alpha)}\le\frac{\eta}{4}$ and $A$ is larger than some constant depending on $\eta$.
        
        \item If \eqref{eq:deficit_grande} does not hold and $x\in\de\Omega$, provided $\lambda\le\lambda_1$, the assumptions of \Cref{thm:improvement_flatness} are in place, thus
        \begin{equation}
            F_x(\eta r)\le\frac{\eta}{2}\inf_{\nu\in\sphere}\osc_\nu(\de E;B_r(x))+A\eta^{1+\lambda(\beta\wedge\alpha)} r\big(\vartheta r^\beta+\varkappa r^\alpha\big)^\lambda\le\frac{\eta}{2}F_x(r)
        \end{equation}
        as claimed, where we have used again the fact that $\eta^{1+\lambda(\beta\wedge\alpha)}\le\frac{\eta}{4}$.
    \end{itemize}
    For all $x\in\de\Omega\cap B_{1/2}$, by choosing $\eps_5$ small enough it holds $F_x(1/2)\le\eps_1/2$, thus by induction
    $F_x(\eta^k/2)\le\left(\frac{\eta}{2}\right)^2 F_x(1/2)$ for all $k\in\NN$ and every $x\in\de\Omega\cap B_{1/2}$.
    Moreover, as $k\to\infty$, the unit vectors realizing the infimum in the definition of $F_x(\eta^k r)$ converge to a unit vector $\nu_x$ which by \Cref{thm:improvement_flatness} is orthogonal to $\nu_\Omega(x)$. Interpolating between scales $\eta^kr$ for $k\in\NN$, we finally find
    \begin{equation}\label{eq:iof_decay_boundary}
        \osc_{\nu_{x}}(\de E;B_r(x))\le C\eps r^{1+\gamma}
    \end{equation}
    for all $x\in\de\Omega\cap B_{1/2}$ and every $r\in(0,1/2)$, where $C$ and $\gamma$ are universal constants.

    \textbf{Step 2: Decay away from the boundary. }
    For $x\in B_{1/4}\cap\Omega$, let
    \begin{gather}
        G_x(r,\nu)\coloneqq \osc_\nu(\de E;B_r(x))+Ar\left(r\frac{||\nabla d_\Omega\cdot \nu||_{L^\infty(B_r(x))}}{d_\Omega(x)}+(\vartheta r^\beta)^\lambda\right)\\
        G_x(r)\coloneqq\inf_{\nu\in\sphere} G_x(r,\nu)
    \end{gather}
    where $A$ is a large universal constant (possibly different than the one chosen in Step 1) that will be specified later.
    Notice that the infimum in the definition of $G_x(r)$ is attained.
    We claim that, if for some $x\in B_{1/4}\cap\Omega$ and $r\le \frac{d_\Omega(x)}{\bar C}$ (with $\bar C\ge C_4$ large, to be specified later) it holds $G_x(r)\le\eps_4r$, then
    \begin{equation}
        G_x(\eta r)\le\frac{3}{4}\eta\,G_x(\eta r).
    \end{equation}
    As above, we distinguish two cases.
    \begin{itemize}
        \item Let $\nu\in\sphere$ realize the infimum in the definition of $G_x(r)$.
        If
        \begin{equation}\label{eq:interior_flat_vince}
            C_4 \left(r\frac{||\nabla d_\Omega\cdot \nu||_{L^\infty(B_r(x))}}{d_\Omega(x)}+(\vartheta r^\beta)^\lambda\right)\le\frac{1}{r}\osc_\nu(\de E;B_r(x)),
        \end{equation}
        then by \Cref{prop:interior_IOF} there exists $\tilde\nu\in\sphere$ such that
        \begin{equation}
            \osc_{\tilde\nu}(\de E;B_{\eta r}(x))\le\frac{\eta}{2}\osc_\nu(\de E;B_r(x))\quad\mbox{and}\quad |\tilde \nu-\nu|\le C_4\frac{\osc_\nu(\de E;B_r(x))}{r},
        \end{equation}
        hence
        \begin{equation}
            ||\nabla d_\Omega\cdot \tilde\nu||_{L^\infty(B_{\eta r}(x))}\le ||\nabla d_\Omega\cdot \nu||_{L^\infty(B_{r}(x))}+C_4\frac{\osc_\nu(\de E;B_r(x))}{r}.
        \end{equation}
        Therefore
        \begin{align}
            G_x(\eta r)&\le G_x(\eta r,\tilde\nu)\\
            &\le \left(\frac{\eta}{2} + C_4 A\frac{\eta^2 r}{d_\Omega(x)}\right) \osc_\nu(\de E;B_r(x)) \\
            &\qquad+ A \eta r\left(\eta r\frac{||\nabla d_\Omega\cdot\nu||_{L^\infty(B_r(x))}}{d_\Omega(x)} + (\vartheta(\eta r)^\beta)^\lambda\right)\\
            &\le \frac{3}{4}\eta \,G_x(r,\nu)
        \end{align}
        as claimed, where we have used the fact that $d_\Omega(x)\ge r$ and we have assumed that $\eta$ is small enough depending on $C_4$ and $A$.
        \item On the other hand, if $\nu\in\sphere$ realizes the infimum in the definition of $G_x(r)$ and \eqref{eq:interior_flat_vince} fails, then we trivially have
        \begin{align}
            G_x(\eta r)&\le G_x(\eta r,\nu)\\
            &\le \osc_\nu(\de E;B_r(x)) + A\eta r\left(\eta r\frac{||\nabla d_\Omega\cdot\nu||_{L^\infty(B_r(x))}}{d_\Omega(x)}+\vartheta^\lambda(\eta r)^{\beta\lambda}\right)\\
            &\le A r \left( r\frac{||\nabla d_\Omega\cdot\nu||_{L^\infty(B_r(x))}}{d_\Omega(x)}+(\vartheta r^\beta)^{\lambda}\right) \left(\frac{C_4}{A}+\eta^{1+(1\wedge(\beta\lambda))} \right)\\
            &\le \frac{3}{4}\eta\, G_x(r)
            \end{align}
            provided $\eta$ is small enough and $A$ is large enough, depending on $C_4$ and $\eta$.
    \end{itemize}

    \textbf{Step 3: Conclusion. }
    Let $x\in\Omega\cap B_{1/4}$ and let $y\in\de\Omega\cap B_{1/2}$ be such that $d_\Omega(x)=|x-y|$.
    Let also $\rho\coloneqq d_\Omega(x)/\bar C$.
    By Step 1, if $r\ge\rho$, then 
    \begin{equation}\label{eq:decay_per_r_grandi}
        \osc_{\nu_y}(\de E;B_r(x))\le \osc_{\nu_y}(\de E;B_{(1+\bar C)r}(y))\le C\eps r^{1+\gamma}.
    \end{equation}
    Now, if $z\in B_\rho(x)$ and $\tilde z\in\de\Omega$ satisfies $|\tilde z-z|=d_\Omega(z)$, then $|\tilde z-y|\le 2d_\Omega(z)\le 2(1+\bar C)\rho$ and, by \Cref{lemma:tech_distance}
    \begin{equation}
        |\nabla d_\Omega(z)\cdot\nu_y|=|\nu_\Omega(\tilde z)\cdot\nu_y|\le|\nu_\Omega(\tilde z)-\nu_\Omega(y)|\le C\varkappa \rho^\alpha. 
    \end{equation}
    Therefore
    \begin{align}
        G_x(\rho)&\le C\eps \rho^{1+\gamma} + A\rho\left(\rho\frac{C\varkappa\rho^\alpha}{d_\Omega(x)}+(\vartheta\rho^\beta)^\lambda\right)\\
        &\le 2C\eps \rho^{1+\gamma}\\
        &\le\eps_5\rho
    \end{align}
    where in the second inequality $\gamma$ was chosen small enough and $C$ large enough, and in the third one $\bar C$ was chosen large so that $\rho\le 1/\bar C$ is smaller than some universal constant.
    Assuming $\eps_5\le\eps_4$, using the above computation and Step 2 and arguing by induction, we find
    \begin{equation}
        G_x(\eta^k\rho)\le \left(\frac{3}{4}\eta\right)^kG_x(\rho).
    \end{equation}
    By interpolating between scales and using $G_x(\rho)\le C\eps \rho^{1+\gamma}$, we finally obtain
    \begin{equation}
        G_x(r)\le C\eps r^{1+\gamma'}
    \end{equation}
    for some $\gamma'$ small enough and $C$ large enough.
    The above inequality and \eqref{eq:decay_per_r_grandi} give \eqref{eq:iof_final_thesis}.
    
\end{proof}

\appendix

\section{Proof of the weak Harnack inequality}\label{sec:appendix_DeSilvaSavin}

For the reader’s convenience, we restate the results from \cite[Section 3]{DeSilva_Savin_2021} using slightly different notation. This adjustment aligns the notation with ours and clarifies which results or hypotheses are used and when.

We clarify that the main result involves deriving a Harnack inequality from an ABP-type estimate, interpreting this as an adaptation of the methods proposed in \cite{Savin_2007}. 
The key challenge stems from the weak viscosity framework that is used, which does not inherently yield ABP estimates or pointwise information. 
To tackle this issue, a discretization technique in combination with a Calderón-Zygmund-inspired argument are used.

In this appendix we strongly use the notation introduced in \Cref{subsec:interior_harnack} and to streamline the analysis and avoid introducing complex conditions on the radii, we will consistently assume
\begin{equation}
	\Lambda > 2, \qquad \mbox{and} \qquad r< \frac{1}{8}.
\end{equation}

Lastly, through this Appendix, we refer to constants that depend on also on $\Lambda$ as universal.

The proof of \Cref{prop:weakharnack} is based on the following
\begin{proposition}[Corollary 3.2 in \cite{DeSilva_Savin_2021}]\label{prop:firstmeas}
	There exists a universal constant $\mu\in(0,1)$ with the following property.
    Let $\tau,r>0$ and $E\in\Pc^{\{4\tau\}}_\Lambda(r)$ in $B_{1/2}$, then for all $\ell\ge1$ with $r < 2^{-\ell}$ and $\rho>0$ smaller than some universal constant such that
    \begin{equation}
        A^{\tau}(E) \cap Q''_{\rho/2} \ne \emptyset,
	\end{equation}
    it holds
	\begin{equation}
		\Lc^{n-1}(A^{2\tau}_{\ell}(E) \cap Q''_\rho)\ge (1-\mu) \Lc^{n-1}(Q''_\rho).
	\end{equation}
\end{proposition}

The proof \Cref{prop:firstmeas} relies on a purely geometric arguments on paraboloids, hence we directly refer the reader to \cite{DeSilva_Savin_2021}.

With this preliminary result, we are now in a position to prove \Cref{prop:weakharnack}.

\begin{proof}[Proof of \cref{prop:weakharnack}]
    We set $\rho=2^{-M}$ for some $M$ larger than some universal constant to be determined later.
	The proof works by induction on $h$.
	
    \textbf{Base case $h=1$.}
		By \Cref{prop:firstmeas}, it suffices to show that $A^{C\tau}(E) \cap B''_{\rho/2}\ne\emptyset$ for some large and universal $C>0$. 
		Once this is established, the conclusion follows by taking $\bar{C} = 5C$.
		The proof relies on the use of a proper barrier.

    \textbf{Step 1 of Base Case: Barriers' definition.}
        We start by defining an auxiliary function $h$.
        To this end, for $s>0$, we define
        \begin{equation}
            \psi(x'') = -\frac{\rho^{2\Lambda}}{|x''|^{2\Lambda}} + t_\Lambda,\qquad \mbox{and}\qquad p^{\sigma_s}(x'') = \frac{\sigma_s}{2}|x''|^2 + t_s,
        \end{equation}
        where $t_\Lambda=|6\sqrt{n-1}|^{-2\Lambda}$, $\sigma_s= \frac{2\Lambda}{s^{2\Lambda+2}}>0$ and $t_s$ are chosen such that $\psi(x'')>0$ outside $B''_{6\sqrt{n-1}\rho}$,  $p^{\sigma_s}(x'')\ge \psi(x'')$ in $\RR^{n-1}$ and $p^{\sigma_s}(x'') = \psi(x'')$ on $\de B_s''$.
        
        We observe that if $z''\in\de B_{s}''$, taking the second order expansion of $\psi$ in $z''$, we find 
        
        \begin{equation}\label{eq:secondderbarrAppendix}
			p^{\sigma_s}(x'') - {\psi}(x'') \ge 2\Lambda \frac{\rho^{2\Lambda}}{s^{2\Lambda+2}}(\Lambda+1) \left( (x''-z'')\cdot \frac{z''}{|z''|} \right)^2,
		\end{equation}
        for all $x''\in B''_{cs}(z'')$, for some small $c=c(\Lambda)>0$.

        We now fix $s=\frac{\rho}{2}$ and we define $h(x'')$ as 
        \begin{equation}
            {h}(x'') = 
            \begin{cases}
                \psi( x'') & \mbox{if } |x''|\ge \frac{\rho}{2},\\
                p^{\sigma_{\rho/2}}(x'') &\mbox{if } |x''|< \frac{\rho}{2}.
            \end{cases}
        \end{equation}
        Notice that, by construction, $h\in C^{1}(\RR^{n-1})$.

        Finally, we define the barrier $G$ we are going to use.
		By assumption $A^{\tau}(E) \cap Q''_{3\rho}\ne\emptyset$, hence there exists a paraboloid $p_y^\tau(x'')$ (which we fix) such that $\{x_1 < p_y^{\tau}(x'')\}$ touches $E$ from outside at a point whose projection onto $\RR^{n-1}$ lies in $Q''_{3\rho}$.
		For a sufficiently large constant $C>1$, we define the set
		\begin{equation}
			G := \left\{x_1 < p^{\tau}_y(x'') + C\tau h(x'') + t \right\}.
		\end{equation}
	where $t\in\RR$ is such that $G$ touches $E$ from outside in $B_1$. 
        This $t$ exists since $h(0)<0$ and $h>0$ outside $B''_{6\sqrt{n-1}\rho}$ and $\{x_1 < p_y^{\tau}(x'')\}$ touches $E$ from outside.
        We call $z\in B_1$ the contact point between $G$ and $E$.

    \textbf{Step 2 of Base case.} We now show that $z\in Q_{\rho/2}$.
		
	By contradiction, if $z''\notin B''_{\rho/2}$, calling $s=|z''|$, then from \eqref{eq:secondderbarrAppendix} applied to $p^{C\tau\sigma_s}(x'')$, it follows that
	\begin{equation}
			\left\{
			x_1 < p^{\tau}_y(x'') + p^{C\tau\sigma_{s}}(x'') - C\tau\Lambda\frac{\rho^{2\Lambda}}{s^{2\Lambda+2}}(\Lambda+1)\left((x''-z'')\cdot\frac{z''}{|z''|} \right)^2
			\right\},
	\end{equation}		
	touches $E$ from outside in $B_{c\rho}(z)$.
        Taking $C = |6\sqrt{n-1}|^{2\Lambda+2}\Lambda^{-1}$, we get
        \begin{equation}
			\bar{C} r \le \rho, \qquad C\tau\Lambda (\Lambda+1) \frac{\rho^{2\Lambda}}{s^{2\Lambda+2}} > \tau \left(1+C\frac{\rho^{2\Lambda}}{s^{2\Lambda+2}}\right) \Lambda, \qquad T > 5\tau\left(\frac{|12\sqrt{n-1}|^{2\Lambda+2}}{\Lambda}+1\right),
	\end{equation}
	that contradict $E\in\Pc_\Lambda^{[\tau,T]}(r)$.
		
    \textbf{Step 3 of Base case.} Since $G$ touches from outside $E$ in $B_1$ at $z\in Q_{\rho/2}^1$, then by elementary polynomial manipulations, we get that
	\begin{equation}
			\left\{ x_1< p^{\tau}_y(x'') + C\tau p^{\sigma_{\rho/2}}(x'') = p^{(C\sigma_{\rho/2} +1)\tau}_{y_0}(x'') \right\}
	\end{equation}
	touches $E$ from outside in $Q^1_{\rho/2}$, where $y_0$ is a point such that $y_0'' = \frac{1}{C\sigma_p+1}y''
        \in B_1''$.

    \textbf{Inductive Step, $h \Rightarrow h+1$}.
        We now show that if the statement holds for $h$ and $C^{h+1}\tau \le T$, then it holds also for $h+1$.
        To this end, we define a Calderón-Zygmund decomposition of $Q_\rho''$ inductively as follows. Fix $\ell\in\NN$ such that $\bar C^h\tau\le T$ as in the assumptions of \Cref{prop:weakharnack};
        \begin{itemize}
                \item we set $\Fc:=\{Q''_\rho\}$ and $\Bc=\emptyset$;
                \item For all integers $j$ such that $M\le j\le \ell$ and for all $Q''\in\Fc\cap\Qc_{j}$:
                \begin{itemize}
                    \item if $Q''\cap A^{\bar{C}h\tau}_\ell(E)=\emptyset$, we add $Q''$ to $\Bc$,
                    \item otherwise, we add the dyadic decomposition of $Q''$ to $\Fc$. 
                \end{itemize}
        \end{itemize}
        By construction, $B$ has the following two properties. 
        \begin{itemize}
                \item $Q_\rho'' \setminus A^{\bar{C}^h\tau}_\ell(E) = \bigcup\{Q'': Q''\in \Bc\}$. 
                Indeed it is straightforward from the definition that $\bigcup\{Q_j'': Q_j''\in \Bc\} \subset Q_\rho'' \setminus A^{\bar{C}^h\tau}_\ell(E)$.
                The other inclusion follows by observing that any $Q''\subset Q_\rho\setminus A^{\bar{C}^h\tau}_\ell(E)$ and $Q''\in\Qc_\ell$ belongs to $\Bc$.
                \item If we dilate any $Q''\in \Bc$ around its center by a factor $3$, it intersects $A^{\bar{C}^h\tau}(E)$.
            Since $\bar{C}^{h+1}\tau\le T$, by the Base Case ($h=1$) applied to an appropriate translation of $Q''$, we find 
            \begin{equation}
                \Lc^{n-1}\left(A^{\bar{C}^{h+1}\tau}_\ell (E)\cap Q''\right) \ge (1-\mu)\Lc^{n-1}(Q'').
            \end{equation}
        \end{itemize}
        From these considerations, it follows
        \begin{equation}
			\begin{split}
				\Lc^{n-1}(A^{\bar{C}^{h+1}\tau}_\ell(E)\cap Q''_\rho)
				&= \Lc^{n-1}(A^{\bar{C}^{h}\tau}_\ell(E)\cap  Q''_\rho) + \sum_{Q''\in \Bc} \Lc^{n-1}(A^{\bar{C}^{h+1}\tau}_\ell(E) \cap Q'')\\
				&\ge \Lc^{n-1}(A^{\bar{C}^{h}\tau}_\ell(E)\cap  Q''_\rho) + (1-\mu) \sum_{Q''\in \Bc} \Lc^{n-1}(Q'')\\
				&= \mu \Lc^{n-1}(A^{\bar{C}^{h}\tau}_\ell(E)\cap  Q''_\rho) + (1-\mu) \Lc^{n-1} (Q''_\rho)\\
				& \ge \left(\mu(1-\mu^h) + (1-\mu)\right)\Lc^{n-1}(Q_\rho'')\\
				&= (1-\mu^{h+1})\Lc^{n-1}( Q''_\rho) 
			\end{split}
	\end{equation}
        which concludes the proof.
		
		\begin{figure}
        \centering
		\begin{tikzpicture}[use Hobby shortcut, scale=0.7]
			\begin{scope}
				\clip (-3.5,-3.5) rectangle (3.5,3.5);
				\draw[fill=yellow!10] (-3,-3) rectangle (1,1);	
				\draw[fill=red!10] (-1,-1) rectangle (1,1);										
				\draw[fill=blue!10,rotate=90,shift={(0,3)},scale=0.3] (-2.8,-2) .. (-2,-1.5) .. (-2,-2) ..  (-2.5,-2.5) .. (-2.8,-2);
				\draw[fill=blue!10, rotate=45, shift={(-1,2.5)}, scale=0.5] (-1.8,-2) .. (-1.6,-1.5) .. (-1.4,-1.5) .. (-1,-2) .. (-1.3,-2.2) .. (-1.7,-2.5)  .. (-2.4,-3) .. (-1.8,-2);
				\draw[fill=blue!10, shift={(-1.3,0.2)}] 
					(0,0) circle (0.2);
				\draw[fill=blue!10] (-4,-1.9) .. (-3,-2.3) .. (-2.5,-1.3) .. (-1.4,-2.5) .. (-3,-5);
				\draw (-3,-3) rectangle (1,1);	
				
				\node at (0,0) {$Q_j$};
				\node at (0,-1.5) {$Q_{j-1}$};
				\node[scale=0.8] at (-2.2,-2.4) {$A^{C^{h}\tau}(E)$};
				
				\draw (-3,-3) rectangle (3,3);
				\draw (-3,-1) -- (3,-1);
				\draw (-3,1) -- (3,1);
				\draw (-1,-3) -- (-1,3);
				\draw (1,-3) -- (1,3);
				\draw[dashed] (-4,-1) -- (4,-1);
				\draw[dashed] (-4,1) -- (4,1);
				\draw[dashed] (-4,-3) -- (4,-3);
				\draw[dashed] (-4,3) -- (4,3);

				\draw[dashed] (-1,-4) -- (-1,4);
				\draw[dashed] (1,-4) -- (1,4);
				\draw[dashed] (3,-4) -- (3,4);
				\draw[dashed] (-3,-4) -- (-3,4);

			\end{scope}
		\end{tikzpicture}
        \caption{One step in the Calderón-Zygmund decomposition of $Q_\rho''$}
		\end{figure}
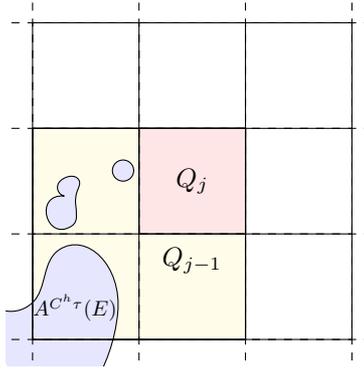

\end{proof}

\bibliographystyle{siam} 
\bibliography{citations2}

\end{document}